\documentclass[final]{siamltex}
\usepackage{graphicx}
\usepackage{psfrag}
\usepackage[utf8]{luainputenc}
\usepackage{amsmath}
\usepackage{amssymb}
\usepackage{color}
\usepackage{afterpage}
\usepackage{subfig}
\usepackage{esint}
\usepackage{lineno}
\newtheorem{remark}[theorem]{Remark}
\DeclareMathOperator{\sech}{sech}

\usepackage{hyperref}
\hypersetup{
    colorlinks,
    bookmarksopen,
    bookmarksnumbered,
    citecolor=cyan,
    filecolor=cyan,
    linkcolor=blue,
    urlcolor=green
}

\usepackage{xcolor}
\colorlet{linkequation}{red}
\newcommand*{\SavedEqref}{}
\let\SavedEqref\eqref
\renewcommand*{\eqref}[1]{%
  \begingroup
    \hypersetup{
      linkcolor=linkequation,
      linkbordercolor=linkequation,
    }%
    \SavedEqref{#1}%
  \endgroup
}

\title{Dirichlet-Neumann Waveform Relaxation Method for the 1D and 2D Heat and Wave Equations in Multiple subdomains}

\author{Martin J. Gander\thanks{Department of Mathematics,
    University of Geneva, Switzerland ({\tt martin.gander@unige.ch}).}
    \and Felix Kwok\thanks{Department of Mathematics, Hong Kong Baptist University, Hong Kong ({\tt felix\_kwok@hkbu.edu.hk}).}
   \and Bankim C. Mandal \thanks{Department of Mathematical Sciences, Michigan Technological University, USA ({\tt bmandal@mtu.edu}).}
}

\begin{document}

\maketitle

\begin{abstract}
We present a Waveform Relaxation (WR) version of the Dirichlet-Neumann
algorithm, formulated specially for multiple subdomains splitting for general parabolic and hyperbolic problems. This method is based on a non-overlapping spatial domain decomposition, and the iteration involves subdomain solves in space-time with corresponding interface condition, and finally organize an exchange of information between neighboring subdomains. Using a Fourier-Laplace transform argument, for a particular relaxation parameter, we present convergence analysis of the algorithm for the heat and wave equations. We prove superlinear convergence for finite time 
window in case of the heat equation, and finite step convergence for the wave equation. The convergence behavior however depends on the size of the subdomains and the time window length on which the algorithm is employed. We illustrate the performance of the algorithm with numerical results, and show a comparison with classical and optimized Schwarz WR methods.
\end{abstract}

\begin{keywords}
Dirichlet-Neumann, Waveform Relaxation, Heat equation, Wave equation, Domain decomposition.
\end{keywords}
 
\section{Introduction}
A recent version of Waveform Relaxation (WR) methods, namely Dirichlet-Neumann Waveform Relaxation (DNWR) has been introduced in \cite{Mandal,GKM1,GKM3} to solve space-time problems in parallel computer. This iterative method is based on a non-overlapping domain decomoposition in space, and the iteration requires subdomain solves with Dirichlet boundary
conditions followed by subdomain solves with Neumann boundary conditions. For a two-subdomain decomposition, we have proved superlinear convergence for 1D heat equation, and finite step convergence for 1D wave equation. In this paper, we extend the DNWR method to multiple subdomains, and present convergence analysis for one dimensional heat and wave equation. We also present convergence result for two dimensional wave equation. 

In a different viewpoint, the WR-type methods can be seen as an extension of DD methods for elliptic PDEs. The
systematic extension of the classical Schwarz method to time-dependent
parabolic problems was started in
\cite{GanStu,GilKel}; later optimized SWR methods have been introduced to achieve faster convergence or convergence with no overlap, see \cite{GH1} for parabolic problems, and
\cite{GHN} for hyperbolic problems. Recently Neumann-Neumann Waveform Relaxation (NNWR) algorithm is formulated from substructuring-type Neumann-Neumann algorithm \cite{BouGT,RoTal,TalRoV} to solve space-time problems; for more details see \cite{Kwok,Mandal2,BankThes}. The DNWR method thus can be regarded as an extension of Dirichlet-Neumann (DN) method for solving elliptic problems. The DN algorithm was first considered by Bj{\o}rstad \& Widlund \cite{BjWid} and further studied in \cite{BramPas}, \cite{MarQuar02} and \cite{MarQuar01}. The performance of the algorithm is now well
understood for elliptic problems, see for example the book \cite{TosWid} and the references therein.

We consider the following two PDEs on a bounded domain $\Omega\subset\mathbb{R}^{d},0<t<T$,
$d=1,2,3$, with a smooth boundary as our guiding examples: the parabolic problem
\begin{equation}
\begin{array}{rcll}
{\displaystyle \frac{\partial u}{\partial t}} & = & \nabla\cdot\left(\kappa(\boldsymbol{x},t)\nabla u\right)+f(\boldsymbol{x},t), & \boldsymbol{x}\in\Omega,\ 0<t<T,\\
u(\boldsymbol{x},0) & = & u_{0}(\boldsymbol{x}), & \boldsymbol{x}\in\Omega,\\
u(\boldsymbol{x},t) & = & g(\boldsymbol{x},t), & \boldsymbol{x}\in\partial\Omega,\ 0<t<T,
\end{array}\label{modelproblem1}
\end{equation}
where $\kappa(\boldsymbol{x},t)>0$, and the hyperbolic problem
\begin{equation}
\begin{array}{rcll}
\frac{\partial^{2}u}{\partial t^{2}}-c^{2}(\boldsymbol{x})\Delta u & = & f(\boldsymbol{x},t), & \boldsymbol{x}\in\Omega,0<t<T,\\
u(\boldsymbol{x},0) & = & v_{0}(\boldsymbol{x}), & \boldsymbol{x}\in\Omega,\\
u_{t}(\boldsymbol{x},0) & = & w_{0}(\boldsymbol{x}), & \boldsymbol{x}\in\Omega,\\
u(\boldsymbol{x},t) & = & g(\boldsymbol{x},t), & \boldsymbol{x}\in\partial\Omega,0<t<T,
\end{array}\label{modelwave}
\end{equation}
with $c(\boldsymbol{x})$ being a positive function.

We introduce in Section \ref{Section2} the non-overlapping DNWR algorithm with multiple subdomains for \eqref{modelproblem1}, and then 
analyze its convergence for the one dimensional heat equation. 
In Section \ref{Section3} we define this method to hyperbolic problems, and analyze convergence behavior for one dimensional wave equation. We extend our result to two dimensional wave equation and prove similar convergence behavior as in 1D in Section \ref{Section6}. Finally we present numerical results in Section \ref{Section7}, which illustrate our analysis.

\section{DNWR for parabolic problems}\label{Section2}

The Dirichlet-Neumann Waveform Relaxation (DNWR) method for parabolic problems with two subdomains is introduced in \cite{GKM1,Mandal}. In this section we generalize the 
algorithm to multiple subdomains
in one spatial dimension. We present different possible arrangements
(in terms of placing Dirichlet and Neumann boundary conditions) and
with a numerical implementation of these arrangements for a model
problem we determine the best possible one. We then formally define the DNWR method for \eqref{modelproblem1}, and analyze its convergence for the one dimensional heat equation. 

\subsection{Motivation}\label{Section2a}

Suppose we want to solve the 1D heat equation 
\begin{equation}
\begin{array}{rcll}
{\displaystyle \frac{\partial u}{\partial t}} & = & \Delta u, & x\in\Omega,\ 0<t<T,\\
u(x,0) & = & u_{0}(x), & x\in\Omega,\\
u(x,t) & = & g(x,t), & x\in\partial\Omega,\ 0<t<T,
\end{array}\label{MultiDNmodel}
\end{equation}
using the DNWR method. The spatial domain $\Omega=(0,5)$ is decomposed
into five non-overlapping subdomains $\Omega_{i}=(x_{i-1},x_{i}),i=1,\ldots,5$,
see the left panel of Figure \ref{NumFig1}, with three possible combinations of boundary
conditions along the interfaces, right panel of Figure \ref{NumFig1} and two arrangements in Figure \ref{NumFig2}. $D$ in blue denotes the Dirichlet
condition along the two physical boundaries, whereas $D$ and $N$
in red denote the Dirichlet and Neumann boundary conditions along
the interfaces. We are given Dirichlet traces $\left\{ g_{i}^{0}(t)\right\} _{i=1}^{4}$
as initial guesses along the interfaces $\left\{ x_{i}\right\} _{i=1}^{4}$.

\subsection*{First arrangement (A1):}

Here we extend the two subdomain-formulation \cite{GKM1,Mandal} to many subdomains in a natural
way, see the right panel of Figure \ref{NumFig1}. With the intial guesses, a Dirichlet subproblem
is solved in the first subdomain $\Omega_{1}$, followed by a series
of mixed Neumann-Dirichlet subproblem solves in the subsequent subdomains
($\Omega_{i},i=2,\ldots,5$), exactly like in the two-subdomain case. Thus
the DNWR algorithm is given by: for
$k=1,2\ldots$ and for $\theta\in(0,1]$ compute 
\[
\begin{array}{rcll}
\partial_{t}u_{1}^{k}-\partial_{xx}u_{1}^{k} & = & 0, & \textrm{in}\;\Omega_{1},t>0,\\
u_{1}^{k}(x,0) & = & u_{0}(x), & \textrm{in}\;\Omega_{1},\\
u_{1}^{k}(0,t) & = & g(0,t), & t>0,\\
u_{1}^{k}(x_{1},t) & = & g_{1}^{k-1}(t), & t>0,
\end{array}
\]
and for $i=2,\ldots,5$ 
\[
\begin{array}{rcll}
\partial_{t}u_{i}^{k}-\partial_{xx}u_{i}^{k} & = & 0, & \textrm{in}\;\Omega_{i},t>0,\\
u_{i}^{k}(x,0) & = & u_{0}(x), & \textrm{in}\;\Omega_{i},\\
\partial_{x}u_{i}^{k}(x_{i-1},t) & = & \partial_{x}u_{i-1}^{k}(x_{i-1},t), & t>0,\\
u_{i}^{k}(x_{i},t) & = & g_{i}^{k-1}(t), & t>0,
\end{array}
\]
with $g_{5}^{k}(t)=g(5,t)$ for the last subdomain along the physical
boundary. The updated interface values for the next step are then
defined as
\[
g_{i}^{k}(t)=\theta u_{i+1}^{k}(x_{i},t)+(1-\theta)g_{i}^{k-1}(t).
\]
\begin{figure}
  \centering
  \includegraphics[width=0.49\textwidth]{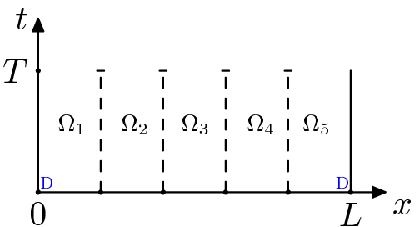}
  \includegraphics[width=0.49\textwidth]{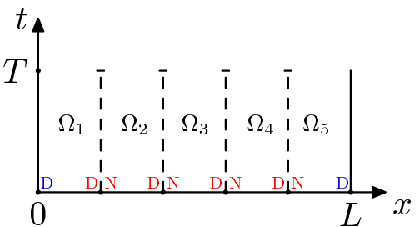}
  \caption{Decomposition of the domain on the left, and arrangement of boundary conditions due to A1 on the right}
  \label{NumFig1}
\end{figure}
\begin{figure}
  \centering
  \includegraphics[width=0.49\textwidth]{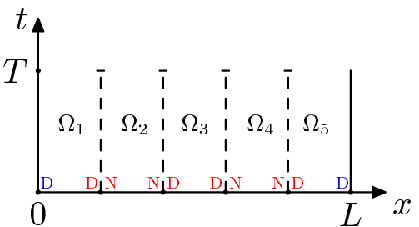}
  \includegraphics[width=0.49\textwidth]{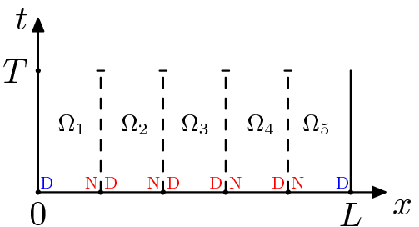}
  \caption{Two different arrangements of boundary conditions: A2 on the left, and A3 on the right}
  \label{NumFig2}
\end{figure}
We now discretize \eqref{MultiDNmodel} using standard centered
finite differences in space and backward Euler in time, and solve
the equation numerically using the above algorithm for different time
windows. For the test we choose $u_{0}(x)=0,g(x,t)=(x+1)t,g_{i}^{0}(t)=t^{2},t\in[0,T]$.
Figure \ref{FigDNmA1} gives the convergence curves for different values of
the parameter $\theta$ for $T=2$ on the left, and $T=20$ on the
right. 
\begin{figure}
\centering
\includegraphics[width=0.49\textwidth]{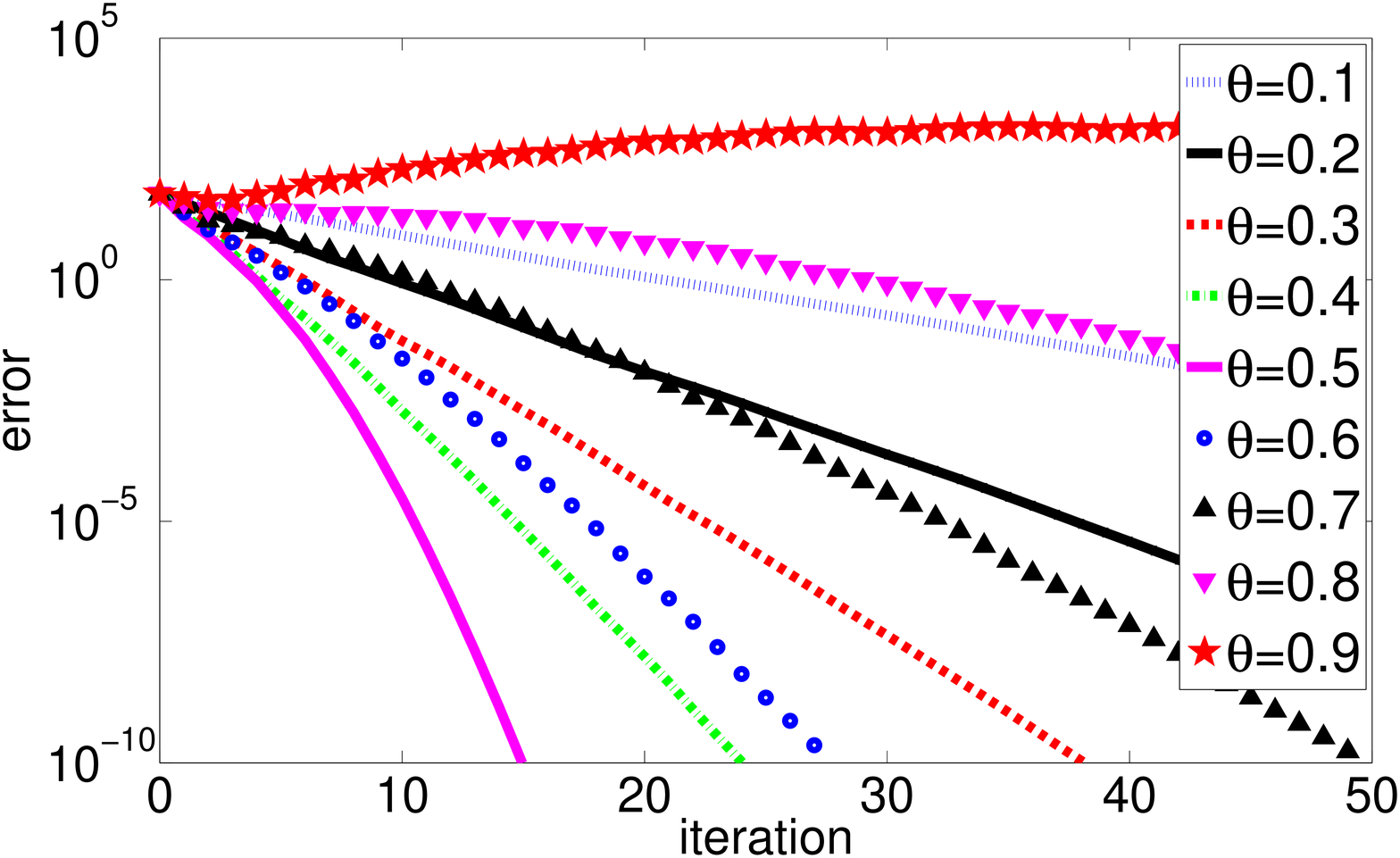}
\includegraphics[width=0.49\textwidth]{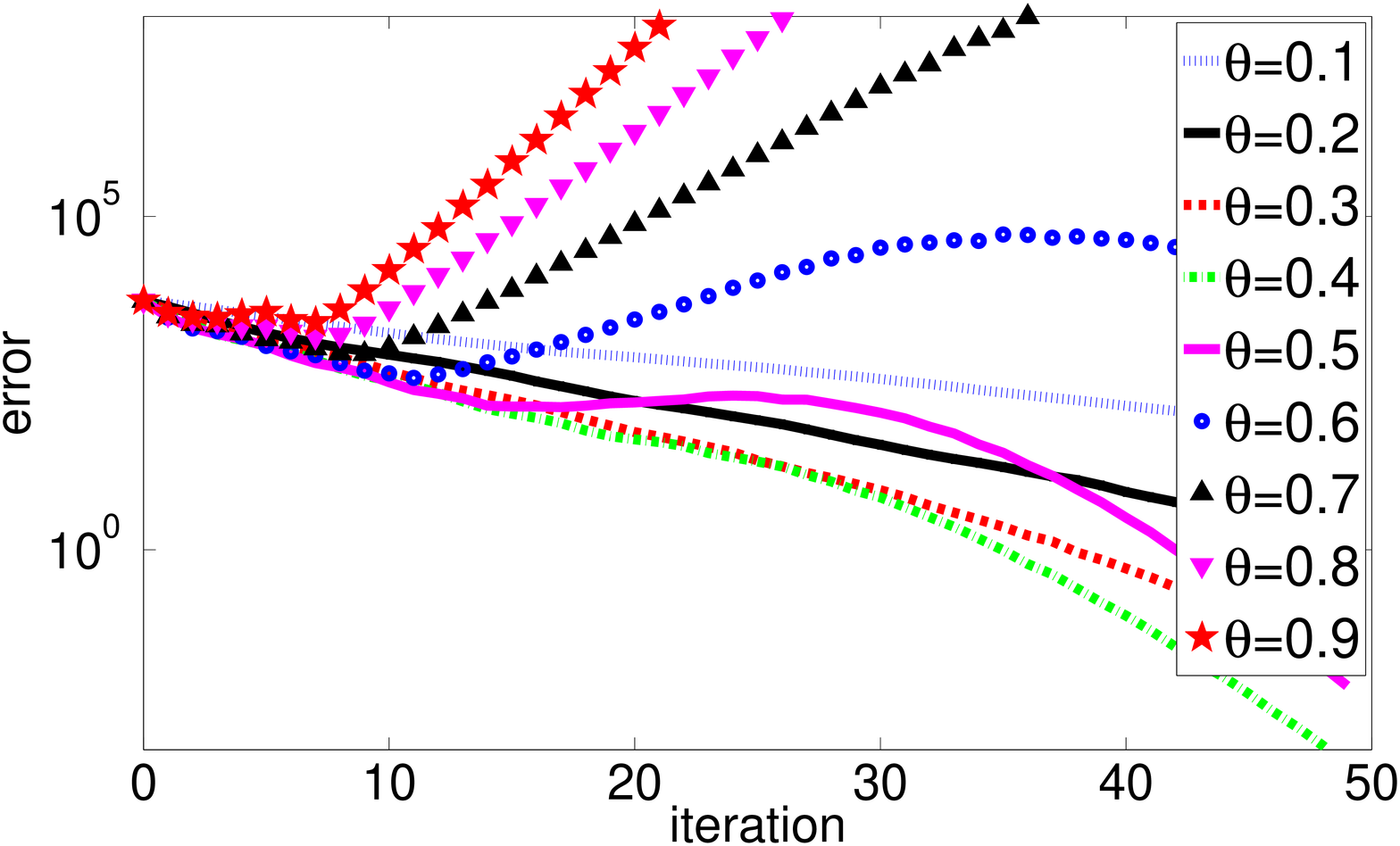}
\caption{Arrangement A1: Convergence of multi-subdomain DNWR for various relaxation parameters
$\theta$, on the left for $T=2$ and on the right for $T=20$}
\label{FigDNmA1}
\end{figure}

\subsection*{Second arrangement (A2):}

This is the well-known red-black block formulation, described in the left panel of Figure
\ref{NumFig2}. In this arrangement, we solve a Dirichlet subproblem and
a Neumann subproblem in alternating fashion. Given initial Dirichlet
traces along the interfaces, a series of Dirichlet subproblems is
first solved in parallel in alternating subdomains ($\Omega_{1},\Omega_{3},\Omega_{5}$),
and then a series of Neumann subproblems is solved in the remaining
subdomains ($\Omega_{2},\Omega_{4}$). So this type of DNWR algorithm is given by: for $k=1,2\ldots$ compute for $i=1,3,5$ 
\[
\begin{array}{rcll}
\partial_{t}u_{i}^{k}-\partial_{xx}u_{i}^{k} & = & 0, & \textrm{in}\;\Omega_{i},t>0,\\
u_{i}^{k}(x,0) & = & u_{0}(x), & \textrm{in}\;\Omega_{i},\\
u_{i}^{k}(x_{i-1},t) & = & g_{i-1}^{k-1}(t), & t>0,\\
u_{i}^{k}(x_{i},t) & = & g_{i}^{k-1}(t), & t>0,
\end{array}
\]
with $g_{0}^{k}(t)=g(0,t),g_{5}^{k}(t)=g(5,t)$, and for $i=2,4$
\[
\begin{array}{rcll}
\partial_{t}u_{i}^{k}-\partial_{xx}u_{i}^{k} & = & 0, & \textrm{in}\;\Omega_{i},t>0,\\
u_{i}^{k}(x,0) & = & u_{0}(x), & \textrm{in}\;\Omega_{i},\\
\partial_{x}u_{i}^{k}(x_{i-1},t) & = & \partial_{x}u_{i-1}^{k}(x_{i-1},t), & t>0,\\
\partial_{x}u_{i}^{k}(x_{i},t) & = & \partial_{x}u_{i+1}^{k}(x_{i},t), & t>0,
\end{array}
\]
together with the updating conditions 
\begin{eqnarray*}
g_{i}^{k}(t) & = & \theta u_{i+1}^{k}(x_{i},t)+(1-\theta)g_{i}^{k-1}(t),\; i=1,3,\\
g_{i}^{k}(t) & = & \theta u_{i}^{k}(x_{i},t)+(1-\theta)g_{i}^{k-1}(t),\; i=2,4,
\end{eqnarray*}
where $\theta\in(0,1]$ is a relaxation parameter.

We now implement this version of DNWR algorithm for different time
windows, picking the same problem and initial guesses as for A1. Figure
\ref{FigDNmA2} gives the convergence curves for different values of the parameter
$\theta$ for $T=2$ on the left, and $T=20$ on the right.
\begin{figure}
\centering
\includegraphics[width=0.49\textwidth]{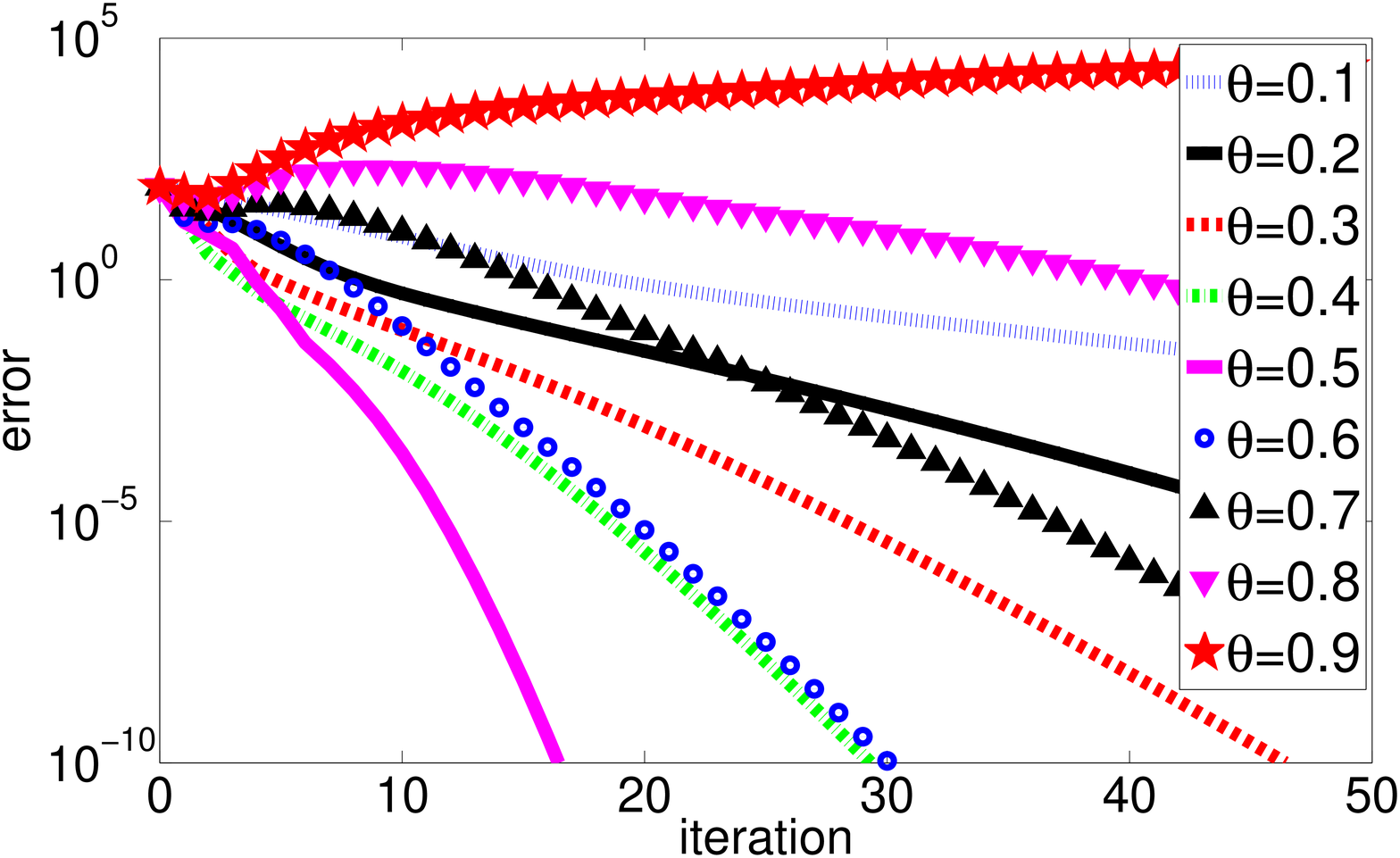}
\includegraphics[width=0.49\textwidth]{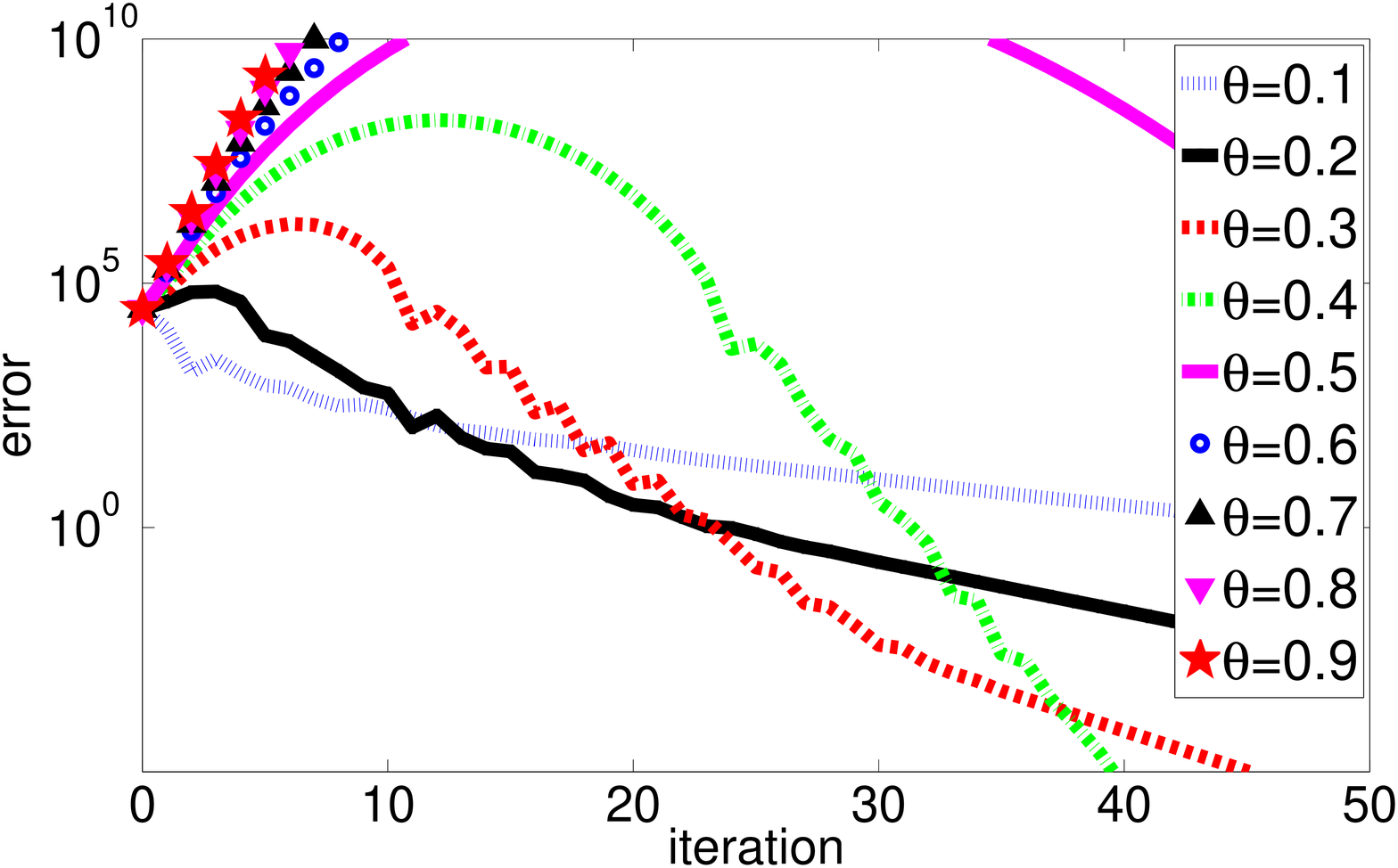}
\caption{Arrangement A2: Convergence of multi-subdomain DNWR for various relaxation parameters
$\theta$, on the left for $T=2$ and on the right for $T=20$}
\label{FigDNmA2}
\end{figure}

\subsection*{Third arrangement (A3):}
We now consider a completely different type of arrangement, proposed
in \cite{FunQuar} and shown in the right panel of Figure \ref{NumFig2}. Given initial guesses along
the interfaces, we begin with a Dirichlet solve in the middle subdomain
$\Omega_{3}$, followed by mixed Neumann-Dirichlet subproblem solves
in the adjacent subdomains, in an order $\Omega_{2},\Omega_{4}$ first
and then in $\Omega_{1},\Omega_{5}$. This third version of the
DNWR algorithms for multiple subdomains
is given by: for $k=1,2\ldots$ and for $\theta\in(0,1]$ compute
\begin{equation}
\begin{array}{rcll}
\partial_{t}u_{3}^{k}-\partial_{xx}u_{3}^{k} & = & 0, & \textrm{in}\;\Omega_{3},t>0,\\
u_{3}^{k}(x,0) & = & u_{0}(x), & \textrm{in}\;\Omega_{3},\\
u_{3}^{k}(x_{2},t) & = & g_{2}^{k-1}(t), & t>0,\\
u_{3}^{k}(x_{3},t) & = & g_{3}^{k-1}(t), & t>0,
\end{array}\label{A3-1}
\end{equation}
 and then for $i=2,1$ 
\begin{equation}
\begin{array}{rcll}
\partial_{t}u_{i}^{k}-\partial_{xx}u_{i}^{k} & = & 0, & \textrm{in}\;\Omega_{i},t>0,\\
u_{i}^{k}(x,0) & = & u_{0}(x), & \textrm{in}\;\Omega_{i},\\
u_{i}^{k}(x_{i-1},t) & = & g_{i-1}^{k-1}(t), & t>0,\\
\partial_{x}u_{i}^{k}(x_{i},t) & = & \partial_{x}u_{i+1}^{k}(x_{i},t), & t>0,
\end{array}\label{A3-2}
\end{equation}
and finally for $i=4,5$ 
\begin{equation}
\begin{array}{rcll}
\partial_{t}u_{i}^{k}-\partial_{xx}u_{i}^{k} & = & 0, & \textrm{in}\;\Omega_{i},t>0,\\
u_{i}^{k}(x,0) & = & u_{0}(x), & \textrm{in}\;\Omega_{i},\\
\partial_{x}u_{i}^{k}(x_{i-1},t) & = & \partial_{x}u_{i-1}^{k}(x_{i-1},t), & t>0,\\
u_{i}^{k}(x_{i},t) & = & g_{i}^{k-1}(t), & t>0,
\end{array}\label{A3-3}
\end{equation}
with $g_{0}^{k}(t)=g(0,t),g_{5}^{k}(t)=g(5,t)$ for the first and
last subdomains at the physical boundaries. The updated interface
values for the next step are defined as
\begin{eqnarray*}
g_{i}^{k}(t) & = & \theta u_{i}^{k}(x_{i},t)+(1-\theta)g_{i}^{k-1}(t),\; i=1,2,\\
g_{i}^{k}(t) & = & \theta u_{i+1}^{k}(x_{i},t)+(1-\theta)g_{i}^{k-1}(t),\; i=3,4.
\end{eqnarray*}
We solve \eqref{MultiDNmodel} using the above DNWR algorithm for
different time windows for the same setting as in A1. Figure \ref{FigDNmA3}
gives the convergence curves for different values of the parameter
$\theta$ for $T=2$ on the left, and $T=20$ on the right. 
\begin{figure}
\centering
\includegraphics[width=0.49\textwidth]{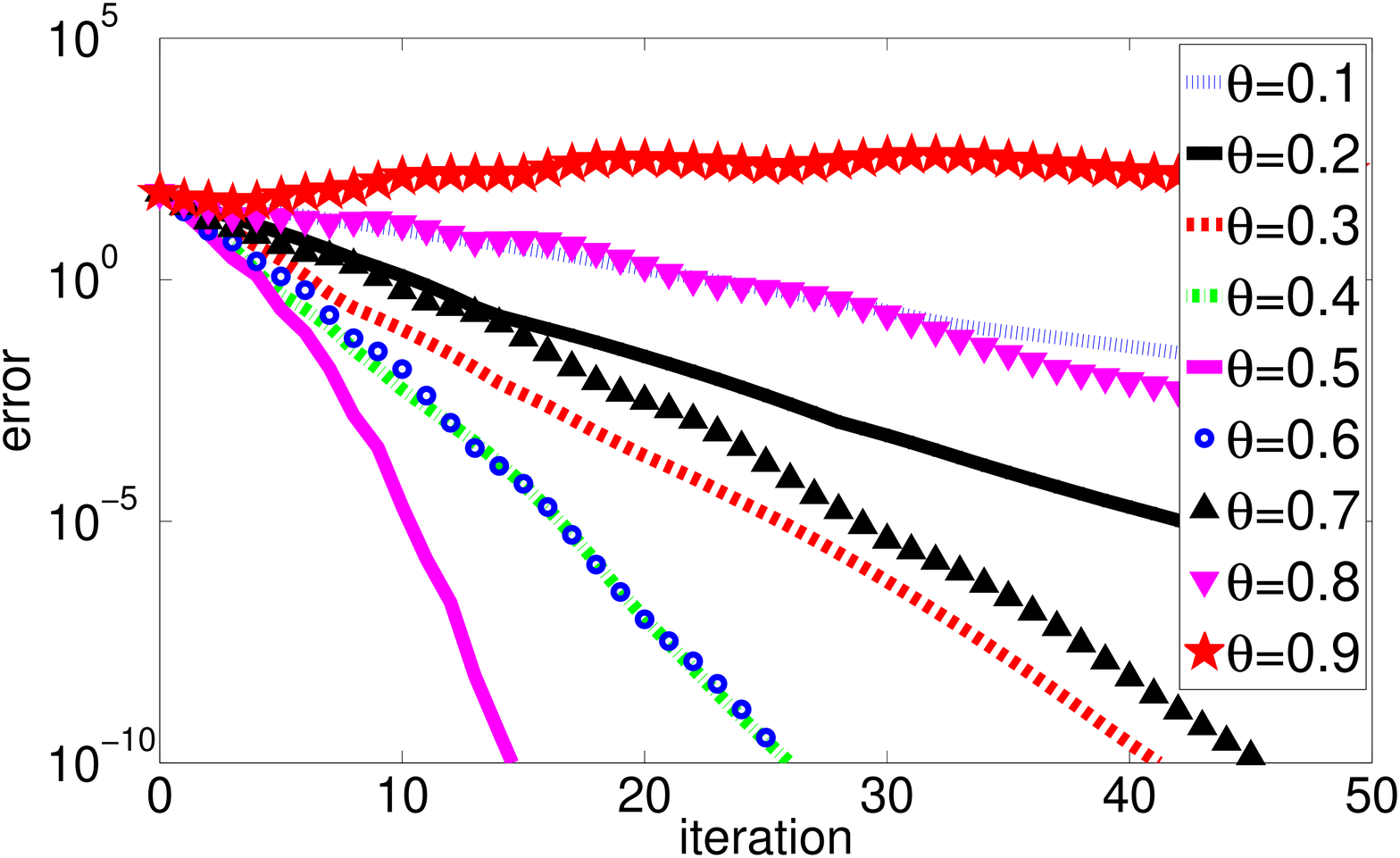}
\includegraphics[width=0.49\textwidth]{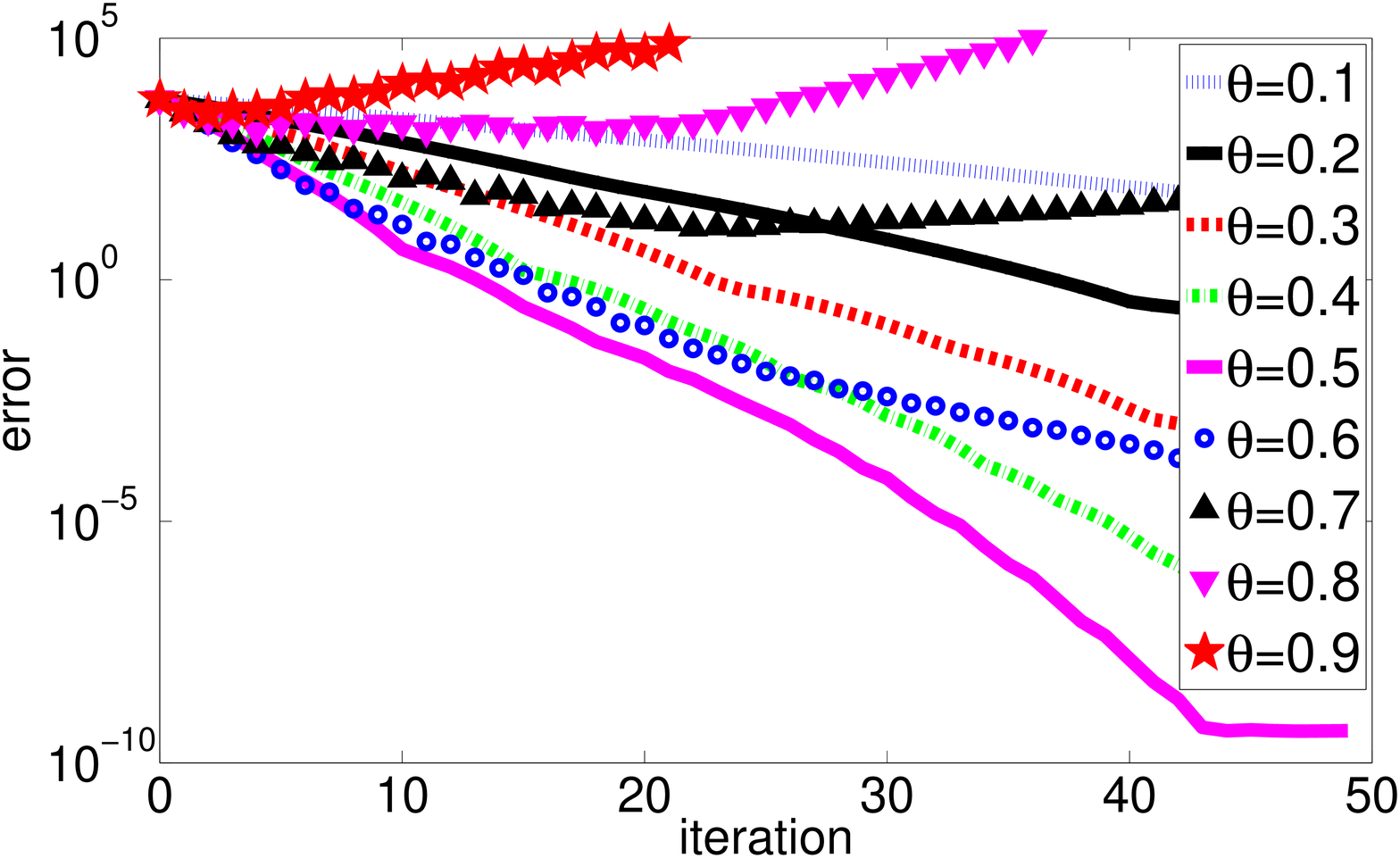}
\caption{Arrangement A3: Convergence of multi-subdomain DNWR for various relaxation parameters $\theta$, on the left for $T=2$ and on the right for $T=20$}
\label{FigDNmA3}
\end{figure}

From the three numerical tests of the DNWR methods (A1, A2 and A3),
it is evident that the behavior of these algorithms are similar for
smaller time windows. But we notice clearly faster convergence for
the arrangement A3 for large time windows. We therefore focus on the
third version (A3) of the DNWR algorithms, and formally define the
DNWR method for the general parabolic model problem \eqref{modelproblem1}
for multiple subdomains in the next subsection. 

\subsection{DNWR algorithm}

We now formally define the Dirichlet-Neumann Waveform Relaxation method
for the model problem \eqref{modelproblem1} on the space-time domain
$\Omega\times(0,T)$ with Dirichlet data given on $\partial\Omega$.
Suppose the spatial domain $\Omega$ is partitioned into $2m+1$ non-overlapping
subdomains $\Omega_{i},i=1,\ldots,2m+1$ without any cross-points, as illustrated 
in Figure
\ref{FigDNmGen}. We denote by $u_{i}$ the restriction of the solution $u$
of \eqref{modelproblem1} to $\Omega_{i}$. For $i=1,\ldots,2m$,
set $\Gamma_{i}:=\partial\Omega_{i}\cap\partial\Omega_{i+1}$. We
further define $\Gamma_{0}=\Gamma_{2m+1}=\emptyset$. We denote by
$\boldsymbol{n}_{i,j}$ the unit outward normal for $\Omega_{i}$
on the interface $\Gamma_{j},j=i-1,i$ (for $\Omega_{1},\Omega_{2m+1}$
we have only $\boldsymbol{n}_{1,2}$ and $\boldsymbol{n}_{2m+1,2m}$
respectively). In Figure \ref{FigDNmGen}, $D$ and $N$ in red denote the Dirichlet
and Neumann boundary conditions respectively along the interfaces
as in the arrangement A3. 

The DNWR algorithm starts with initial Dirichlet traces $g_{i}^{0}(\boldsymbol{x},t)$
along the interfaces $\Gamma_{i}\times(0,T)$, $i=1,\ldots,2m$, and
then performs the following computation for
$k=1,2,\ldots$
\begin{equation}
\begin{array}{rcll}
\partial_{t}u_{m+1}^{k}-\nabla\cdot\left(\kappa(\boldsymbol{x},t)\nabla u_{m+1}^{k}\right) & = & f, & \textrm{in}\;\Omega_{m+1},\\
u_{m+1}^{k}(\boldsymbol{x},0) & = & u_{0}(\boldsymbol{x}), & \textrm{in}\;\Omega_{m+1},\\
u_{m+1}^{k} & = & g, & \textrm{on}\;\partial\Omega\cap\partial\Omega_{m+1},\\
u_{m+1}^{k} & = & g_{i}^{k-1}, & \textrm{on}\;\Gamma_{i},i=m,m+1,
\end{array}\label{DNm3_1}
\end{equation}
and then for $m\geq i\geq1$ and $m+2\leq j\leq2m+1$
\begin{equation}
\arraycolsep0.001em\begin{array}{rcll}
\partial_{t}u_{i}^{k} & = & \nabla\cdot\left(\kappa(\boldsymbol{x},t)\nabla u_{i}^{k}\right)+f, & \,\textrm{in}\;\Omega_{i},\\
u_{i}^{k}(\boldsymbol{x},0) & = & u_{0}(\boldsymbol{x}), & \textrm{in}\;\Omega_{i},\\
u_{i}^{k} & = & \widetilde{g}_{i}^{k-1}, & \textrm{on}\;\partial\Omega_{i}\setminus\Gamma_{i},\\
\partial_{\boldsymbol{n}_{i,i+1}}u_{i}^{k} & = & -\partial_{\boldsymbol{n}_{i+1,i}}u_{i+1}^{k}, & \textrm{on}\;\Gamma_{i},
\end{array}\;\begin{array}{rcll}
\partial_{t}u_{j}^{k} & = & \nabla\cdot\left(\kappa(\boldsymbol{x},t)\nabla u_{j}^{k}\right)+f, & \,\textrm{in}\;\Omega_{j},\\
u_{j}^{k}(\boldsymbol{x},0) & = & u_{0}(\boldsymbol{x}), & \textrm{in}\;\Omega_{j},\\
\partial_{\boldsymbol{n}_{j,j-1}}u_{j}^{k} & = & -\partial_{\boldsymbol{n}_{j-1,j}}u_{j-1}^{k}, & \textrm{on}\;\Gamma_{j-1},\\
u_{j}^{k} & = & \check{g}_{j}^{k-1}, & \textrm{on}\;\partial\Omega_{j}\setminus\Gamma_{j-1},
\end{array}\label{DNm3_2}
\end{equation}
with the update conditions along the interfaces 
\begin{equation}
\arraycolsep0.1em\begin{array}{rcl}
g_{i}^{k}(\boldsymbol{x},t) & = & \theta u_{i}^{k}\left|_{\Gamma_{i}\times(0,T)}\right.+(1-\theta)g_{i}^{k-1}(\boldsymbol{x},t),\;1\leq i\leq m,\\
g_{j}^{k}(\boldsymbol{x},t) & = & \theta u_{j+1}^{k}\left|_{\Gamma_{j}\times(0,T)}\right.+(1-\theta)g_{j}^{k-1}(\boldsymbol{x},t),\; m+1\leq j\leq2m,
\end{array} \label{DNm3_3}
\end{equation}
where $\theta\in(0,1]$, and for $i=1,\ldots,m$ and $j=m+2,\ldots,2m+1$ 
\[
\widetilde{g}_{i}^{k-1}=\begin{cases}
g, & \textrm{on}\;\partial\Omega\cap\partial\Omega_{i},\\
g_{i-1}^{k-1}, & \textrm{on}\;\Gamma_{i-1},
\end{cases}, \quad \check{g}_{j}^{k-1}=\begin{cases}
g, & \textrm{on}\;\partial\Omega\cap\partial\Omega_{j},\\
g_{j}^{k-1}, & \textrm{on}\;\Gamma_{j},
\end{cases}.
\]  
\begin{figure}
\centering{}\includegraphics[width=8cm,height=5.5cm]{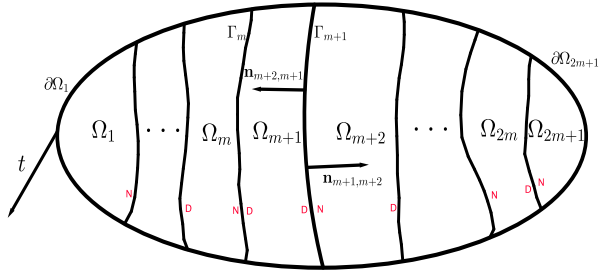}
\caption{Splitting into many non-overlapping strip-like subdomains}
\label{FigDNmGen}
\end{figure}
\begin{remark}\label{Rem1}
The DNWR algorithm \eqref{DNm3_1}-\eqref{DNm3_2}-\eqref{DNm3_3}
is defined for an odd number of subdomains. In case of an even number
of subdomains $2m+2$, we treat in a similar way as above the first
$2m+1$ subdomains, keeping the last one aside. Then for the last
subdomain, we apply a Neumann transmission condition along the interface
$\Gamma_{2m+1}$ and a Dirichlet boundary condition along the physical
boundary.
\end{remark}

\subsection{Convergence analysis}

We present the convergence result of the DNWR algorithm \eqref{DNm3_1}-\eqref{DNm3_2}-\eqref{DNm3_3} for the 1D heat equation with $\kappa(x,t)=\nu$. We split
the domain $\Omega:=(0,L)$ into non-overlapping subdomains $\Omega_{i}:=(x_{i-1},x_{i})$,
$i=1,\ldots,2m+1$, and define the subdomain length $h_{i}:=x_{i}-x_{i-1}$.
Also, define the physical boundary conditions as $u(0,t)=g_{0}(t)$
and $u(L,t)=g_{L}(t)$, which in turn become zeros as we consider
the error equations, $f(x,t)=0,g_{0}(t)=g_{L}(t)=0=u_{0}(x)$.
We take $\left\{ g_{i}^{0}(t)\right\} _{i=1}^{2m}$ as initial guesses
along the interfaces $\left\{ x=x_{i}\right\} \times(0,T)$, and for
sake of consistency we denote $g_{0}^{k}(t)=g_{2m+1}^{k}=0$ for all
$k$ corresponding to the physical boundaries. Denoting
by $\left\{ w_{i}^{k}(t)\right\} _{i=1}^{2m}$ for $k=1,2\ldots$
the Neumann traces along the interfaces, we compute
\begin{equation}
\begin{array}{rcll}
\partial_{t}u_{m+1}^{k}-\nu\partial_{xx}u_{m+1}^{k} & = & 0, & x\in\Omega_{m+1},\\
u_{m+1}^{k}(x,0) & = & 0, & x\in\Omega_{m+1},\\
u_{m+1}^{k}(x_{m},t) & = & g_{m}^{k-1}(t),\\
u_{m+1}^{k}(x_{m+1},t) & = & g_{m+1}^{k-1}(t).
\end{array}\label{DNmultieq1}
\end{equation}
and then for $m\geq i\geq1$ and $m+2\leq j\leq2m+1$
\begin{equation}
\arraycolsep0.08em\begin{array}{rcll}
\partial_{t}u_{i}^{k}-\nu\partial_{xx}u_{i}^{k} & = & 0, & x\in\Omega_{i},\\
u_{i}^{k}(x,0) & = & 0, & x\in\Omega_{i},\\
u_{i}^{k}(x_{i-1},t) & = & g_{i-1}^{k-1}(t),\\
\partial_{x}u_{i}^{k}(x_{i},t) & = & w_{i}^{k}(t),
\end{array}\;\begin{array}{rcll}
\partial_{t}u_{j}^{k}-\nu\partial_{xx}u_{j}^{k} & = & 0, & x\in\Omega_{j},\\
u_{j}^{k}(x,0) & = & 0, & x\in\Omega_{j},\\
-\partial_{x}u_{j}^{k}(x_{j-1},t) & = & w_{j-1}^{k}(t),\\
u_{j}^{k}(x_{j},t) & = & g_{j}^{k-1}(t),
\end{array}\label{DNmultieq2}
\end{equation}
and finally the update conditions with the parameter $\theta\in(0,1]$
\begin{equation}
\arraycolsep0.1em\begin{array}{rclrcl}
g_{i}^{k}(t) & = & \theta u_{i}^{k}(x_{i},t)+(1-\theta)g_{i}^{k-1}(t), & w_{i}^{k}(t) & = & \partial_{x}u_{i+1}^{k}(x_{i},t),\;1\leq i\leq m,\\
g_{j}^{k}(t) & = & \theta u_{j+1}^{k}(x_{j},t)+(1-\theta)g_{j}^{k-1}(t), & w_{j}^{k}(t) & = & -\partial_{x}u_{j}^{k}(x_{j},t),\; m+1\leq j\leq2m.
\end{array}\label{DNmultieq3}
\end{equation}
We have the following main convergence result for DNWR for the heat equation and the proof will be given in Section \ref{Section5}.

\begin{theorem}[Convergence of DNWR for multiple subdomains]\label{TheorDNmheat}
For $\theta=1/2$ and $T>0$ fixed, the DNWR algorithm \eqref{DNmultieq1}-\eqref{DNmultieq2}-\eqref{DNmultieq3}
for multiple subdomains of unequal sizes $h_{i},1\leq i\leq2m+1$ with $h_{\max}:=\max_{1\leq i\leq2m+1}h_{i}$
and $h_{\min}:=\min_{1\leq i\leq2m+1}h_{i}$ converges superlinearly
with the estimate
\[
{\displaystyle \max_{1\leq i\leq2m}}\parallel g_{i}^{k}\parallel_{L^{\infty}(0,T)}\leq\left(2m-3+\frac{2h_{\max}}{h_{m+1}}\right)^{k}{\rm erfc}\left(\frac{kh_{\min}}{2\sqrt{\nu T}}\right){\displaystyle \max_{1\leq i\leq2m}}\parallel g_{i}^{0}\parallel_{L^{\infty}(0,T)}.
\]
\end{theorem}
\begin{remark}\label{Rem2}
The estimate in Theorem \ref{TheorDNmheat} holds for an odd number
of subdomains. In case of an even number of subdomains $2m+2$,
we define the algorithm as given in Remark \ref{Rem1}, and the estimate will
be of the form
\begin{equation}\label{DNmevenest}
{\displaystyle \max_{1\leq i\leq2m+1}}\parallel g_{i}^{k}\parallel_{L^{\infty}(0,T)}\leq\left(2m-1+\frac{2h_{\max}}{h_{m+1}}\right)^{k}{\rm erfc}\left(\frac{kh_{\min}}{2\sqrt{\nu T}}\right){\displaystyle \max_{1\leq i\leq2m+1}}\parallel g_{i}^{0}\parallel_{L^{\infty}(0,T)}.
\end{equation}
$2m-1$ appears in \eqref{DNmevenest} as we calculate the estimate
for $2m+3$ subdomains. 
\end{remark}

\section{DNWR for hyperbolic problems}\label{Section3}

In this section we define the Dirichlet-Neumann Waveform Relaxation
method with many subdomains for the model problem \eqref{modelwave} on the space-time
domain $\Omega\times(0,T)$.
This can be treated as a generalization of the DNWR algorithm for
two subdomains, for which see \cite{GKM3}. As before, there are three
possible arrangements: A1, A2 and A3. However, for consistency we formally define below the DNWR algorithm for the arrangement A3. 

\subsection{DNWR algorithm}

Suppose the spatial domain $\Omega$ is partitioned into $2m+1$ non-overlapping
subdomains $\Omega_{i},i=1,\ldots,2m+1$, without any cross-points
as illustrated in Figure \ref{FigDNmGen}. For $i=1,\ldots,2m$,
set $\Gamma_{i}:=\partial\Omega_{i}\cap\partial\Omega_{i+1}$. 
For consistency, we
set $\Gamma_{0}=\Gamma_{2m+1}=\emptyset$. We denote by $u_{i}$ the restriction
of the solution $u$ of \eqref{modelwave} to $\Omega_{i}$, and by
$\boldsymbol{n}_{i,j}$ the unit outward normal for $\Omega_{i}$
on the interface $\Gamma_{j},j=i-1,i$ (for $\Omega_{1},\Omega_{2m+1}$
we have only $\boldsymbol{n}_{1,2}$ and $\boldsymbol{n}_{2m+1,2m}$
respectively). We define the DNWR method as in the arrangement A3,
but one can also consider as in A1 and A2.

Given initial Dirichlet traces $g_{i}^{0}(\boldsymbol{x},t)$ along
the interfaces $\Gamma_{i}\times(0,T)$, $i=1,\ldots,2m$, the DNWR
algorithm consists of the following computation for $k=1,2,\ldots$
\begin{equation}
\begin{array}{rcll}
\partial_{tt}u_{m+1}^{k}-c^{2}(\boldsymbol{x})\Delta u_{m+1}^{k} & = & f, & \textrm{in}\;\Omega_{m+1},\\
u_{m+1}^{k}(\boldsymbol{x},0) & = & v_{0}(\boldsymbol{x}), & \textrm{in}\;\Omega_{m+1},\\
\partial_{t}u_{m+1}^{k}(\boldsymbol{x},0) & = & w_{0}(\boldsymbol{x}), & \textrm{in}\;\Omega_{m+1},\\
u_{m+1}^{k} & = & g, & \textrm{on}\;\partial\Omega\cap\partial\Omega_{m+1},\\
u_{m+1}^{k} & = & g_{i}^{k-1}, & \textrm{on}\;\Gamma_{i},i=m,m+1,
\end{array}\label{DNmwave1}
\end{equation}
and then for $m\geq i\geq1$ and $m+2\leq j\leq2m+1$
\begin{equation}
\arraycolsep0.08em\begin{array}{rcll}
\partial_{tt}u_{i}^{k} & = & c^{2}(\boldsymbol{x})\Delta u_{i}^{k} + f, & \,\textrm{in}\;\Omega_{i},\\
u_{i}^{k}(\boldsymbol{x},0) & = & v_{0}(\boldsymbol{x}), & \textrm{in}\;\Omega_{i},\\
\partial_{t}u_{i}^{k}(\boldsymbol{x},0) & = & w_{0}(\boldsymbol{x}), & \textrm{in}\;\Omega_{i},\\
u_{i}^{k} & = & \widetilde{g}_{i}^{k-1}, & \textrm{on}\;\partial\Omega_{i}\setminus\Gamma_{i},\\
\partial_{\boldsymbol{n}_{i,i+1}}u_{i}^{k} & = & -\partial_{\boldsymbol{n}_{i+1,i}}u_{i+1}^{k}, & \textrm{on}\;\Gamma_{i},
\end{array}\;\begin{array}{rcll}
\partial_{tt}u_{j}^{k} & = & c^{2}(\boldsymbol{x})\Delta u_{j}^{k} + f, & \,\textrm{in}\;\Omega_{j},\\
u_{j}^{k}(\boldsymbol{x},0) & = & v_{0}(\boldsymbol{x}), & \textrm{in}\;\Omega_{j},\\
\partial_{t}u_{j}^{k}(\boldsymbol{x},0) & = & w_{0}(\boldsymbol{x}), & \textrm{in}\;\Omega_{j},\\
\partial_{\boldsymbol{n}_{j,j-1}}u_{j}^{k} & = & -\partial_{\boldsymbol{n}_{j-1,j}}u_{j-1}^{k}, & \textrm{on}\;\Gamma_{j-1},\\
u_{j}^{k} & = & \check{g}_{j}^{k-1}, & \textrm{on}\;\partial\Omega_{j}\setminus\Gamma_{j-1},
\end{array}\label{DNmwave2}
\end{equation}
with the update conditions along the interfaces
\begin{equation}
\arraycolsep0.1em\begin{array}{rcl}
g_{i}^{k}(\boldsymbol{x},t) & = & \theta u_{i}^{k}\left|_{\Gamma_{i}\times(0,T)}\right.+(1-\theta)g_{i}^{k-1}(\boldsymbol{x},t),\;1\leq i\leq m,\\
g_{j}^{k}(\boldsymbol{x},t) & = & \theta u_{j+1}^{k}\left|_{\Gamma_{j}\times(0,T)}\right.+(1-\theta)g_{j}^{k-1}(\boldsymbol{x},t),\;m+1\leq j\leq2m,
\end{array}\label{DNmwave3}
\end{equation}
where $\theta\in(0,1]$, and for $i=1,\ldots,m$ and $j=m+2,\ldots,2m+1$ 
\[
\widetilde{g}_{i}^{k-1}=\begin{cases}
g, & \textrm{on}\;\partial\Omega\cap\partial\Omega_{i},\\
g_{i-1}^{k-1}, & \textrm{on}\;\Gamma_{i-1},
\end{cases}, \quad \check{g}_{j}^{k-1}=\begin{cases}
g, & \textrm{on}\;\partial\Omega\cap\partial\Omega_{j},\\
g_{j}^{k-1}, & \textrm{on}\;\Gamma_{j},
\end{cases}.
\] 

\subsection{Convergence analysis}

We analyze the DNWR algorithm \eqref{DNmwave1}-\eqref{DNmwave2}-\eqref{DNmwave3} for the 1D wave equation with a constant wave speed, $c(x)=c$. Consider a splitting
of the domain $\Omega:=(0,L)$ into non-overlapping subdomains $\Omega_{i}:=(x_{i-1},x_{i})$,
$i=1,\ldots,2m+1$, and define the subdomain length $g_{i}:=x_{i}-x_{i-1}$.
As we consider the error equations, the physical boundary conditions $u(0,t)=g_{0}(t)$
and $u(L,t)=g_{L}(t)$ become zeros along with $f(x,t)=0,v_{0}(x)=0=w_{0}(x)$.
We take $\left\{ g_{i}^{0}(t)\right\} _{i=1}^{2m}$ as initial guesses
along the interfaces $\left\{ x=x_{i}\right\} \times(0,T)$, and for
sake of consistency we denote $g_{0}^{k}(t)=g_{2m+1}^{k}=0$ for all
$k$ corresponding to the physical boundaries. Denoting by $\left\{ w_{i}^{k}(t)\right\} _{i=1}^{2m}$
for $k=1,2\ldots$ the Neumann traces along the interfaces, we compute
\begin{equation}
\begin{array}{rcll}
\partial_{tt}u_{m+1}^{k}-c^{2}\partial_{xx}u_{m+1}^{k} & = & 0, & x\in\Omega_{m+1},\\
u_{m+1}^{k}(x,0) & = & 0, & x\in\Omega_{m+1},\\
\partial_{t}u_{m+1}^{k}(x,0) & = & 0, & x\in\Omega_{m+1},\\
u_{m+1}^{k}(x_{m},t) & = & g_{m}^{k-1}(t),\\
u_{m+1}^{k}(x_{m+1},t) & = & g_{m+1}^{k-1}(t).
\end{array}\label{DNmwaveA3-1}
\end{equation}
and then for $m\geq i\geq1$ and $m+2\leq j\leq2m+1$
\begin{equation}
\arraycolsep0.08em\begin{array}{rcll}
\partial_{tt}u_{i}^{k}-c^{2}\partial_{xx}u_{i}^{k} & = & 0, & x\in\Omega_{i},\\
u_{i}^{k}(x,0) & = & 0, & x\in\Omega_{i},\\
\partial_{t}u_{i}^{k}(x,0) & = & 0, & x\in\Omega_{i},\\
u_{i}^{k}(x_{i-1},t) & = & g_{i-1}^{k-1}(t),\\
\partial_{x}u_{i}^{k}(x_{i},t) & = & w_{i}^{k}(t),
\end{array}\;\begin{array}{rcll}
\partial_{tt}u_{j}^{k}-c^{2}\partial_{xx}u_{j}^{k} & = & 0, & x\in\Omega_{j},\\
u_{j}^{k}(x,0) & = & 0, & x\in\Omega_{j},\\
\partial_{t}u_{j}^{k}(x,0) & = & 0, & x\in\Omega_{j},\\
-\partial_{x}u_{j}^{k}(x_{j-1},t) & = & w_{j-1}^{k}(t),\\
u_{j}^{k}(x_{j},t) & = & g_{j}^{k-1}(t),
\end{array}\label{DNmwaveA3-2}
\end{equation}
and finally the update conditions with the parameter $\theta\in(0,1]$
\begin{equation}
\arraycolsep0.1em\begin{array}{rclrcl}
g_{i}^{k}(t) & = & \theta u_{i}^{k}(x_{i},t)+(1-\theta)g_{i}^{k-1}(t), & w_{i}^{k}(t) & = & \partial_{x}u_{i+1}^{k}(x_{i},t),\;1\leq i\leq m,\\
g_{j}^{k}(t) & = & \theta u_{j+1}^{k}(x_{j},t)+(1-\theta)g_{j}^{k-1}(t), & w_{j}^{k}(t) & = & -\partial_{x}u_{j}^{k}(x_{j},t),\;m+1\leq j\leq2m.
\end{array}\label{DNmwaveA3-3}
\end{equation}
We now state the main convergence result for DNWR for the wave equation. The proof of Theorem \ref{DNwavemulti} will also be given in Section \ref{Section5}.

\begin{theorem}[Convergence of DNWR for multiple subdomains]\label{DNwavemulti} 
Let $\theta=1/2$. Then the DNWR algorithm \eqref{DNmwaveA3-1}-\eqref{DNmwaveA3-2}-\eqref{DNmwaveA3-3}
converges in at most $k+1$ iterations for multiple subdomains, if
the time window length $T$ satisfies $T/k\leq h_{\min}/c$, where
$c$ is the wave speed. 
\end{theorem}

\section{Auxiliary results}\label{Section4}
We need a few auxiliary results related to Laplace transform to prove our main convergence results. We define the convolution of two functions $g,w: (0,\infty)\rightarrow\mathbb{R}$ by
    $$(g*w)(t):=\int_{0}^{t}g(t-\tau)w(\tau)d\tau.$$
\begin{lemma}\label{SimpleLaplaceLemma}
Let $g$ and $w$ be two real-valued functions in $(0,\infty)$ with
$\hat{w}(s)=\mathcal{L}\left\{ w(t)\right\}$ the Laplace transform of
$w$. Then for $t\in(0,T)$, we have the following properties:
\begin{enumerate}
  \item \label{L1}If $g(t)\geq0$ and $w(t)\geq0$, then $(g*w)(t)\geq0$.

  \item \label{L2} $\| g*w\|_{L^{1}(0,T)}\leq\| g\|_{L^{1}(0,T)}\| w\|_{L^{1}(0,T)}.$

  \item \label{L3} $\bigl|(g*w)(t)\bigl|\leq\| g\|_{L^{\infty}(0,T)}\int_{0}^{T}\bigl|w(\tau)\bigl|d\tau.$

 \item \label{L4} $\int_{0}^{t}w(\tau)d\tau=(H*w)(t)=\mathcal{L}^{-1}\left(
   \frac{\hat{w}(s)}{s}\right) $, $H(t)$ being the Heaviside step
   function.
\end{enumerate}
\end{lemma}
\begin{proof}
  The proofs follow directly from the definitions.
\hfill \end{proof}
\begin{lemma}[Limit]\label{LimitLemma}
  Let, $w(t)$ be a continuous and $L^{1}$-integrable function on
  $(0,\infty)$ with $w(t)\geq0$ for all $t\geq0$, and
  $\hat{w}(s)=\mathcal{L}\left\{ w(t)\right\}$ be its Laplace
  transform. Then, for $\tau>0$, we have the bound
  \begin{equation*}
    \int_{0}^{\tau}|w(t)|dt\leq \lim_{s\rightarrow0+}\hat{w}(s).
  \end{equation*}
\end{lemma}
\begin{proof}
For a proof of this lemma, see \cite{GKM1}.
\hfill \end{proof}
\begin{lemma}[Positivity]\label{PositivityLemma}
  Let $\beta>\alpha\geq0$ and $s$ be a complex variable. Then, for
$t\in(0,\infty)$
\[
  \varphi(t):=\mathcal{L}^{-1}\left\{ \frac{\sinh(\alpha\sqrt{s})}
    {\sinh(\beta\sqrt{s})}\right\} \geq0\quad\mbox{and}\quad
  \psi(t):=\mathcal{L}^{-1}\left\{ \frac{\cosh(\alpha\sqrt{s})}
  {\cosh(\beta\sqrt{s})}\right\} \geq0.
\]
\end{lemma}
\begin{proof}
For a proof, we again cite \cite{GKM1}.
\hfill \end{proof}
\begin{lemma}\label{KernelPositive}
  Let $\alpha,\beta > 0$ be two real numbers and $s$ be a complex variable. Set
  \begin{equation}\label{Kernel}
  \chi(s):= \frac{\sinh((\alpha-\beta)\sqrt{s})}{\cosh(\alpha\sqrt{s})\sinh(\beta\sqrt{s})}.
  \end{equation}
  Then 
  \[
  \int_{0}^{T}\bigl|\mathcal{L}^{-1}\left\{\chi(\cdot)\right\}(\tau)\bigr|d\tau={\displaystyle \lim_{s\rightarrow0+}}\chi(s)=\left|\frac{\alpha-\beta}{\beta}\right|.
  \]
\end{lemma}
\begin{proof}
There are two possibilities: $\alpha>\beta$ or $\alpha<\beta$. In either case, we need Lemma \ref{PositivityLemma} to prove positivity of the expression. For a detailed proof, see \cite{GKM1}.
\hfill \end{proof}

\section{Proof of main results}\label{Section5}
We now prove the main convergence results for the DNWR algorithm applied to the heat and wave equations, stated in Section \ref{Section2} and \ref{Section3}.
\subsection{Proof of Theorem \ref{TheorDNmheat}}
We start by applying the Laplace transform to the homogeneous Dirichlet
subproblem in \eqref{DNmultieq1}, and obtain 
\[
s\hat{u}_{m+1}^k-\nu\hat{u}_{m+1,xx}^k=0,\quad\hat{u}_{m+1}^k(x_{m},s)=\hat{g}_{m}^{k-1}(s),\quad\hat{u}_{m+1}^k(x_{m+1},s)=\hat{g}_{m+1}^{k-1}(s),
\]
Defining $\sigma_{i}:=\sinh\left(h_{i}\sqrt{s/\nu}\right)$ and $\gamma_{i}:=\cosh\left(h_{i}\sqrt{s/\nu}\right)$,
the subproblem \eqref{DNmultieq1} solution becomes 
\[
\arraycolsep0.004em\begin{array}{rcl}
\hat{u}_{m+1}^{k}(x,s)&=&\frac{1}{\sigma_{m+1}}\left(\hat{g}_{m+1}^{k-1}(s)\sinh\left((x\!-\!x_{m})\sqrt{s/\nu}\right)\!+\!\hat{g}_{m}^{k-1}(s)\sinh\left((x_{m+1}\!-\!x)\sqrt{s/\nu}\right)\right).
\end{array}
\]
Similarly the solutions of the subproblems \eqref{DNmultieq2} in Laplace
space are
\[
\arraycolsep0.1em\begin{array}{rcl}
\hat{u}_{i}^{k}(x,s) & = & \frac{1}{\gamma_{i}}\frac{\hat{w}_{i}^{k}}{\sqrt{s/\nu}}\sinh((x-x_{i-1})\sqrt{s/\nu})+\frac{1}{\gamma_{i}}\hat{g}_{i-1}^{k-1}\cosh((x_{i}-x)\sqrt{s/\nu}),\\
\hat{u}_{j}^{k}(x,s) & = & \frac{1}{\gamma_{j}}\hat{g}_{j}^{k-1}\cosh((x-x_{j-1})\sqrt{s/\nu})+\frac{1}{\gamma_{j}}\frac{\hat{w}_{j-1}^{k}}{\sqrt{s/\nu}}\sinh((x_{j}-x)\sqrt{s/\nu}),
\end{array}
\]
for $1\leq i\leq m$ and $m+2\leq j\leq2m+1$.
Therefore for $\theta=1/2$ the update conditions \eqref{DNmultieq3}
become
\begin{equation}
\arraycolsep0.2em\begin{array}{rclrcl}
\hat{g}_{i}^{k} & = & \frac{1}{2\gamma_{i}}\hat{g}_{i-1}^{k-1}+\frac{1}{2}\hat{g}_{i}^{k-1}+\frac{\sigma_{i}}{2\gamma_{i}}\frac{\hat{w}_{i}^{k}}{\sqrt{s/\nu}},\;&
\hat{g}_{j}^{k} & = & \frac{\sigma_{j+1}}{2\gamma_{j+1}}\frac{\hat{w}_{j}^{k}}{\sqrt{s/\nu}}+\frac{1}{2}\hat{g}_{j}^{k-1}+\frac{1}{2\gamma_{j+1}}\hat{g}_{j+1}^{k-1},
\end{array}\label{DNmHeatupdate2}
\end{equation}
for $1\leq i\leq m$, $m+1\leq j\leq2m$ and
\begin{equation}
\arraycolsep0.01em\begin{array}{rcl}
\hat{w}_{i}^{k} & = & -\sqrt{\frac{s}{\nu}}\frac{\sigma_{i+1}}{\gamma_{i+1}}\hat{g}_{i}^{k-1}+\frac{1}{\gamma_{i+1}}\hat{w}_{i+1}^{k},1\leq i\leq m-1;\;\hat{w}_{m}^{k}=-\sqrt{\frac{s}{\nu}}\frac{\gamma_{m+1}}{\sigma_{m+1}}\hat{g}_{m}^{k-1}+\frac{\sqrt{s/\nu}}{\sigma_{m+1}}\hat{g}_{m+1}^{k-1},\\
\hat{w}_{m+1}^{k} & = & \frac{\sqrt{s/\nu}}{\sigma_{m+1}}\hat{g}_{m}^{k-1}-\sqrt{\frac{s}{\nu}}\frac{\gamma_{m+1}}{\sigma_{m+1}}\hat{g}_{m+1}^{k-1};\quad\hat{w}_{j}^{k}=\frac{1}{\gamma_{j}}\hat{w}_{j-1}^{k}-\sqrt{\frac{s}{\nu}}\frac{\sigma_{j}}{\gamma_{j}}\hat{g}_{j}^{k-1},\; m+2\leq j\leq2m.
\end{array}\label{DNmHeatupdate1}
\end{equation}
We choose $\bar{g}_{i}^{k}:=\gamma_{i}\hat{g}_{i}^{k}$, $\bar{w}_{i}^{k}:=\frac{\hat{w}_{i}^{k}}{\sqrt{s/\nu}}\sigma_{i},1\leq i\leq m$
and $\bar{g}_{j}^{k}:=\gamma_{j+1}\hat{g}_{j}^{k},\bar{w}_{j}^{k}:=\frac{\hat{w}_{j}^{k}}{\sqrt{s/\nu}}\sigma_{j+1},m+1\leq j\leq2m$
with $\gamma_{0}=\gamma_{2m+2}=1$ in \eqref{DNmHeatupdate2}-\eqref{DNmHeatupdate1}
to get
\[
\arraycolsep0.007em\begin{array}{rclrcl}
\bar{g}_{i}^{k}&=&\frac{1}{2\gamma_{i-1}}\bar{g}_{i-1}^{k-1}\!+\!\frac{1}{2}\bar{g}_{i}^{k-1}\!+\!\frac{1}{2}\bar{w}_{i}^{k},\,1\leq i\leq m;\;&\bar{g}_{j}^{k}&=&\frac{1}{2}\bar{w}_{j}^{k}\!+\!\frac{1}{2}\bar{g}_{j}^{k-1}\!\!+\!\!\frac{1}{2\gamma_{j+2}}\bar{g}_{j+1}^{k-1},\, m+1\leq j\leq2m,
\end{array}
\]
and for $1\leq i\leq m-1$ and $m+2\leq j\leq2m$,
\[
\arraycolsep0.08em\begin{array}{rcl}
\bar{w}_{i}^{k} & = & -\frac{\sigma_{i}\sigma_{i+1}}{\gamma_{i}\gamma_{i+1}}\bar{g}_{i}^{k-1}\!+\!\frac{\sigma_{i}}{\sigma_{i+1}\gamma_{i+1}}\bar{w}_{i+1}^{k};\;\bar{w}_{m}^{k}=-\frac{\sigma_{m}\gamma_{m+1}}{\gamma_{m}\sigma_{m+1}}\bar{g}_{m}^{k-1}\!+\!\frac{\sigma_{m}}{\sigma_{m+1}\gamma_{m+2}}\bar{g}_{m+1}^{k-1},\\
\bar{w}_{m+1}^{k} & = & \frac{\sigma_{m+2}}{\gamma_{m}\sigma_{m+1}}\bar{g}_{m}^{k-1}-\frac{\gamma_{m+1}\sigma_{m+2}}{\sigma_{m+1}\gamma_{m+2}}\bar{g}_{m+1}^{k-1};\quad\bar{w}_{j}^{k}=\frac{\sigma_{j+1}}{\gamma_{j}\sigma_{j}}\bar{w}_{j-1}^{k}-\frac{\sigma_{j}\sigma_{j+1}}{\gamma_{j}\gamma_{j+1}}\bar{g}_{j}^{k-1}.
\end{array}
\]
Therefore in matrix form, the system becomes
\[
\arraycolsep0.03emA\begin{pmatrix}\bar{g}_{1}^{k}\\
\bar{w}_{1}^{k}\\
\vdots\\
\vdots\\
\bar{g}_{m}^{k}\\
\bar{w}_{m}^{k}\\
\bar{w}_{m+1}^{k}\\
\bar{g}_{m+1}^{k}\\
\vdots\\
\vdots\\
\bar{w}_{2m}^{k}\\
\bar{g}_{2m}^{k}
\end{pmatrix}=B\begin{pmatrix}\bar{g}_{1}^{k-1}\\
\vdots\\
\bar{g}_{m}^{k-1}\\
\bar{g}_{m+1}^{k-1}\\
\vdots\\
\bar{g}_{2m}^{k-1}
\end{pmatrix},\:\textrm{where } A=\begin{pmatrix}L_{1} & \mathbf{0}\\
\mathbf{0} & L_{2}
\end{pmatrix}_{4m\times4m},
\]
with
\[
\arraycolsep0.06emL_{1}=\begin{bmatrix}1 & -\frac{1}{2}\\
 & 1 & 0 & -\frac{\sigma_{1}}{\sigma_{2}\gamma_{2}}\\
 &  & 1 & -\frac{1}{2}\\
 &  &  & 1 & 0 & -\frac{\sigma_{2}}{\sigma_{3}\gamma_{3}}\\
 &  &  &  & \ddots & \ddots\\
 &  &  &  &  & 1 & 0 & -\frac{\sigma_{m-1}}{\sigma_{m}\gamma_{m}}\\
 &  &  &  &  &  & 1 & -\frac{1}{2}\\
 &  &  &  &  &  &  & 1
\end{bmatrix},L_{2}=\begin{bmatrix}1\\
-\frac{1}{2} & 1\\
-\frac{\sigma_{i+1}}{\sigma_{i}\gamma_{i}} & 0 & 1\\
 &  & -\frac{1}{2} & 1\\
 &  & \ddots & \ddots & \ddots\\
 &  &  &  & -\frac{1}{2} & 1\\
 &  &  &  & -\frac{\sigma_{2m+1}}{\sigma_{2m}\gamma_{2m}} & 0 & 1\\
 &  &  &  &  &  & -\frac{1}{2} & 1
\end{bmatrix},
\]
and 
\[
\arraycolsep0.08emB=\left[\begin{array}{cccc|cccc}
\frac{1}{2} &  &  & \\
-\frac{\sigma_{1}\sigma_{2}}{\gamma_{1}\gamma_{2}} &  &  & \\
\frac{1}{2\gamma_{1}} & \frac{1}{2} &  & \\
 & -\frac{\sigma_{2}\sigma_{3}}{\gamma_{2}\gamma_{3}} &  & \\
 & \frac{1}{2\gamma_{2}} & \frac{1}{2} & \\
 &  & \ddots & \\
 &  & \frac{1}{2\gamma_{m-1}} & \frac{1}{2}\\
 &  &  & -\frac{\sigma_{m}\gamma_{m+1}}{\gamma_{m}\sigma_{m+1}} & \frac{\sigma_{m}}{\sigma_{m+1}\gamma_{m+2}}\\
\hline  &  &  & \frac{\sigma_{m+2}}{\sigma_{m+1}\gamma_{m}} & -\frac{\sigma_{m+2}\gamma_{m+1}}{\gamma_{m+2}\sigma_{m+1}}\\
 &  &  &  & \frac{1}{2} & \frac{1}{2\gamma_{m+3}}\\
 &  &  &  &  & -\frac{\sigma_{m+2}\sigma_{m+3}}{\gamma_{m+2}\gamma_{m+3}}\\
 &  &  &  &  & \frac{1}{2} & \frac{1}{2\gamma_{m+4}}\\
 &  &  &  &  &  & \ddots\\
 &  &  &  &  &  & \frac{1}{2} & \frac{1}{2\gamma_{2m+1}}\\
 &  &  &  &  &  &  & -\frac{\sigma_{2m}\sigma_{2m+1}}{\gamma_{2m}\gamma_{2m+1}}\\
 &  &  &  &  &  &  & \frac{1}{2}
\end{array}\right]_{4m\times2m}.
\]
Thus inverting $A$, we obtain
\begin{equation}\label{matrixP}
\begin{pmatrix}
\bar{g}_{1}^{k}\\
\vdots\\
\bar{g}_{m}^{k}\\
\bar{g}_{m+1}^{k}\\
\vdots\\
\bar{g}_{2m}^{k}
\end{pmatrix}=P\begin{pmatrix}\bar{g}_{1}^{k-1}\\
\vdots\\
\bar{g}_{m}^{k-1}\\
\bar{g}_{m+1}^{k-1}\\
\vdots\\
\bar{g}_{2m}^{k-1}
\end{pmatrix},
\end{equation}
where
\[
\arraycolsep0.008emP=\left(\begin{array}{c|c}
K_{1} & K_{2}\\
\hline K_{3} & K_{4}
\end{array}\right)_{2m\times2m},K_{1}=\begin{bmatrix}\frac{\gamma_{1,2}}{2\gamma_{1}\gamma_{2}} & -\frac{\sigma_{1}\sigma_{3}}{2\gamma_{2}^{2}\gamma_{3}} & \cdots & \frac{-\sigma_{1}\sigma_{m}}{2\gamma_{2}\ldots\gamma_{m-2}\gamma_{m-1}^{2}\gamma_{m}} & \frac{-\sigma_{1}\gamma_{m+1}}{2\gamma_{2}\ldots\gamma_{m-1}\gamma_{m}^{2}\sigma_{m+1}}\\
\frac{1}{2\gamma_{1}} & \frac{\gamma_{2,3}}{2\gamma_{2}\gamma_{3}} & \ddots & \frac{-\sigma_{2}\sigma_{m}}{2\gamma_{3}\ldots\gamma_{m-2}\gamma_{m-1}^{2}\gamma_{m}} & \frac{-\sigma_{2}\gamma_{m+1}}{2\gamma_{3}\ldots\gamma_{m-1}\gamma_{m}^{2}\sigma_{m+1}}\\
0 & \ddots & \ddots & \vdots & \vdots\\
\vdots & 0 & \frac{1}{2\gamma_{m-2}} & \frac{\gamma_{m-1,m}}{2\gamma_{m-1}\gamma_{m}} & -\frac{\sigma_{m-1}\gamma_{m+1}}{2\gamma_{m}^{2}\sigma_{m+1}}\\
0 & \cdots & 0 & \frac{1}{2\gamma_{m-1}} & \frac{\sigma_{m+1,m}}{2\sigma_{m+1}\gamma_{m}}
\end{bmatrix},
\]
\[
\arraycolsep0.1emK_{2}=\begin{bmatrix}\frac{\sigma_{1}}{2\gamma_{2}\ldots\gamma_{m}\sigma_{m+1}\gamma_{m+2}} & 0 & \cdots & \cdots & 0\\
\frac{\sigma_{2}}{2\gamma_{3}\ldots\gamma_{m}\sigma_{m+1}\gamma_{m+2}} & \vdots & \vdots & \vdots & \vdots\\
\vdots & \vdots & \vdots & \vdots & \vdots\\
\frac{\sigma_{m-1}}{2\gamma_{m}\sigma_{m+1}\gamma_{m+2}} & \vdots & \vdots & \vdots & \vdots\\
\frac{\sigma_{m}}{2\sigma_{m+1}\gamma_{m+2}} & 0 & \cdots & \cdots & 0
\end{bmatrix},
K_{3}=\begin{bmatrix}0 & \cdots & \cdots & 0 & \frac{\sigma_{m+2}}{2\gamma_{m}\sigma_{m+1}}\\
\vdots & \vdots & \vdots & \vdots & \frac{\sigma_{m+3}}{2\gamma_{m}\sigma_{m+1}\gamma_{m+2}}\\
\vdots & \vdots & \vdots & \vdots & \vdots\\
\vdots & \vdots & \vdots & \vdots & \frac{\sigma_{2m}}{2\gamma_{m}\sigma_{m+1}\gamma_{m+2}\ldots\gamma_{2m-1}}\\
0 & \cdots & \cdots & 0 & \frac{\sigma_{2m+1}}{2\gamma_{m}\sigma_{m+1}\gamma_{m+2}\ldots\gamma_{2m}}
\end{bmatrix},
\]
 \[
\arraycolsep0.08emK_{4}=\begin{bmatrix}\frac{\sigma_{m+1,m+2}}{2\sigma_{m+1}\gamma_{m+2}} & \frac{1}{2\gamma_{m+3}} & 0 & \cdots & 0\\
-\frac{\sigma_{m+3}\gamma_{m+1}}{2\sigma_{m+1}\gamma_{m+2}^{2}} & \frac{\gamma_{m+2,m+3}}{2\gamma_{m+2}\gamma_{m+3}} & \frac{1}{2\gamma_{m+4}} & 0 & \vdots\\
\vdots & \vdots & \ddots & \ddots & 0\\
\frac{-\sigma_{2m}\gamma_{m+1}}{2\sigma_{m+1}\gamma_{m+2}^{2}\gamma_{m+3}\ldots\gamma_{2m-1}} & \frac{-\sigma_{2m}\sigma_{m+2}}{2\gamma_{m+2}\gamma_{m+3}^{2}\gamma_{m+4}\ldots\gamma_{2m-1}} & \ddots & \frac{\gamma_{2m-1,2m}}{2\gamma_{2m-1}\gamma_{2m}} & \frac{1}{2\gamma_{2m+1}}\\
\frac{-\sigma_{2m+1}\gamma_{m+1}}{2\sigma_{m+1}\gamma_{m+2}^{2}\gamma_{m+3}\ldots\gamma_{2m}} & \frac{-\sigma_{2m+1}\sigma_{m+2}}{2\gamma_{m+2}\gamma_{m+3}^{2}\gamma_{m+4}\ldots\gamma_{2m}} & \cdots & -\frac{\sigma_{2m-1}\sigma_{2m+1}}{2\gamma_{2m-1}\gamma_{2m}^{2}} & \frac{\gamma_{2m,2m+1}}{2\gamma_{2m}\gamma_{2m+1}}
\end{bmatrix},
\]
with $\gamma_{i,j}=\cosh\left((h_{i}-h_{j})\sqrt{s/\nu}\right),\sigma_{i,j}=\sinh\left((h_{i}-h_{j})\sqrt{s/\nu}\right)$. Therefore we can write the first equation as:
\begin{equation}
\arraycolsep0.008em\begin{array}{rcl}
2\hat{g}_{1}^{k} & = & \frac{\gamma_{1,2}}{\gamma_{1}\gamma_{2}}\hat{g}_{1}^{k-1}
\!-\!\frac{\sigma_{1}\sigma_{3}}{\gamma_{1}\gamma_{2}\gamma_{3}}\hat{g}_{2}^{k-1}\!-\!\cdots\!-\!\frac{\sigma_{1}\sigma_{m}}{\gamma_{1}\gamma_{2}\ldots\gamma_{m-1}\gamma_{m}}\hat{g}_{m-1}^{k-1}\!-\!\frac{\sigma_{1}\gamma_{m+1}}{\gamma_{1}\ldots\gamma_{m}\sigma_{m+1}}\hat{g}_{m}^{k-1}\!+\!\frac{\sigma_{1}}{\gamma_{1}\ldots\gamma_{m}\sigma_{m+1}}\hat{g}_{m+1}^{k-1}.
\end{array}\label{firsteq}
\end{equation}
So with the substitution $\hat{v}_{i}^{k}:=2^{k}\gamma^{k}\hat{g}_{i}^{k}$,
$\gamma=\cosh\left(h_{\min}\sqrt{s/\nu}\right)$, we get
\[
\arraycolsep0.008em\begin{array}{rcl}
\hat{v}_{1}^{k} & = & \frac{\gamma\gamma_{1,2}}{\gamma_{1}\gamma_{2}}\hat{v}_{1}^{k-1}
\!-\!\frac{\gamma\sigma_{1}\sigma_{3}}{\gamma_{1}\gamma_{2}\gamma_{3}}\hat{v}_{2}^{k-1}\!-\!\cdots\!-\!\frac{\gamma\sigma_{1}\sigma_{m}}{\gamma_{1}\gamma_{2}\ldots\gamma_{m-1}\gamma_{m}}\hat{v}_{m-1}^{k-1}\!-\!\frac{\gamma\sigma_{1}\gamma_{m+1}}{\gamma_{1}\ldots\gamma_{m}\sigma_{m+1}}\hat{v}_{m}^{k-1}\!+\!\frac{\gamma\sigma_{1}}{\gamma_{1}\ldots\gamma_{m}\sigma_{m+1}}\hat{v}_{m+1}^{k-1}\\
 & =: & \displaystyle{\sum_{i=1}^{m+1}}\hat{d}_{1,i}\hat{v}_{i}^{k-1}.
\end{array}
\]
Now using Lemma \ref{PositivityLemma} and \ref{LimitLemma}, we obtain 
\[
\int_{0}^{\infty}|d_{1,1}(t)|dt\le\lim_{s\rightarrow0+}\frac{\cosh(h_{\min}\sqrt{s/\nu})\cosh\left((h_{1}-h_{2})\sqrt{s/\nu}\right)}{\cosh(h_{2}\sqrt{s/\nu})\cosh(h_{1}\sqrt{s/\nu})}\leq1,
\]
so that using Lemma \ref{SimpleLaplaceLemma}, part \ref{L3}, we get 
\[
\left\Vert\mathcal{L}^{-1}\left(\hat{d}_{1,1}\,\hat{v}_{1}^{k-1}(s)\right)\right\Vert_{L^{\infty}(0,T)}\leq\|v_{1}^{k-1}\|_{L^{\infty}(0,T)}.
\]
Moreover, we write for $j=2,\ldots,m-1$ 
\begin{eqnarray*}
\hat{d}_{1,j} \!& = &\! \frac{\gamma}{\gamma_{2}\ldots\gamma_{j}}\left(1-\frac{\sigma_{1}\sigma_{j+1}}{\gamma_{1}\gamma_{j+1}}-1\right)=\frac{\gamma\cosh\left((h_{1}-h_{j+1})\sqrt{s/\nu}\right)}{\gamma_{2}\ldots\gamma_{j}\gamma_{1}\gamma_{j+1}}-\frac{\gamma}{\gamma_{2}\ldots\gamma_{j}}\\
\! & = & \!\frac{\cosh\left((h_{\min}+h_{1}-h_{j+1})\sqrt{s/\nu}\right)+\cosh\left((h_{\min}-h_{1}+h_{j+1})\sqrt{s/\nu}\right)}{2\gamma_{2}\ldots\gamma_{j}\gamma_{1}\gamma_{j+1}}-\frac{\gamma}{\gamma_{2}\ldots\gamma_{j}}.
\end{eqnarray*}
Now $-h_{j+1}\leq h_{\min}+h_{1}-h_{j+1}=h_{1}+h_{\min}-h_{j+1}\leq h_{1}$,
so for the first term, we choose the pairing 
\[
\frac{1}{\gamma_{2}\ldots\gamma_{j}}\cdot\frac{1}{2\cosh(h_{j+1}\sqrt{s/\nu})}\cdot\frac{\cosh\left((h_{\min}+h_{1}-h_{j+1})\sqrt{s/\nu}\right)}{\cosh(h_{1}\sqrt{s/\nu})}
\]
 and 
\[
\frac{1}{\gamma_{2}\ldots\gamma_{j}}\cdot\frac{1}{2\cosh(h_{1}\sqrt{s/\nu})}\cdot\frac{\cosh\left((h_{\min}-h_{1}+h_{j+1})\sqrt{s/\nu}\right)}{\cosh(h_{j+1}\sqrt{s/\nu})}.
\]
 We therefore get by Lemma \ref{PositivityLemma} and \ref{LimitLemma}
\[
\arraycolsep0.1em\begin{array}{rcl}
\int_{0}^{\infty}|d_{1,j}(t)|\, dt & \le & {\displaystyle \lim_{s\rightarrow0+}}\left(\frac{\cosh\left((h_{\min}+h_{1}-h_{j+1})\sqrt{s/\nu}\right)+\cosh\left((h_{\min}-h_{1}+h_{j+1})\sqrt{s/\nu}\right)}{2\gamma_{2}\ldots\gamma_{j}\gamma_{1}\gamma_{j+1}}+\frac{\gamma}{\gamma_{2}\ldots\gamma_{j}}\right)\\
&\leq&2,\end{array}
\]
so that 
\[
\left\Vert \mathcal{L}^{-1}\left(\hat{d}_{1,j}\hat{v}_{j}^{k-1}(s)\right)\right\Vert _{L^{\infty}(0,T)}\leq2\,\|v_{j}^{k-1}\|_{L^{\infty}(0,T)}.
\]
Finally we write $\hat{d}_{1,m}=\frac{\gamma}{\gamma_{2}\ldots\gamma_{m}}\left(1-\frac{\sigma_{1}\gamma_{m+1}}{\gamma_{1}\sigma_{m+1}}-1\right)=-\frac{\gamma}{\gamma_{2}\ldots\gamma_{m}}\cdot K_{1}(s)-\frac{\gamma}{\gamma_{2}\ldots\gamma_{m}}$,
where $K_{1}(s)=\frac{\sinh\left((h_{1}-h_{m+1})\sqrt{s/\nu}\right)}{\sinh(h_{m+1}\sqrt{s/\nu})\cosh(h_{1}\sqrt{s/\nu})}$,
and write
$$\hat{d}_{1,m+1}=\frac{\sigma}{\sigma_{m+1}}\cdot\frac{1}{\gamma_{2}\ldots\gamma_{m}}\left(\frac{\sigma_{1}\gamma}{\gamma_{1}\sigma}-1+1\right)=\frac{\sigma}{\sigma_{m+1}}\cdot\frac{1}{\gamma_{2}\ldots\gamma_{m}}\cdot K_{2}(s)+\frac{\sigma}{\sigma_{m+1}}\cdot\frac{1}{\gamma_{2}\ldots\gamma_{m}},$$ 
where $\sigma=\sinh\left(h_{\min}\sqrt{s/\nu}\right)$, $K_{2}(s)=\frac{\sinh\left((h_{1}-h_{\min})\sqrt{s/\nu}\right)}{\sinh(h_{\min}\sqrt{s/\nu})\cosh(h_{1}\sqrt{s/\nu})}$.
Note that both $K_{1}(s)$ and $K_{2}(s)$ are of the form \eqref{Kernel}.
So by Lemma \ref{KernelPositive}, we obtain the bounds
\[
\int_{0}^{\infty}|d_{1,m}(t)|\, dt\le\lim_{s\rightarrow0+}\frac{\gamma}{\gamma_{2}\ldots\gamma_{m}}\left(K_{1}(s)+1\right)\leq\bigl|\frac{h_{1}-h_{m+1}}{h_{m+1}}\bigr|+1,
\]
and 
\[
\int_{0}^{\infty}|d_{1,m+1}(t)|dt\le\lim_{s\rightarrow0+}\frac{\sigma}{\sigma_{m+1}}\cdot\frac{1}{\gamma_{2}\ldots\gamma_{m}}\left(K_{2}(s)\!+\!1\right)\leq\frac{h_{\min}}{h_{m+1}}\left(\!\frac{h_{1}\!-\!h_{\min}}{h_{\min}}+1\!\right)\!=\!\frac{h_{1}}{h_{m+1}}.
\]
We therefore get
\begin{eqnarray*}
\left\Vert v_{1}^{k}(\cdot)\right\Vert _{L^{\infty}(0,T)} & \leq & \left(1+2(m-2)+\frac{2h_{\max}}{h_{m+1}}\right){\displaystyle \max_{1\leq j\leq m+1}}\parallel v_{j}^{k-1}(\cdot)\parallel_{L^{\infty}(0,T)}\\
 & \leq & \left(2m-3+\frac{2h_{\max}}{h_{m+1}}\right){\displaystyle \max_{1\leq j\leq2m}}\parallel v_{j}^{k-1}(\cdot)\parallel_{L^{\infty}(0,T)}.
\end{eqnarray*}
We also get similar relations for other equations. By induction, we
therefore get for all $i$ 
\begin{eqnarray}\label{bound1}
\left\Vert v_{i}^{k}(\cdot)\right\Vert _{L^{\infty}(0,T)} & \leq & \left(2m-3+\frac{2h_{\max}}{h_{m+1}}\right){\displaystyle \max_{1\leq j\leq2m}}\parallel v_{j}^{k-1}(\cdot)\parallel_{L^{\infty}(0,T)}\nonumber\\
 & \leq & \cdots\leq\left(2m-3+\frac{2h_{\max}}{h_{m+1}}\right)^{k}{\displaystyle \max_{1\leq j\leq2m}}\parallel g_{j}^{0}(\cdot)\parallel_{L^{\infty}(0,T)}.
\end{eqnarray}
Setting $f_{k}(t):=\mathcal{L}^{-1}\left\{\frac{1}{\gamma^{k}}\right\}$
we have
\[
  g_i^{k}(t)=\frac{1}{2^{k}}\left(f_{k}*v_i^{k}\right)(t)=\frac{1}{2^{k}}\int_{0}^{t}f_{k}(t-\tau)v_i^{k}(\tau)d\tau,
\]
from which it follows, using part \ref{L3} of Lemma
\ref{SimpleLaplaceLemma} that
\begin{equation}\label{5.3}
  \| g_i^{k}(\cdot)\|_{L^{\infty}(0,T)}\leq
  \frac{1}{2^{k}}\| v_i^{k}(\cdot)\|_{L^{\infty}(0,T)}\int_{0}^{T}|f_{k}(\tau)|d\tau.
\end{equation}
By Lemma \ref{PositivityLemma}, $f_{k}(t)\geq0$ for all $t$. To obtain a bound for
$\int_0^T f_k(\tau)\,d\tau$, we first show that the function
$r_k(t) = \mathcal{L}^{-1}(2^k e^{-kh_{\min}\sqrt{s/\nu}})$ is greater than or equal to $f_k(t)$ for all $t > 0$,
and then bound $\int_0^T r_k(\tau)\,d\tau$ instead. Indeed, we have
\begin{align*}
\mathcal{L}\left\{r_k(t) - f_k(t)\right\} &=
2^ke^{-kh_{\min}\sqrt{s/\nu}} - \frac{2^k}{(e^{h_{\min}\sqrt{s/\nu}} + e^{-h_{\min}\sqrt{s/\nu}})^k}\\
&= \frac{2^k((1 + e^{-2h_{\min}\sqrt{s/\nu}})^k - 1)}{(e^{h_{\min}\sqrt{s/\nu}}+e^{-h_{\min}\sqrt{s/\nu}})^k}\\
&= \sum_{j=1}^k {k\choose j}e^{-2jh_{\min}\sqrt{s/\nu}}\sech^k(h_{\min}\sqrt{s/\nu}).
\end{align*}
Also, by \cite{Oberhett}
\begin{equation}\label{InversionFormula0}
  \mathcal{L}^{-1}\left( e^{-2jh_{\min}\sqrt{s/\nu}}\right)
    =\frac{jh_{\min}}{\sqrt{\pi\nu t^3}}e^{-j^2h_{\min}^2/\nu t},
\end{equation} is a positive function for $j=1,\ldots, k$. Thus, $\mathcal{L}^{-1}\left(e^{-2jh_{\min}\sqrt{s/\nu}}\sech^k(h_{\min}\sqrt{s/\nu})\right)$ is a convolution of positive functions, and hence positive by part \ref{L1} of Lemma \ref{SimpleLaplaceLemma}. This implies $r_k(t) - f_k(t) \geq 0$, so we deduce that 
\begin{equation}\label{DNerfc}
\int_{0}^{T}f_{k}(\tau)\,d\tau \leq \int_0^T r_k(\tau)\,d \tau
= \mathcal{L}^{-1}\left(\frac{2^{k}e^{-kh_{\min}\sqrt{s/\nu}}}{s}\right) = 2^k\, \textrm{erfc}
  \left(\frac{kh_{\min}}{2\sqrt{\nu T}}\right),
\end{equation}
where we expressed the second integral as an inverse Laplace transform using
Lemma \ref{SimpleLaplaceLemma}, part \ref{L4}, which we then evaluated using
the following identity from \cite{Oberhett}:
\begin{equation}\label{InversionFormula}
  \mathcal{L}^{-1}\left( \frac{1}{s}e^{-\lambda\sqrt{s}}\right)
    =\textrm{erfc}\left(\frac{\lambda}{2\sqrt{t}}\right),
   \quad\lambda > 0.
\end{equation}
Finally, we combine the above bound \eqref{DNerfc} with \eqref{bound1} and \eqref{5.3} to conclude the proof of the theorem.
\begin{remark}\label{Rem3}
For equal subdomains, with subdomain size $h$, we can deduce a better estimate:
\[
{\displaystyle \max_{1\leq i\leq2m}}\parallel g_{i}^{k}\parallel_{L^{\infty}(0,T)}\leq\left(\min\left\{ 2m-1,Q(h,\nu,T)\right\} \right)^{k}{\rm erfc}\left(\frac{kh}{2\sqrt{\nu T}}\right){\displaystyle \max_{1\leq i\leq2m}}\parallel g_{i}^{0}\parallel_{L^{\infty}(0,T)},
\]
where $Q(h,\nu,T):=2\,{\rm erfc}\left(\frac{h}{2\sqrt{\nu T}}\right)+\sum_{i=0}^{\infty}2^{i+1}{\rm erfc}\left(\frac{ih}{2\sqrt{\nu T}}\right)$.

Here is the outline of the proof: For equal-length subdomains we have $h_{i}=h$ for $i=1,\ldots,2m+1$, so that $\gamma_{i}=\gamma,\sigma_{i}=\sigma$ and $\gamma_{i,j}=1,\sigma_{i,j}=0$ for all $i,j$ in the matrix $P$ in \eqref{matrixP}. Therefore, the first update equation \eqref{firsteq} becomes
\begin{equation}
\hat{g}_{1}^{k}=\frac{1}{2\gamma^{2}}\hat{g}_{1}^{k-1}-\frac{\sigma^{2}}{2\gamma^{3}}\hat{g}_{2}^{k-1}-\cdots-\frac{\sigma^{2}}{2\gamma^{m}}\hat{g}_{m-1}^{k-1}-\frac{1}{2\gamma^{m-1}}\hat{g}_{m}^{k-1}+\frac{1}{2\gamma^{m}}\hat{g}_{m+1}^{k-1},\; m\geq2,
\end{equation}
which in turn becomes after setting $\hat{v}_{i}^{k}(s)=2^{k}\gamma^{k}\hat{g}_{i}^{k}(s)$, 
\begin{eqnarray}\label{2ndeq}
\hat{v}_{1}^{k} \!\!& = & \!\!\frac{1}{\gamma}\hat{v}_{1}^{k-1}-\left(1-\frac{1}{\gamma^{2}}\right)\hat{v}_{2}^{k-1}-\cdots-\frac{1}{\gamma^{m-3}}\left(1-\frac{1}{\gamma^{2}}\right)\hat{v}_{m-1}^{k-1}-\frac{1}{\gamma^{m-2}}\hat{v}_{m}^{k-1}+\frac{1}{\gamma^{m-1}}\hat{v}_{m+1}^{k-1}\nonumber \\
 \!\!& =: &\!\! \sum_{i=1}^{m+1}\hat{r}_{i}\hat{v}_{i}^{k-1}.
\end{eqnarray}
Note that the estimate from Theorem \ref{TheorDNmheat} also holds in this case with $h_{\max}=h_{m+1}$:
\begin{equation}\label{DNmestimate1}
{\displaystyle \max_{1\leq i\leq2m}}\parallel g_{i}^{k}\parallel_{L^{\infty}(0,T)}\leq\left(2m-1\right)^{k}{\rm erfc}\left(\frac{kh}{2\sqrt{\nu T}}\right){\displaystyle \max_{1\leq i\leq2m}}\parallel g_{i}^{0}\parallel_{L^{\infty}(0,T)}.
\end{equation}

We now present a different estimate starting from \eqref{2ndeq}.
By back transforming into the time domain we obtain $v_{1}^{k}(t)=\sum_{i=1}^{m+1}\left(r_{i}*v_{i}^{k-1}\right)(t).$
Now part \ref{L3} of Lemma \ref{SimpleLaplaceLemma} yields 
\begin{eqnarray}\label{3rdeq}
\|v_{1}^{k}\|_{L^{\infty}(0,T)}&\leq &\sum_{i=1}^{m+1}\|v_{i}^{k-1}\|_{L^{\infty}(0,T)}\int_{0}^{T}\!|r_{i}(\tau)|d\tau\nonumber\\
&\leq &{\displaystyle \max_{1\leq i\leq m+1}}\parallel v_{i}^{k-1}\parallel_{L^{\infty}(0,T)}\sum_{i=1}^{m+1}\int_{0}^{T}\!|r_{i}(\tau)|d\tau.
\end{eqnarray}
Set $f_{k}(t):=\mathcal{L}^{-1}\left(\frac{1}{\gamma^{k}}\right)$. Since 
$\lim_{s\rightarrow0+}\left|1-\frac{1}{\gamma^{2}}\right|\leq2$,
we have from \eqref{3rdeq}
\[
\sum_{i=1}^{m+1}\int_{0}^{T}|r_{i}(\tau)|d\tau\leq\int_{0}^{T}\left(f_{1}(t)+2+2f_{1}(t)+\cdots+2f_{m-3}(t)+f_{m-2}(t)+f_{m-1}(t)\right)dt.
\]
Therefore using the inequality \eqref{DNerfc},  
$\int_{0}^{T}f_{k}(t)dt\leq2^{k}{\rm erfc}\left(\frac{kh}{2\sqrt{\nu T}}\right)$, the above expression is bounded by
\[
\sum_{i=1}^{m+1}\int_{0}^{T}|r_{i}(\tau)|d\tau\leq2\,{\rm erfc}\left(\frac{h}{2\sqrt{\nu T}}\right)+\sum_{i=0}^{m-1}2^{i+1}{\rm erfc}\left(\frac{ih}{2\sqrt{\nu T}}\right)\leq Q(h,\nu,T),
\]
where $Q(h,\nu,T):=2\,{\rm erfc}\left(\frac{h}{2\sqrt{\nu T}}\right)+\sum_{i=0}^{\infty}2^{i+1}{\rm erfc}\left(\frac{ih}{2\sqrt{\nu T}}\right)$.
So we get the second estimate as 
\begin{equation}\label{DNmestimate2}
{\displaystyle \max_{1\leq i\leq2m}}\parallel g_{i}^{k}\parallel_{L^{\infty}(0,T)}\leq Q^{k}\,{\rm erfc}\left(\frac{kh}{2\sqrt{\nu T}}\right){\displaystyle \max_{1\leq i\leq2m}}\parallel g_{i}^{0}\parallel_{L^{\infty}(0,T)}.
\end{equation}
The result follows combining the two estimates \eqref{DNmestimate1}
and \eqref{DNmestimate2}. 

We compare the two estimates \eqref{DNmestimate1} and \eqref{DNmestimate2}
in Figure \ref{FigregionDNm} for $\nu=1$. The region below the red curve is where the estimate \eqref{DNmestimate1} is more accurate than \eqref{DNmestimate2}.
\begin{figure}
\centering
\includegraphics[width=8.5cm,height=6cm]{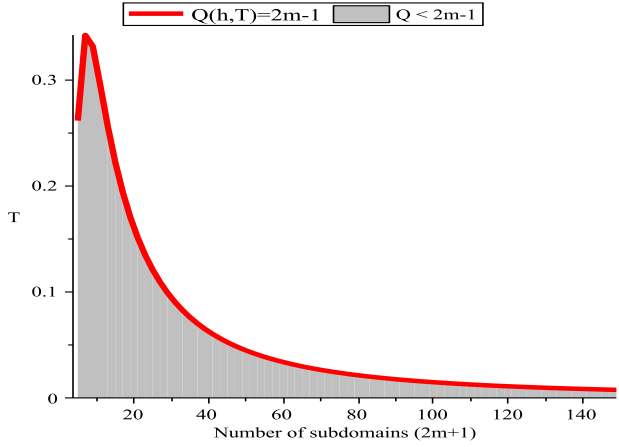}
\caption{Comparison of the two estimates (\ref{DNmestimate1}) and (\ref{DNmestimate2}) for $\nu=1$}
\label{FigregionDNm}
\end{figure}
\end{remark}

\subsection{Proof of Theorem \ref{DNwavemulti}}

We start by applying the Laplace transform to the homogeneous Dirichlet subproblem in \eqref{DNmwaveA3-1} to get 
\[
s^{2}\hat{u}_{m+1}^{k}-c^{2}\hat{u}_{m+1,xx}^{k}=0,\quad\hat{u}_{m+1}^{k}(x_{m},s)=\hat{g}_{m}^{k-1}(s),\quad\hat{u}_{m+1}^{k}(x_{m+1},s)=\hat{g}_{m+1}^{k-1}(s),
\]
Define $\rho_{i}:=\sinh\left(h_{i}s/c\right)$ and $\lambda_{i}:=\cosh\left(h_{i}s/c\right)$.
Then the subproblem \eqref{DNmwaveA3-1} solution becomes 
\[
\hat{u}_{m+1}^{k}(x,s)=\frac{1}{\rho_{m+1}}\left(\hat{g}_{m+1}^{k-1}(s)\sinh\left((x-x_{m})s/c\right)+\hat{g}_{m}^{k-1}(s)\sinh\left((x_{m+1}-x)s/c\right)\right).
\]
The solutions of the subproblems \eqref{DNmwaveA3-2} in Laplace
space are
\[
\arraycolsep0.01em\begin{array}{rcl}
\hat{u}_{i}^{k}(x,s) & = & \frac{1}{\lambda_{i}}\frac{\hat{w}_{i}^{k}}{s/c}\sinh((x-x_{i-1})s/c)+\frac{1}{\lambda_{i}}\hat{g}_{i-1}^{k-1}\cosh((x_{i}-x)s/c),\;1\leq i\leq m,\\
\hat{u}_{j}^{k}(x,s) & = & \frac{1}{\lambda_{j}}\hat{g}_{j}^{k-1}\cosh((x-x_{j-1})s/c)+\frac{1}{\lambda_{j}}\frac{\hat{w}_{j-1}^{k}}{s/c}\sinh((x_{j}-x)s/c),\;m+2\leq j\leq2m+1.
\end{array}
\]
Therefore for $\theta=1/2$ the update conditions \eqref{DNmwaveA3-3}
become
\begin{equation}
\arraycolsep0.1em\begin{array}{rcl}
\hat{g}_{i}^{k} & = & \frac{1}{2\lambda_{i}}\hat{g}_{i-1}^{k-1}+\frac{1}{2}\hat{g}_{i}^{k-1}+\frac{\rho_{i}}{2\lambda_{i}}\frac{\hat{w}_{i}^{k}}{s/c},\;1\leq i\leq m,\\
\hat{g}_{j}^{k} & = & \frac{\rho_{j+1}}{2\lambda_{j+1}}\frac{\hat{w}_{j}^{k}}{s/c}+\frac{1}{2}\hat{g}_{j}^{k-1}+\frac{1}{2\lambda_{j+1}}\hat{g}_{j+1}^{k-1},\;m+1\leq j\leq2m.
\end{array}\label{DNmupdate2}
\end{equation}
and
\begin{equation}
\arraycolsep0.08em\begin{array}{rcl}
\hat{w}_{i}^{k} & = & -\frac{s}{c}\frac{\rho_{i+1}}{\lambda_{i+1}}\hat{g}_{i}^{k-1}+\frac{1}{\lambda_{i+1}}\hat{w}_{i+1}^{k},1\leq i\leq m-1;\;\hat{w}_{m}^{k}=-\frac{s}{c}\frac{\lambda_{m+1}}{\rho_{m+1}}\hat{g}_{m}^{k-1}+\frac{s/c}{\rho_{m+1}}\hat{g}_{m+1}^{k-1},\\
\hat{w}_{m+1}^{k} & = & \frac{s/c}{\rho_{m+1}}\hat{g}_{m}^{k-1}-\frac{s}{c}\frac{\lambda_{m+1}}{\rho_{m+1}}\hat{g}_{m+1}^{k-1};\quad\hat{w}_{j}^{k}=\frac{1}{\lambda_{j}}\hat{w}_{j-1}^{k}-\frac{s}{c}\frac{\rho_{j}}{\lambda_{j}}\hat{g}_{j}^{k-1},\;m+2\leq j\leq2m,
\end{array}\label{DNmupdate1}
\end{equation}
We choose $\bar{g}_{i}^{k}:=\lambda_{i}\hat{g}_{i}^{k},\bar{w}_{i}^{k}:=\frac{\hat{w}_{i}^{k}}{s/c}\rho_{i},1\leq i\leq m$
and $\bar{g}_{j}^{k}:=\lambda_{j+1}\hat{g}_{j}^{k},\bar{w}_{j}^{k}:=\frac{\hat{w}_{j}^{k}}{s/c}\rho_{j+1},m+1\leq j\leq2m$
with $\lambda_{0}=\lambda_{2m+2}=1$ in the corresponding equations
of \eqref{DNmupdate2}-\eqref{DNmupdate1} to get
\[
\arraycolsep0.1em
\begin{array}{rcl}
{g}_{i}^{k}&=&\frac{1}{2\lambda_{i-1}}\bar{g}_{i-1}^{k-1}+\frac{1}{2}\bar{g}_{i}^{k-1}+\frac{1}{2}\bar{w}_{i}^{k},\;1\leq i\leq m;\\
\bar{g}_{j}^{k}&=&\frac{1}{2}\bar{w}_{j}^{k}+\frac{1}{2}\bar{g}_{j}^{k-1}+\frac{1}{2\lambda_{j+2}}\bar{g}_{j+1}^{k-1},\;m+1\leq j\leq2m,
\end{array}
\]
and for $1\leq i\leq m-1$ and $m+2\leq j\leq2m$,
\[
\arraycolsep0.08em\begin{array}{rcl}
\bar{w}_{i}^{k} & = & -\frac{\rho_{i}\rho_{i+1}}{\lambda_{i}\lambda_{i+1}}\bar{g}_{i}^{k-1}+\frac{\rho_{i}}{\rho_{i+1}\lambda_{i+1}}\bar{w}_{i+1}^{k};\;\bar{w}_{m}^{k}=-\frac{\rho_{m}\lambda_{m+1}}{\lambda_{m}\rho_{m+1}}\bar{g}_{m}^{k-1}+\frac{\rho_{m}}{\rho_{m+1}\lambda_{m+2}}\bar{g}_{m+1}^{k-1},\\
\bar{w}_{m+1}^{k} & = & \frac{\rho_{m+2}}{\lambda_{m}\rho_{m+1}}\bar{g}_{m}^{k-1}-\frac{\lambda_{m+1}\rho_{m+2}}{\rho_{m+1}\lambda_{m+2}}\bar{g}_{m+1}^{k-1};\quad\bar{w}_{j}^{k}=\frac{\rho_{j+1}}{\lambda_{j}\rho_{j}}\bar{w}_{j-1}^{k}-\frac{\rho_{j}\rho_{j+1}}{\lambda_{j}\lambda_{j+1}}\bar{g}_{j}^{k-1}.
\end{array}
\]
Therefore we can write the system in matrix form as in the Heat equation case 
\[
\begin{pmatrix}\bar{g}_{1}^{k}\\
\vdots\\
\bar{g}_{m}^{k}\\
\bar{g}_{m+1}^{k}\\
\vdots\\
\bar{g}_{2m}^{k}
\end{pmatrix}=P\begin{pmatrix}\bar{g}_{1}^{k-1}\\
\vdots\\
\bar{g}_{m}^{k-1}\\
\bar{g}_{m+1}^{k-1}\\
\vdots\\
\bar{g}_{2m}^{k-1}
\end{pmatrix},
\]
where $P$ is in the same form as in \eqref{matrixP} with $\gamma_{i,j}$ and $\sigma_{i,j}$ replaced by $\lambda_{i,j}:=\cosh\left((h_{i}-h_{j})s/c\right),\rho_{i,j}:=\sinh\left((h_{i}-h_{j})s/c\right)$ respectively.
Therefore the updating conditions become
\begin{equation}
\hat{g}_{i}^{k}(s)={\displaystyle \sum_{l=i-1}^{m+1}}\hat{k}_{i,l}\hat{g}_{l}^{k-1}(s),\:1\leq i\leq m;\quad\hat{g}_{j}^{k}(s)={\displaystyle \sum_{l=m}^{j+1}}\hat{k}_{j,l}\hat{g}_{l}^{k-1}(s),\:m+1\leq j\leq2m,\label{DNmwaveupdate}
\end{equation}
where $\hat{k}_{1,0}=\hat{k}_{2m,2m+1}=0$, and for $i+1\leq l<m$
$\hat{k}_{i,i-1}=\frac{1}{2\lambda_{i}},\hat{k}_{i,i}=\frac{\lambda_{i,i+1}}{2\lambda_{i}\lambda_{i+1}},\hat{k}_{i,l}=-\frac{\rho_{i}\rho_{l+1}}{2\lambda_{i}\lambda_{i+1}\ldots\lambda_{l+1}},$
$\hat{k}_{i,m}=-\frac{\rho_{i}\lambda_{m+1}}{2\lambda_{i}\ldots\lambda_{m}\rho_{m+1}},\hat{k}_{i,m+1}=\frac{\rho_{i}}{2\lambda_{i}\ldots\lambda_{m}\rho_{m+1}}$
for $1\leq i<m,$ and for $m+1<l\leq j-1$ $\hat{k}_{j,j}=\frac{\lambda_{j,j+1}}{2\lambda_{j}\lambda_{j+1}},\hat{k}_{j,l}=-\frac{\rho_{j+1}\rho_{l}}{2\lambda_{l}\lambda_{l+1}\ldots\lambda_{j+1}},\hat{k}_{j,j+1}=\frac{1}{2\lambda_{j+1}},$
$\hat{k}_{j,m+1}=-\frac{\rho_{j+1}\lambda_{m+1}}{2\rho_{m+1}\lambda_{m+2}\ldots\lambda_{j+1}},\hat{k}_{j,m}=\frac{\rho_{j+1}}{2\rho_{m+1}\lambda_{m+2}\ldots\lambda_{j+1}}$
for $m+1<j\leq2m.$ Also, $\hat{k}_{m,m-1}=\frac{1}{2\lambda_{m}},\hat{k}_{m,m}=\frac{\rho_{m+1,m}}{2\rho_{m+1}\lambda_{m}},\hat{k}_{m,m+1}=\frac{\rho_{m}}{2\rho_{m+1}\lambda_{m}}$
and $\hat{k}_{m+1,m}=\frac{\rho_{m+2}}{2\rho_{m+1}\lambda_{m+2}},\hat{k}_{m+1,m+1}=\frac{\rho_{m+1,m+2}}{2\rho_{m+1}\lambda_{m+2}},\hat{k}_{m+1,m+2}=\frac{1}{2\lambda_{m+2}}.$
So by induction on \eqref{DNmwaveupdate} we can write for $1\leq i\leq2m$
\begin{equation}
\hat{g}_{i}^{k}(s)=\sum_{j=1}^{2m}p_{i,j}^{n}\left(\hat{k}_{1,1},\hat{k}_{1,2},\ldots,\hat{k}_{2m,2m-1},\hat{k}_{2m,2m}\right)\,\hat{g}_{j}^{k-n}(s),\label{DNminduct1}
\end{equation}
where the coefficients $p_{i,j}^{n}$ are either zero or homogeneous
polynomials of degree $n$. Now expanding hyperbolic functions into
infinite binomial series, we obtain for $i+1\leq l<m$ and $1\leq i<m$
\begin{multline*}
\hat{k}_{i,i}=\frac{\cosh\left((h_{i}-h_{i+1})s/c\right)}{2\cosh(h_{i}s/c)\cosh(h_{i+1}s/c)}=\left(e^{-2h_{i}s/c}+e^{-2h_{i+1}s/c}\right)\left[1+{\displaystyle \sum_{l=1}^{\infty}}(-1)^{l}e^{-2h_{i}ls/c}\right.\\
\left.+{\displaystyle \sum_{n=1}^{\infty}}(-1)^{n}e^{-2h_{i+1}ns/c}+{\displaystyle \sum_{l=1}^{\infty}}{\displaystyle \sum_{n=1}^{\infty}}(-1)^{l+n}e^{-2(lh_{i}+nh_{i+1})s/c}\right],
\end{multline*}
\begin{multline*}
\hat{k}_{i,l}=-\frac{\sinh\left(h_{i}s/c\right)\sinh\left(h_{l+1}s/c\right)}{2\cosh(h_{i}s/c)\cosh(h_{i+1}s/c)\ldots\cosh(h_{l+1}s/c)}=-2^{l-i-1}e^{-\left(h_{i+1}+\cdots+h_{l}\right)s/c}\left(1-\right.\\
\left.e^{-2h_{i}s/c}-e^{-2h_{l+1}s/c}+e^{-2(h_{i}+h_{l+1})s/c}\right){\displaystyle \prod_{n=i}^{l+1}}\left(1+e^{-2h_{n}s/c}\right)^{-1},
\end{multline*}
 
\[
\hat{k}_{i,i-1}=\frac{1}{2\cosh(h_{i}s/c)}=e^{-h_{i}s/c}\left[1+{\displaystyle \sum_{l=1}^{\infty}}(-1)^{l}e^{-2h_{i}ls/c}\right],
\]
\begin{multline*}
\hat{k}_{i,m}=-\frac{\sinh\left(h_{i}s/c\right)\cosh\left(h_{m+1}s/c\right)}{2\cosh(h_{i}s/c)\ldots\cosh(h_{m}s/c)\sinh(h_{m+1}s/c)}\\
=-2^{m-i-1}e^{-\left(h_{i+1}+\cdots+h_{m}\right)s/c}\left(1-e^{-2h_{i}s/c}
+ e^{-2h_{m+1}s/c}\right.\\
\left.- e^{-2(h_{i}+h_{m+1})s/c}\right) 
\left(1-e^{-2h_{m+1}s/c}\right)^{-1}{\displaystyle \prod_{l=i}^{m}}\left(1+e^{-2h_{l}s/c}\right)^{-1},
\end{multline*}
\begin{multline*}
\hat{k}_{i,m+1}=\frac{\sinh\left(h_{i}s/c\right)}{2\cosh(h_{i}s/c)\ldots\cosh(h_{m}s/c)\sinh(h_{m+1}s/c)}\\
=2^{m-i}e^{-\left(h_{i+1}+\cdots+h_{m+1}\right)s/c}\left(1-e^{-2h_{i}s/c}\right)
\left(1-e^{-2h_{m+1}s/c}\right)^{-1}{\displaystyle \prod_{l=i}^{m}}\left(1+e^{-2h_{l}s/c}\right)^{-1},
\end{multline*}
\begin{multline*}
\hat{k}_{m,m}=\frac{\sinh\left((h_{m+1}-h_{m})s/c\right)}{2\cosh(h_{m}s/c)\sinh(h_{m+1}s/c)}=\left(e^{-2h_{m}s/c}-e^{-2h_{m+1}s/c}\right)\left[1\right.\\
\left.+ {\displaystyle \sum_{l=1}^{\infty}}(-1)^{l}e^{-2h_{m}ls/c}
+{\displaystyle \sum_{n=1}^{\infty}}e^{-2h_{m+1}ns/c}+{\displaystyle \sum_{l=1}^{\infty}}{\displaystyle \sum_{n=1}^{\infty}}(-1)^{l}e^{-2(lh_{m}+nh_{m+1})s/c}\right],
\end{multline*}
\begin{multline*}
\hat{k}_{m,m+1}=\frac{\sinh\left(h_{m}s/c\right)}{2\cosh(h_{m}s/c)\sinh(h_{m+1}s/c)}=\left(e^{-h_{m+1}s/c}-e^{-(2h_{m}+h_{m+1})s/c}\right)\left[1\right.\\
\left.+{\displaystyle \sum_{l=1}^{\infty}}(-1)^{l}e^{-2h_{m}ls/c}+{\displaystyle \sum_{n=1}^{\infty}}e^{-2h_{m+1}ns/c}+{\displaystyle \sum_{l=1}^{\infty}}{\displaystyle \sum_{n=1}^{\infty}}(-1)^{l}e^{-2(lh_{m}+nh_{m+1})s/c}\right].
\end{multline*}
The argument also holds similarly for other terms. Now using these
expressions we can write \eqref{DNminduct1} as 
\begin{equation}
\hat{g}_{i}^{k}(s)=\sum_{j=1}^{2m}\hat{r}_{i,j}^{k}(s)\,\hat{g}_{j}^{0}(s),\label{DNmfinalind}
\end{equation}
where $\hat{r}_{i,j}^{k}(s)$ are linear combinations of terms of the form
$e^{-sz}$ with $z\geq kh_{l}/c$ for some $l\in\left\{ 1,2,\ldots2m+1\right\} $.
We now recall the shifting property of Laplace transform: 
\begin{equation}
\mathcal{L}^{-1}\left\{ e^{-\alpha s}\hat{f}(s)\right\} =H(t-\alpha)f(t-\alpha),\label{Lformula}
\end{equation}
where $H(t):=\begin{cases}
1, & t>0,\\
0, & t\leq0.
\end{cases}$ is the Heaviside step function. We use \eqref{Lformula} to back
transform \eqref{DNmfinalind} and obtain
\[
g_{i}^{k}(t)=g_{j}^{0}\left(t-\frac{kh_{l}}{c}\right)H\left(t-\frac{kh_{l}}{c}\right)+\:\textrm{other terms},
\]
for some $j\in\left\{ 1,2,\ldots2m\right\} $ and $l\in\left\{ 1,2,\ldots2m+1\right\} $.
Thus for $T\leq kh_{\min}/c$, we get $g_{i}^{k}(t)=0$ for all $i$,
and the conclusion follows.

\begin{remark}\label{Rem4}
 The shifting property of Laplace transform \eqref{Lformula} is the reason behind the finite step convergence of the DNWR for $\theta=1/2$. 
The right hand side of \eqref{Lformula} becomes identically zero for $t\leq\alpha$, so that for sufficiently small time window length $T$ (e.g., $T\leq\alpha$) the
error becomes zero and leads to convergence in the next iteration.
In Figure \ref{FigDNkernel5} we plot $\mathcal{L}^{-1}\left\{ \hat{f}(s)\right\} $
with $f(t)=\sin(t)$ on the left, and show the effect of time-shifting
on the right.
\begin{figure}
\centering
\includegraphics[width=6cm,height=4cm]{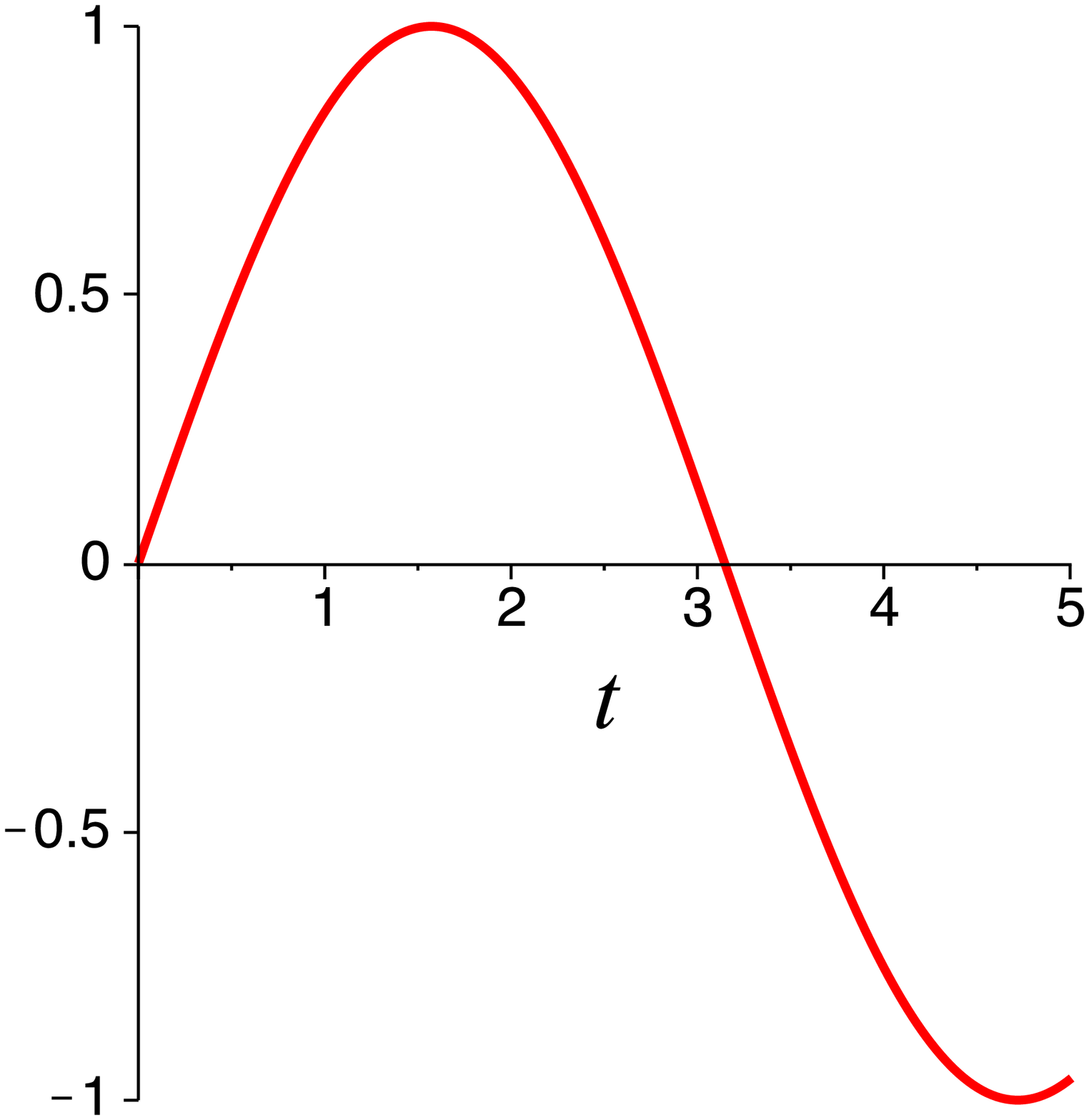}
\includegraphics[width=6cm,height=4cm]{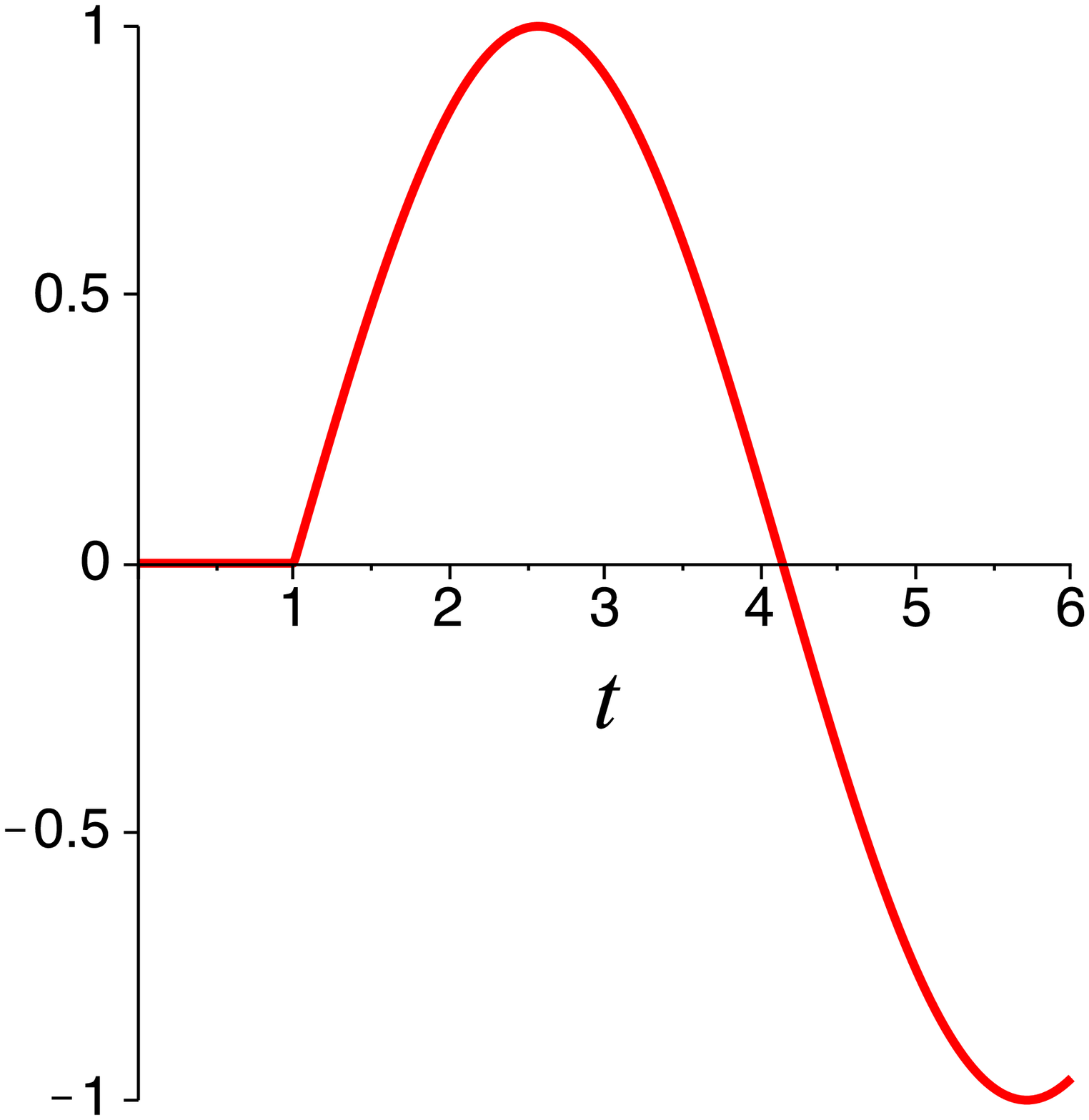}
\caption{Example of time-shifting for the function $f(t)=\sin(t)$: $\mathcal{L}^{-1}\left\{ \hat{f}(s)\right\}$ on the left, and $\mathcal{L}^{-1}\left\{ e^{-s}\hat{f}(s)\right\}$ on the right}
\label{FigDNkernel5}
\end{figure}
\end{remark}

\section{Analysis of DNWR algorithm for Wave equation in 2D}\label{Section6}

We now formulate and analyze the DNWR algorithm for the two-dimensional
wave equation 
\[
\partial_{tt}u-c^{2}\Delta u=f(x,y,t),\quad(x,y)\in\Omega=(l,L)\times(0,\pi),\:t\in(0,T]
\]
with initial condition $u(x,y,0)=v_{0}(x,y),\partial_{t}u(x,y,0)=w_{0}(x,y)$
and Dirichlet boundary conditions $u(x,y,t)=g(x,y,t), (x,y)\in \partial \Omega$. To define the DNWR algorithm,
we decompose $\Omega$ into strips of the form $\Omega_{i}=(x_{i-1},x_{i})\times(0,\pi)$,
$l=x_{0}<x_{1}<\cdots<x_{2m+1}=L$. We define the subdomain width $h_{i}:=x_{i}-x_{i-1}$,
and $h_{\min}:=\min_{1\leq i\leq 2m+1}h_{i}$. Also we directly consider
the error equations with $f(x,y,t)=0,v_{0}(x,y)=0=w_{0}(x,y)$ and
homogeneous Dirichlet boundary conditions. Given initial guesses $\left\{ g_{i}^{0}(y,t)\right\} _{i=1}^{2m}$
along the interface $\left\{ x=x_{i}\right\}$, the DNWR algorithm,
as a particular case of \eqref{DNmwave1}-\eqref{DNmwave2}-\eqref{DNmwave3},
is given by performing iteratively for $k=1,2,\ldots$
\begin{equation}
\begin{array}{rcll}
\partial_{tt}u_{m+1}^{k}-c^{2}\Delta u_{m+1}^{k} & = & 0, & \textrm{in}\;\Omega_{m+1},\\
u_{m+1}^{k}(x,y,0) & = & 0, & \textrm{in}\;\Omega_{m+1},\\
\partial_{t}u_{m+1}^{k}(x,y,0) & = & 0, & \textrm{in}\;\Omega_{m+1},\\
u_{m+1}^{k}(x_{m},y,t) & = & g_{m}^{k-1}(y,t), &\\
u_{m+1}^{k}(x_{m+1},y,t) & = & g_{m+1}^{k-1}(y,t), & \\
u_{m+1}^{k}(x,0,t) & = & u_{m+1}^{k}(x,\pi,t)=0,&
\end{array}\label{DN2dwave1}
\end{equation}
and then denoting by $\left\{ w_{i}^{k}(y,t)\right\} _{i=1}^{2m}$ the Neumann traces along the interfaces, calculate for $m\geq i\geq1$ and $m+2\leq j\leq2m+1$
\begin{equation}
\arraycolsep0.08em\begin{array}{rcll}
\partial_{tt}u_{i}^{k} -c^{2}\Delta u_{i}^{k}& = & 0, & \textrm{in}\;\Omega_{i},\\
u_{i}^{k}(x,y,0) & = & 0, & \textrm{in}\;\Omega_{i},\\
\partial_{t}u_{i}^{k}(x,y,0) & = & 0, & \textrm{in}\;\Omega_{i},\\
u_{i}^{k}(x_{i-1},y,t) & = & g_{i-1}^{k-1}(y,t), & \\
\partial_{x}u_{i}^{k}(x_{i},y,t) & = & w_{i}^{k}(y,t), & \\
u_{i}^{k}(x,0,t) & = & u_{i}^{k}(x,\pi,t)=0,&
\end{array}\;\begin{array}{rcll}
\partial_{tt}u_{j}^{k} -c^{2}\Delta u_{j}^{k}& = & 0, & \textrm{in}\;\Omega_{j},\\
u_{j}^{k}(x,y,0) & = & 0, & \textrm{in}\;\Omega_{j},\\
\partial_{t}u_{j}^{k}(x,y,0) & = & 0, & \textrm{in}\;\Omega_{j},\\
-\partial_{x}u_{j}^{k}(x_{j-1},y,t) & = & w_{j-1}^{k}(y,t), & \\
u_{j}^{k}(x_{j},y,t) & = & g_{j}^{k-1}(y,t), &\\
u_{j}^{k}(x,0,t) & = & u_{j}^{k}(x,\pi,t)=0,& 
\end{array}\label{DN2dwave2}
\end{equation}
with the update conditions for $1\leq i\leq m$ and $m+1\leq j\leq2m$
\begin{equation}
\arraycolsep0.1em\begin{array}{rclrcl}
g_{i}^{k}(y,t) & = & \theta u_{i}^{k}(x_{i},y,t)+(1-\theta)g_{i}^{k-1}(y,t), & w_{i}^{k}(y,t) & = & \partial_{x}u_{i+1}^{k}(x_{i},y,t),\\
g_{j}^{k}(y,t) & = & \theta u_{j+1}^{k}(x_{j},y,t)+(1-\theta)g_{j}^{k-1}(y,t), & w_{j}^{k}(y,t) & = & -\partial_{x}u_{j}^{k}(x_{j},y,t),
\end{array}\label{DN2dwave3}
\end{equation}
where $\theta\in(0,1]$.

We perform a Fourier transform along the $y$ direction to reduce
the original problem into a collection of one-dimensional problems.
Using a Fourier sine series along the $y$-direction, we get 
\[
u_{i}^{k}(x,y,t)=\sum_{n\geq1}U_{i}^{k}(x,n,t)\sin(ny)
\]
where 
\[
U_{i}^{k}(x,n,t)=\frac{2}{\pi}\int_{0}^{\pi}u_{i}^{k}(x,\eta,t)\sin(n\eta)d\eta.
\]
The equation \eqref{DN2dwave1} therefore becomes a sequence
of 1D equations for each $n$,
\begin{equation}
\frac{\partial^{2}U_{m+1}^{k}}{\partial t^{2}}(x,n,t)-c^{2}\frac{\partial^{2}U_{m+1}^{k}}{\partial x^{2}}(x,n,t)+c^{2}n^{2}U_{m+1}^{k}(x,n,t)=0,\label{DNwave2dlap}
\end{equation}
with the boundary conditions for $U_{m+1}^{k}(x,n,t)$. We now define
\begin{equation}
\chi(\alpha,\beta,t):=\mathcal{L}^{-1}\left\{ \exp\left(-\beta\sqrt{s^{2}+\alpha^{2}}\right)\right\} ,\quad\textrm{Re}(s)>0,\label{Xiab}
\end{equation}
with $s$ being the Laplace variable. Before presenting the main convergence
theorem, we prove the following auxiliary result.
\begin{lemma}\label{Lemma28} We have the identity: 
\[
\chi(\alpha,\beta,t)=\begin{cases}
\delta(t-\beta)-\frac{\alpha\beta}{\sqrt{t^{2}-\beta^{2}}}\,J_{1}\left(\alpha\sqrt{t^{2}-\beta^{2}}\right), & t\geq\beta,\\
0, & 0<t<\beta,
\end{cases}
\]
where $\delta$ is the dirac delta function and $J_{1}$ is the Bessel
function of first order given by 
\[
J_{1}(z)=\frac{1}{\pi}\int_{0}^{\pi}\cos\left(z\sin\varphi-\varphi\right)d\varphi.
\]
 \end{lemma}
\begin{proof}
Using the change of variable $r=\sqrt{s^{2}+\alpha^{2}}$ we write
\[
e^{-\beta r}=e^{-\beta s}-(e^{-\beta s}-e^{-\beta r}).
\]
From the table \cite[p.~245]{Schiff} we get 
\begin{equation}
\mathcal{L}^{-1}\left\{ e^{-\beta s}\right\} =\delta(t-\beta),\label{eq:j1}
\end{equation}
Also on page 263 of \cite{Schiff} we find
\begin{equation}
\mathcal{L}^{-1}\left\{ e^{-\beta s}-e^{-\beta r}\right\} =\begin{cases}
\frac{\alpha\beta}{\sqrt{t^{2}-\beta^{2}}}\,J_{1}\left(\alpha\sqrt{t^{2}-\beta^{2}}\right), & t>\beta,\\
0, & 0<t<\beta.
\end{cases}\label{eq:j2}
\end{equation}
Subtracting \eqref{eq:j2} from \eqref{eq:j1} we obtain the expected
inverse Laplace transform.\hfill\end{proof}

\noindent Now we are ready to prove the convergence result for DNWR in 2D:
\begin{theorem}[Convergence of DNWR in 2D]\label{DNwaveTheorem2D} 
Let $\theta=1/2$.
For $T>0$ fixed, the DNWR algorithm \eqref{DN2dwave1}-\eqref{DN2dwave2}-\eqref{DN2dwave3} converges in
$k+1$ iterations, if the time window length $T$ satisfies $T/k<h_{\min}/c$,
$c$ being the wave speed.
\end{theorem}
\begin{proof}
We take Laplace transforms in $t$ of \eqref{DNwave2dlap} to get
\[
(s^{2}+c^{2}n^{2})\hat{U}_{m+1}^{k}-c^{2}\frac{d^{2}\hat{U}_{m+1}^{k}}{dx^{2}}=0.
\]
We obtain similar sequence of equations by applying Fourier series first and then Laplace transform to \eqref{DN2dwave2}. We now treat each $n$ as in the one-dimensional analysis in the
proof of Theorem \ref{DNwavemulti}, where the recurrence relation
\eqref{DNmfinalind}
of the form 
\begin{equation}\label{Update2dDN}
\hat{g}_{i}^{k}(s)=\sum_{j}\hat{r}_{i,j}^{k}(s)\,\hat{g}_{j}^{0}(s)
\end{equation}
become for each $n=1,2,\ldots$
\begin{equation}
\hat{G}_{i}^{k}(n,s)=\sum_{j}\hat{r}_{i,j}^{k}\left(\sqrt{s^{2}+c^{2}n^{2}}\right)\hat{G}_{j}^{0}(n,s).\label{DNwlapup}
\end{equation}
In the equation \eqref{Update2dDN} $\hat{r}_{i,j}^{k}(s)$ are linear combination of terms of the
form $e^{-\varrho s}$ for $\varrho\geq kh_{l}/c$ for some $l\in\left\{ 1,2,\ldots 2m+1\right\} $.
Therefore the coefficients $\hat{r}_{i,j}^{k}\left(\sqrt{s^{2}+c^{2}n^{2}}\right)$
are sum of exponential functions of the form $e^{-\varrho\sqrt{s^{2}+c^{2}n^{2}}}$
for $\varrho\geq kh_{l}/c$. Hence we use the definition of $\chi$
in \eqref{Xiab} to take the inverse Laplace transform of \eqref{DNwlapup},
and obtain
\[
G_{i}^{k}(n,t)=\sum_{j}\sum_{l}\chi(cn,\varrho_{l,j},t)*G_{j}^{0}(n,t),
\]
with $\varrho_{l,j}\geq kh_{\min}/c$. So it is straightforward that
for $t<kh_{\min}/c$, $G_{i}^{k}(n,t)=0$ for each $n$, since the
function $\chi$ is zero there by Lemma \ref{Lemma28}. Therefore
the update functions $g_{i}^{k}(y,t)$, given by $g_{i}^{k}(y,t)=\sum_{n\geq1}G_{i}^{k}(n,t)\sin(ny)$
are also zero for all $i\in\left\{ 1,\ldots,2m\right\}$. Hence
one more iteration produces the desired solution on the entire domain.
\hfill\end{proof}

\section{Numerical Experiments}\label{Section7}
We show some experiments for the DNWR algorithm in the spatial domain
$\Omega=(0,5)$, for the problem $\partial_{t}u=\Delta u,$ with initial
condition $u_{0}(x)=x(5-x)$ and boundary conditions $u(0,t)=t^{2},u(5,t)=te^{-t}$.
We discretize the heat equation using standard centered
finite differences in space and backward Euler in time on a grid with
$\Delta x=2\times10^{-2}$ and $\Delta t=4\times10^{-3}$. In the first experiment we apply the DNWR for a decomposition into
five subdomains and for three different time windows $T=0.2, T=2$
and $T=8$, whereas for a fixed time $T=2$ we run another experiment
for three to six equal subdomains. In Figure \ref{FigDNmA3expr}, on the left panel, we show the convergence estimates in the five-subdomain case as a function of time $T$, whereas on the right panel, we show the convergence for $T=2$ as we vary the number of subdomains. We observe superlinear convergence as predicted by Theorem \ref{TheorDNmheat}, and for small $T$ the estimate is quite sharp. We also see that the convergence slows down as the number of subdomains is increased.
\begin{figure}
\centering
\includegraphics[width=0.49\textwidth]{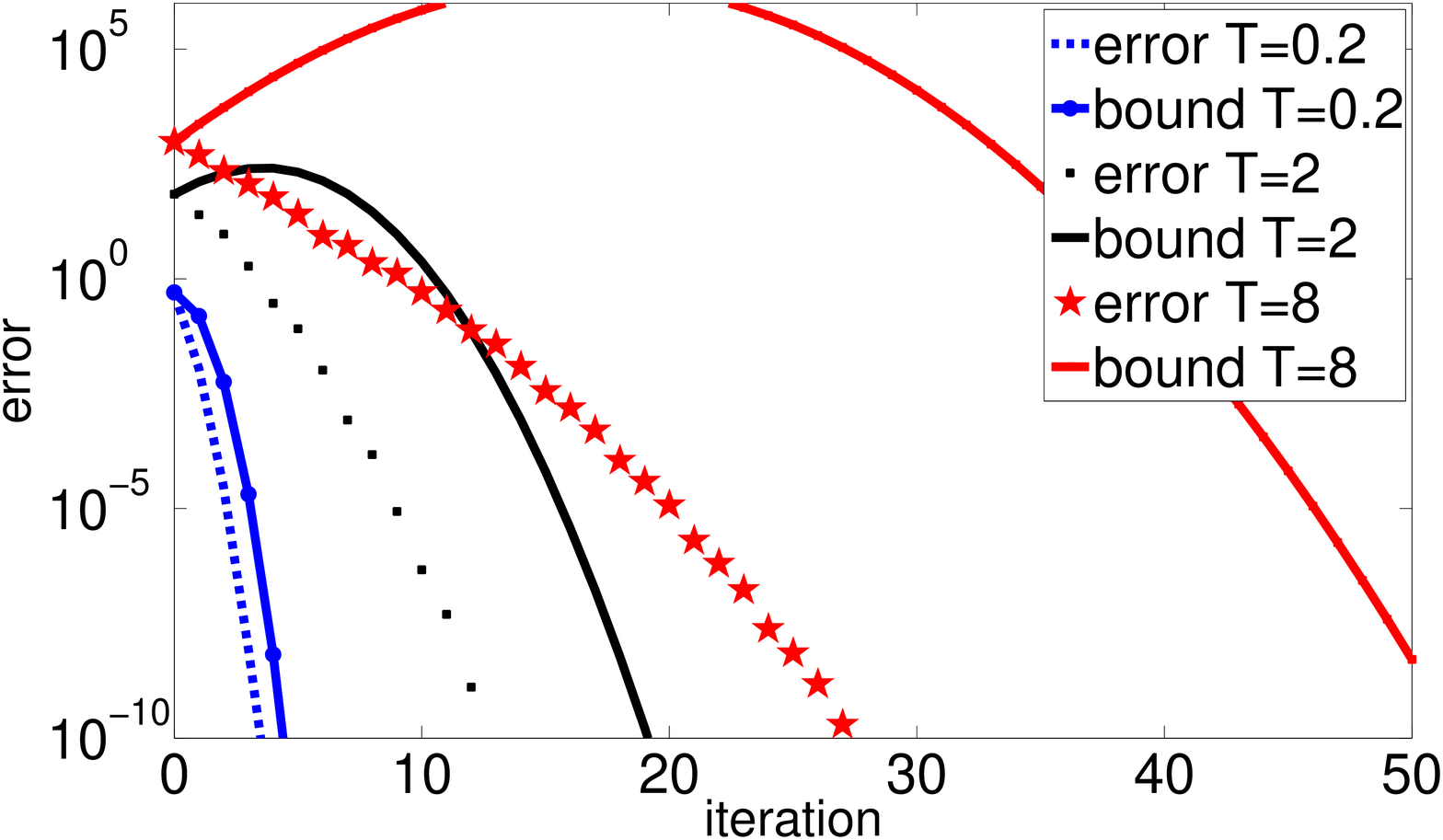}\includegraphics[width=0.49\textwidth]{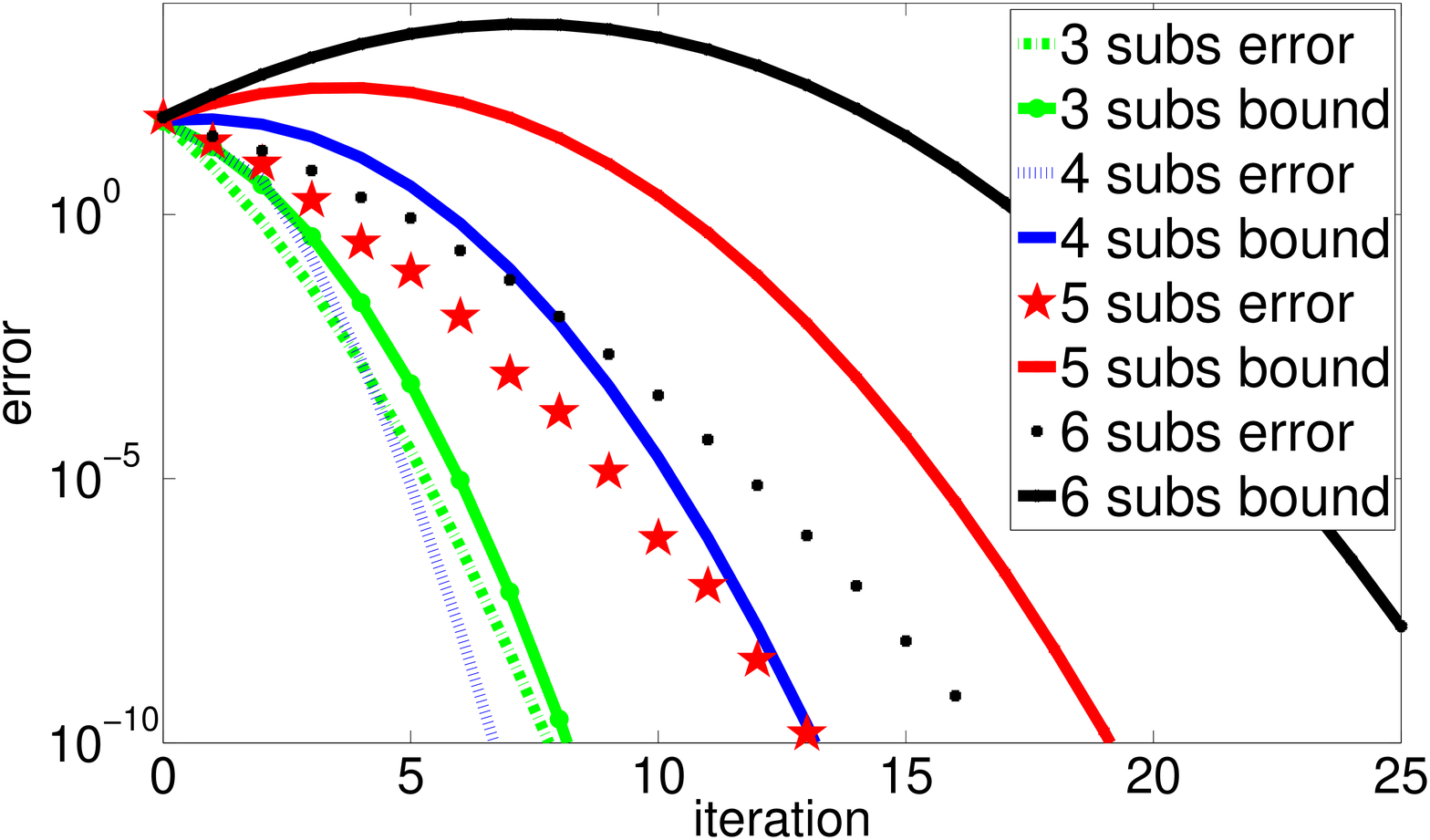}
\caption{Convergence estimates of DNWR for $\theta=1/2$, on the left for various values of $T$ for five subdomains, and on the right for various number of subdomains for $T=2$}
\label{FigDNmA3expr}
\end{figure}
\begin{figure}
  \centering
  \includegraphics[width=0.49\textwidth]{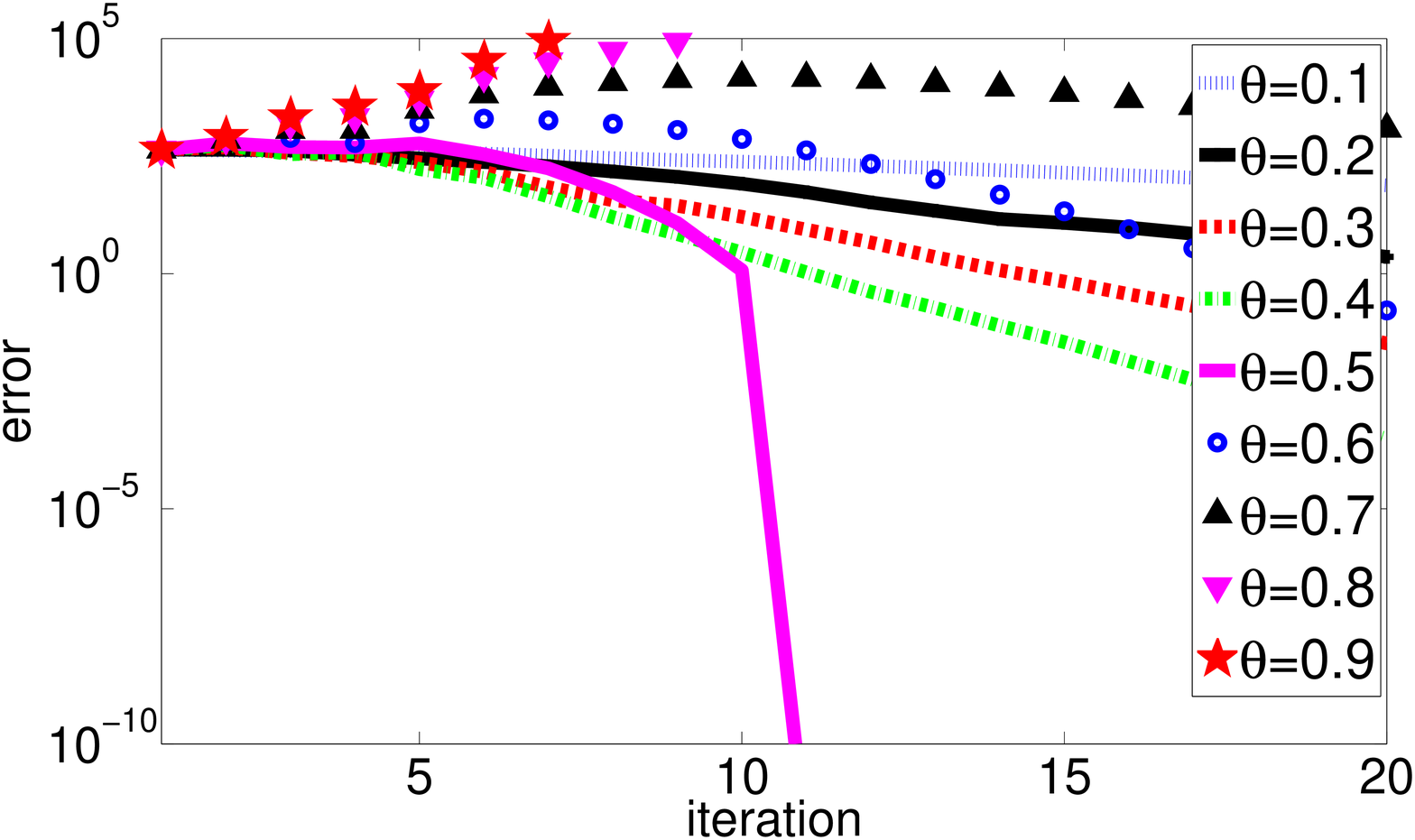}
  \includegraphics[width=0.49\textwidth]{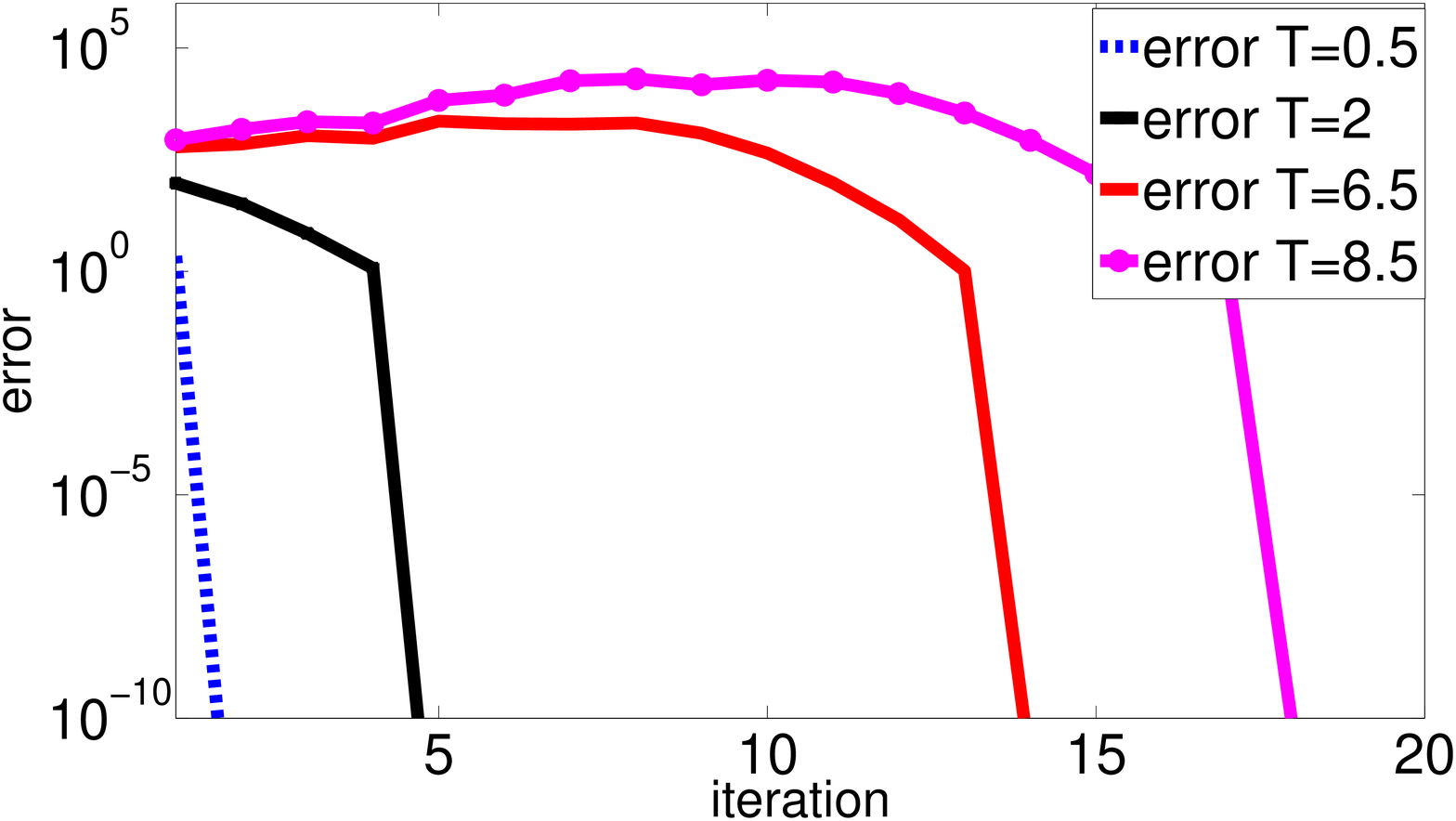}
  \caption{Arrangement A3: convergence of DNWR with various values of $\theta$ for $T=5$ on the left, and for various lengths $T$ of the time window and $\theta=1/2$
on the right}
  \label{NumFig4}
\end{figure} 

We now consider the following model wave equation to see the convergence behavior of the DNWR algorithm with multiple subdomains:
\begin{align}
\partial_{tt}u & =\partial_{xx}u, &  & x\in(0,5),t>0,\nonumber \\
u(x,0) & =0,\:u_{t}(x,0)=0, &  & 0<x<5,\label{Wavenumericmodel}\\
u(0,t) & =t^{2},\:u(5,t)=t^{2}e^{-t}, &  & t>0,\nonumber 
\end{align}
which is discretized using centered finite differences in both space
and time on a grid with $\Delta x=\Delta t=2{\times}10^{-2}$. We
take the initial guesses $g_{j}^{0}(t)=t^{2},t\in(0,T]$ for $1\leq j\leq4$,
and consider a decomposition of $(0,5)$ into five unequal subdomains,
whose widths $h_{i}$ are $1,0.5,1.5,1,1$ respectively, so that $h_{\min}=0.5$.
On the left panel of Figure \ref{NumFig4}, we show the convergence for different values of the parameter $\theta$ for $T=5$, and on the right panel the error curves for the best parameter $\theta=1/2$ for different
time window length $T$. These convergence curves justify our convergence
result for the arrangement A3 in Theorem \ref{DNwavemulti}. We observe two-step convergence for $\theta=1/2$ for a sufficiently small time window $T$. Coincidentally we observe exactly the same convergence behavior for other two corresponding arrangements A1 and A2 from Subsection \ref{Section2a}, see Figure \ref{NumFig5} and Figure \ref{NumFig6} respectively.
\begin{figure}
  \centering
  \includegraphics[width=0.49\textwidth]{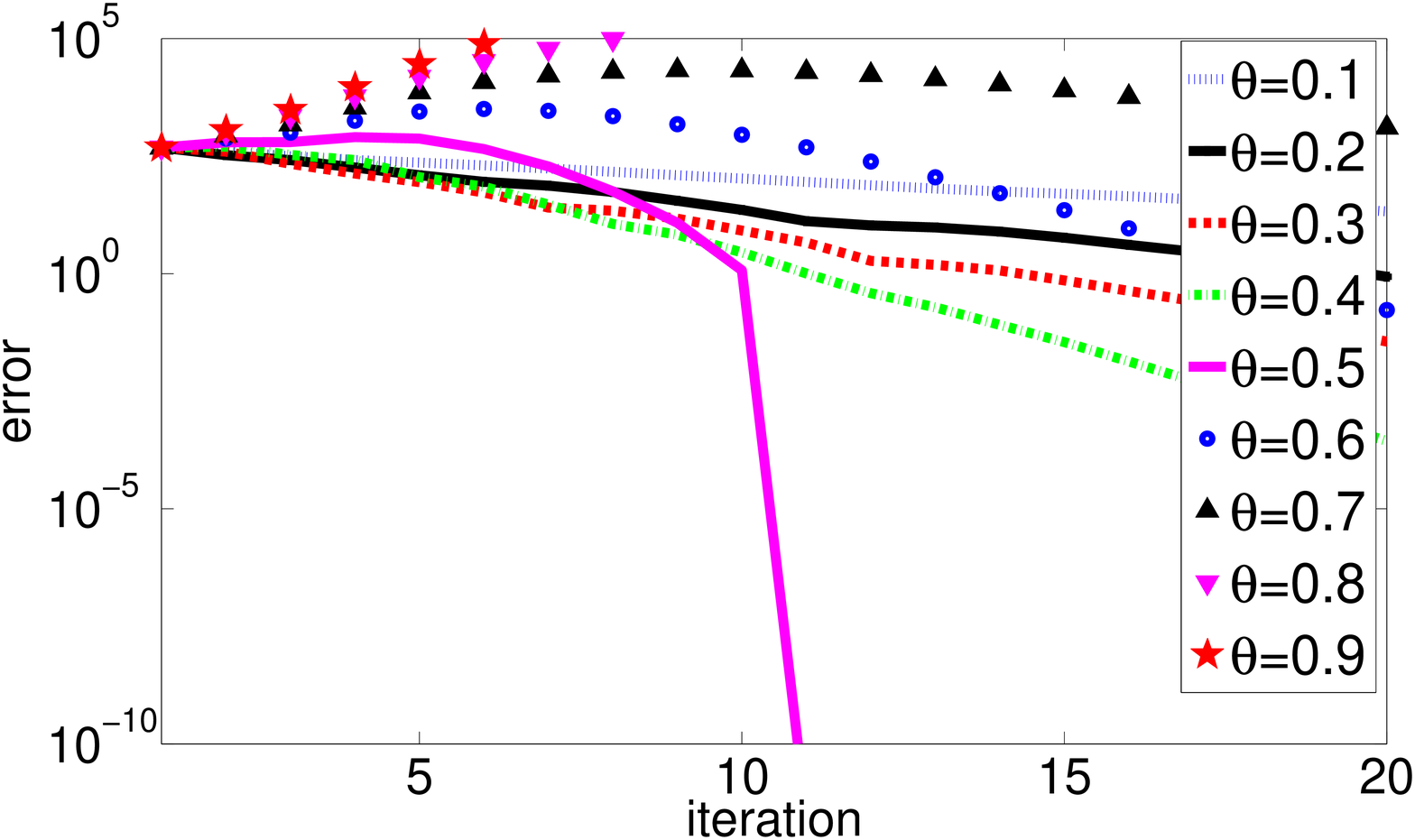}
  \includegraphics[width=0.49\textwidth]{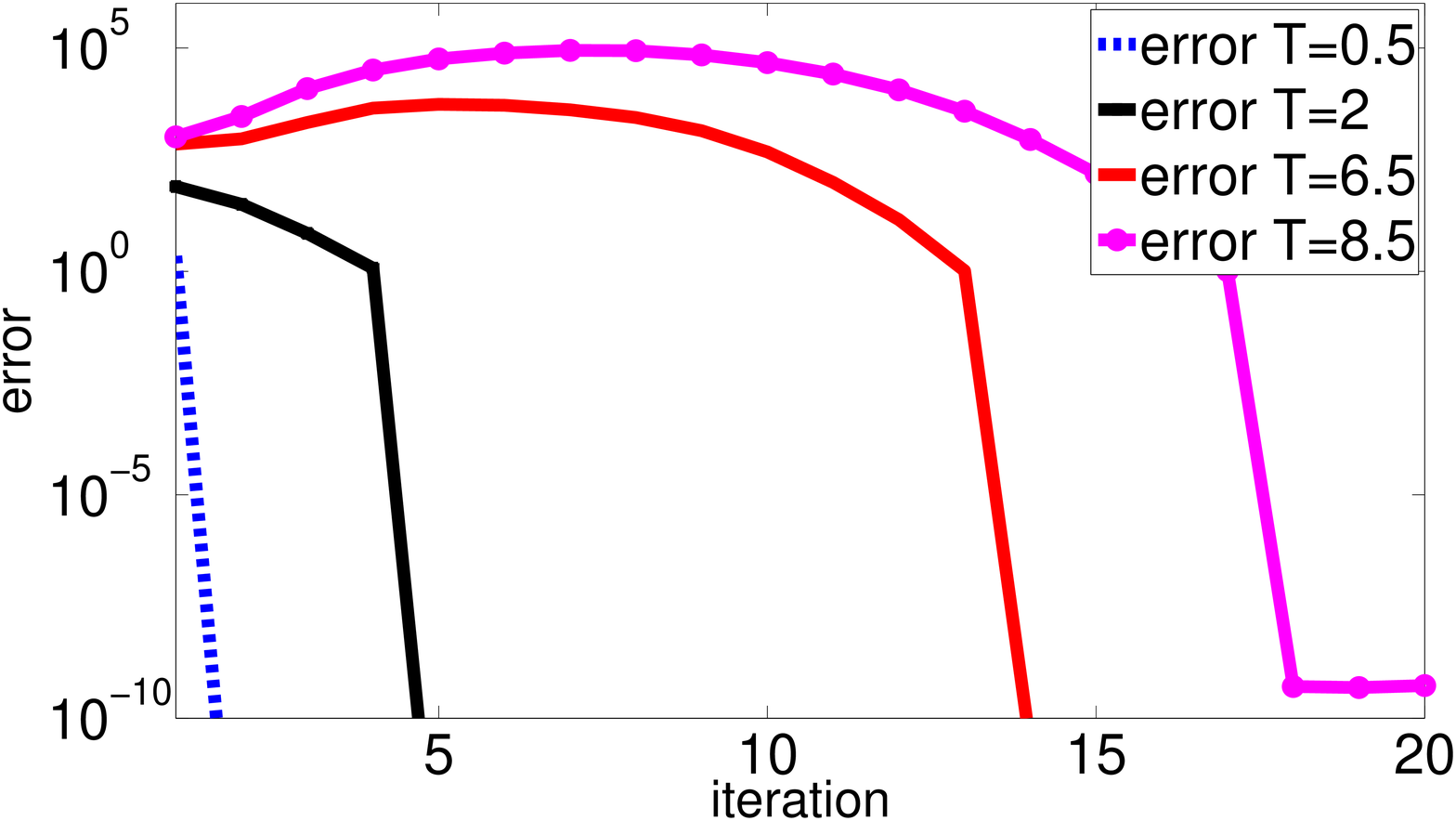}
  \caption{Arrangement A1: convergence of DNWR with various values of $\theta$ for $T=5$ on the left, and for various lengths $T$ of the time window and $\theta=1/2$
on the right}
  \label{NumFig5}
\end{figure} 
\begin{figure}
  \centering
  \includegraphics[width=0.49\textwidth]{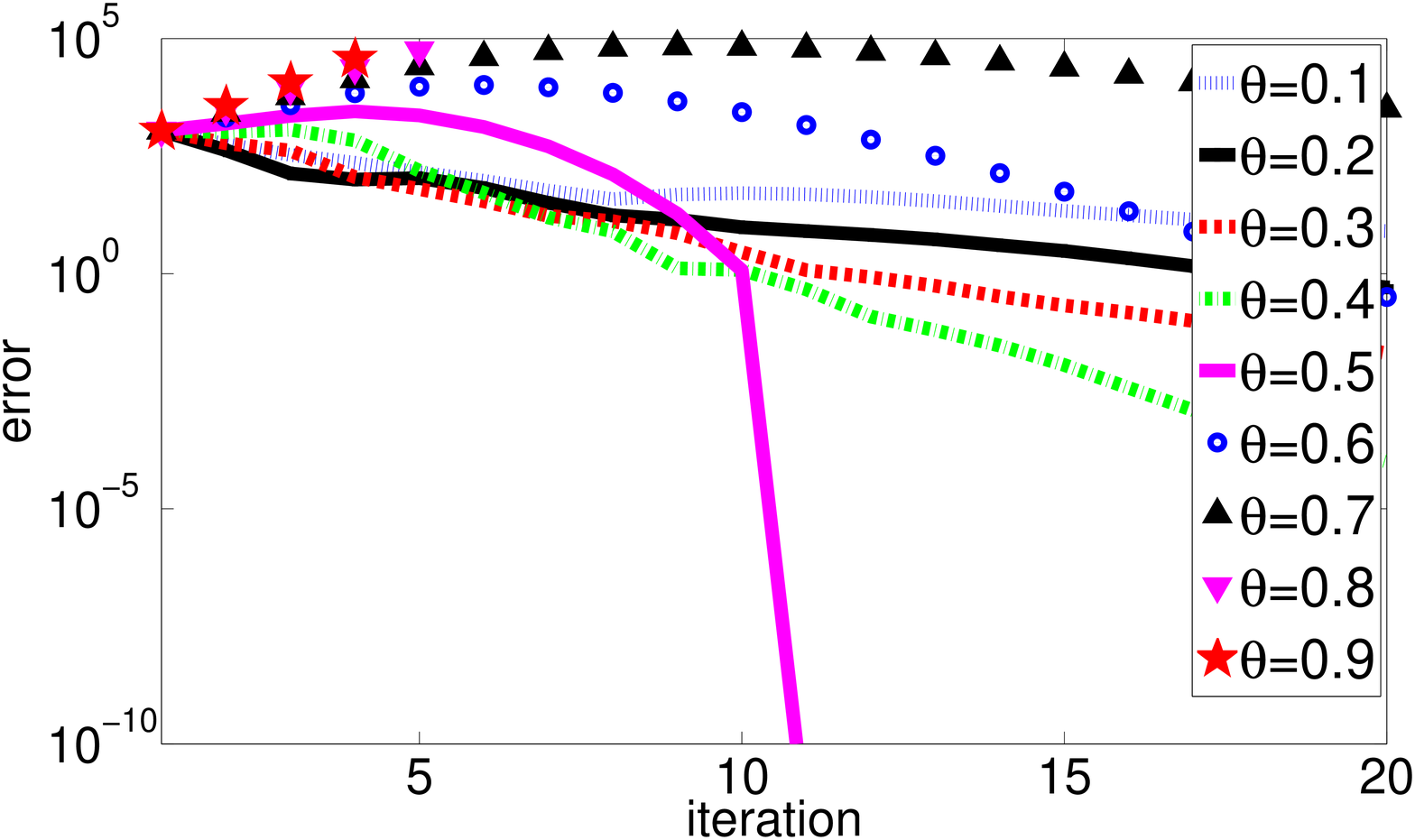}
  \includegraphics[width=0.49\textwidth]{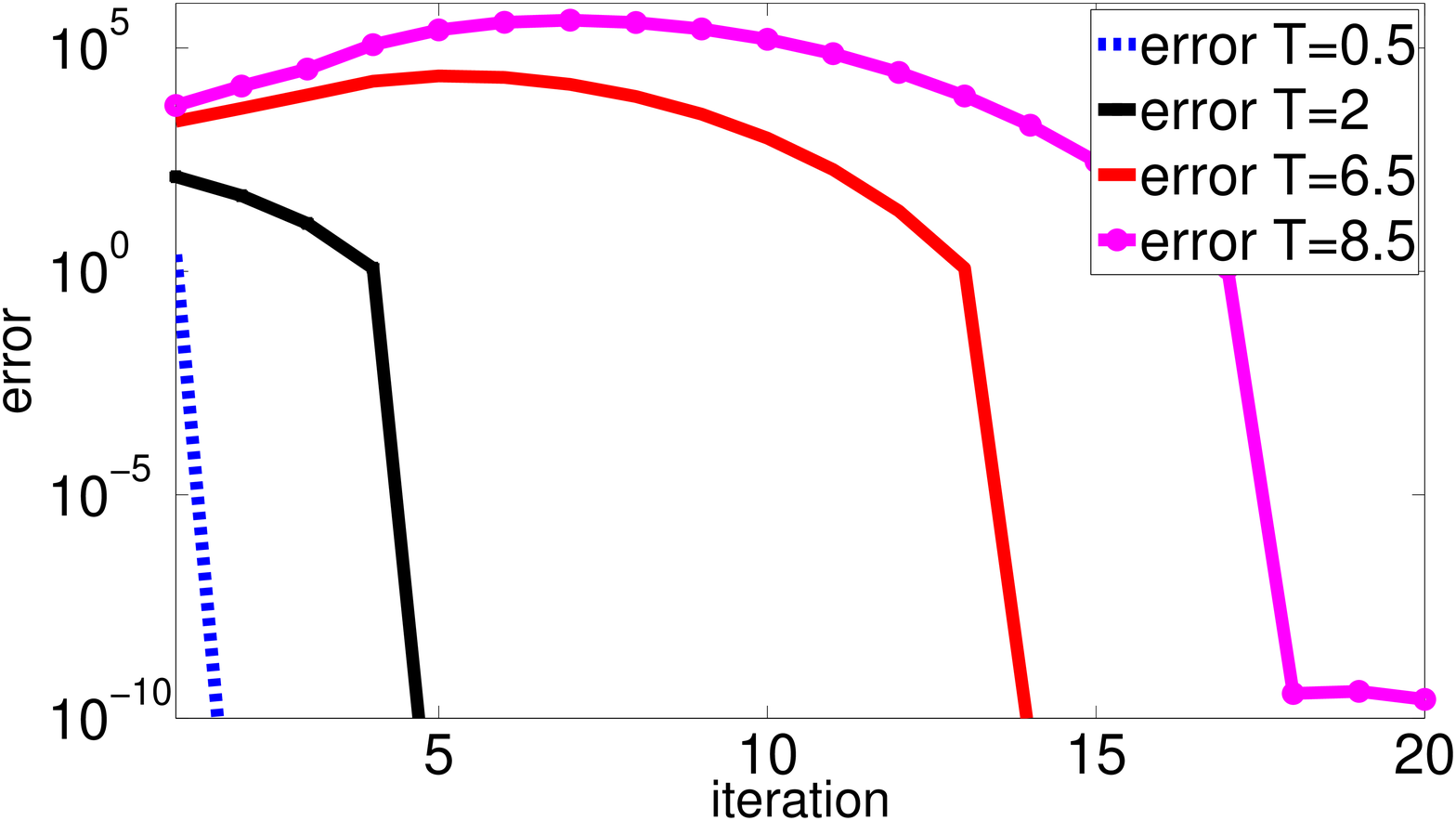}
  \caption{Arrangement A2: convergence of DNWR with various values of $\theta$ for $T=5$ on the left, and for various lengths $T$ of the time window and $\theta=1/2$
on the right}
  \label{NumFig6}
\end{figure} 

Next we show an experiment for the DNWR algorithm in two dimension
for the following model problem 
\[
\partial_{tt}u-\left(\partial_{xx}u+\partial_{yy}u\right)=0,u(x,y,0)=xy(x-1)(y-\pi)(5x-2)(4x-3),u_{t}(x,y,0)=0,
\]
with homogeneous Dirichlet boundary conditions. We discretize the
wave equation using the centered finite difference in both space and
time (Leapfrog scheme) on a grid with $\Delta x=5{\times}10^{-2},\Delta y=16{\times}10^{-2},\Delta t=4{\times}10^{-2}$.
We decompose our domain $\Omega:=(0,1)\times(0,\pi)$ into three non-overlapping subdomains $\Omega_{1}=(0,2/5)\times(0,\pi)$, $\Omega_{2}=(2/5,3/4)\times(0,\pi)$, $\Omega_{3}=(3/4,1)\times(0,\pi)$. As initial guesses, we take $w_{i}^{0}(y,t)=t\sin(y)$.
In Figure \ref{NumFig8} we plot the convergence curves for different values of the parameter $\theta$ for $T=2$ on the left panel, and on the
right the results for the best parameter $\theta=1/2$ for different
time window length $T$.
\begin{figure}
  \centering
  \includegraphics[width=0.49\textwidth]{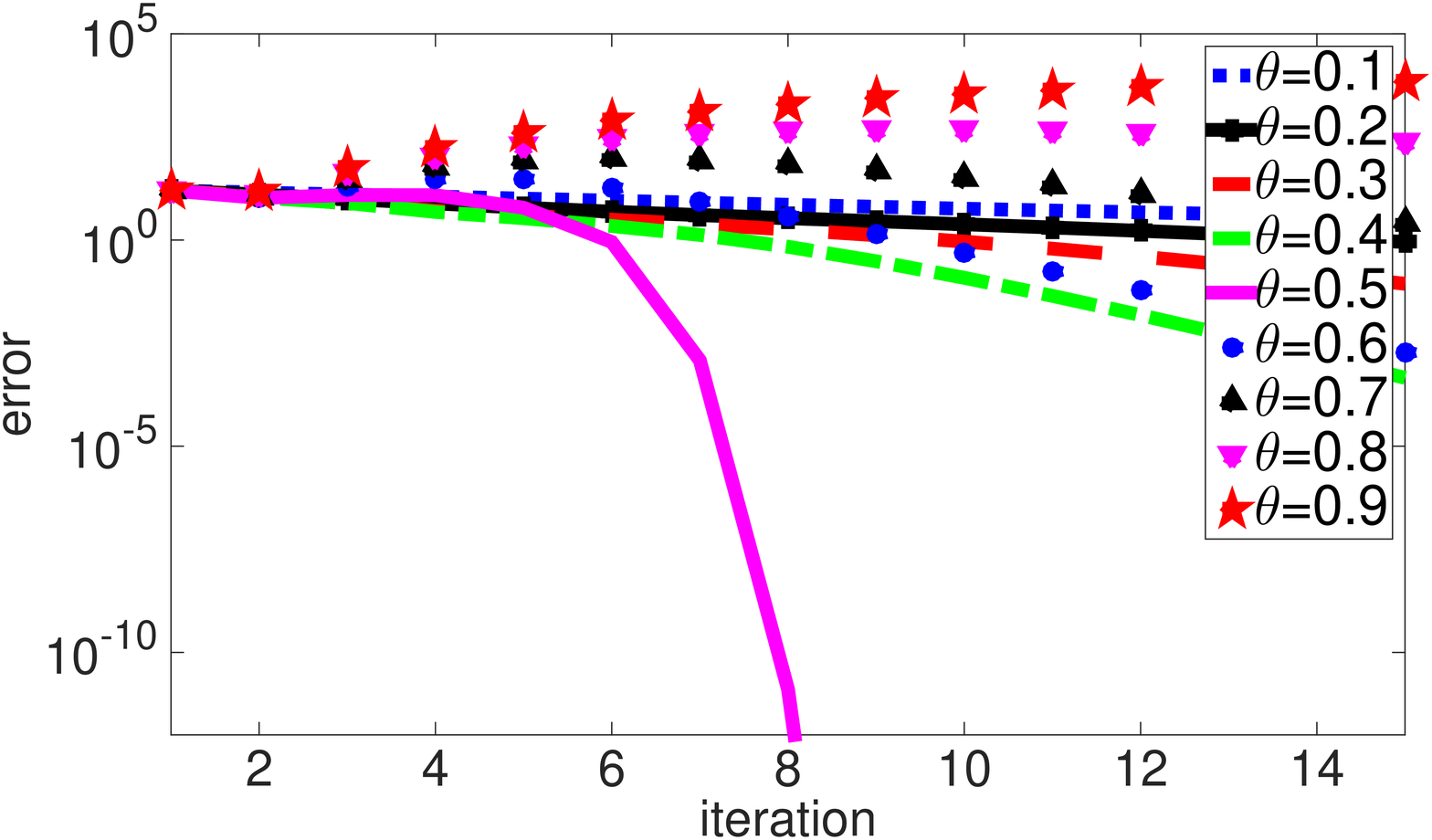}
  \includegraphics[width=0.49\textwidth]{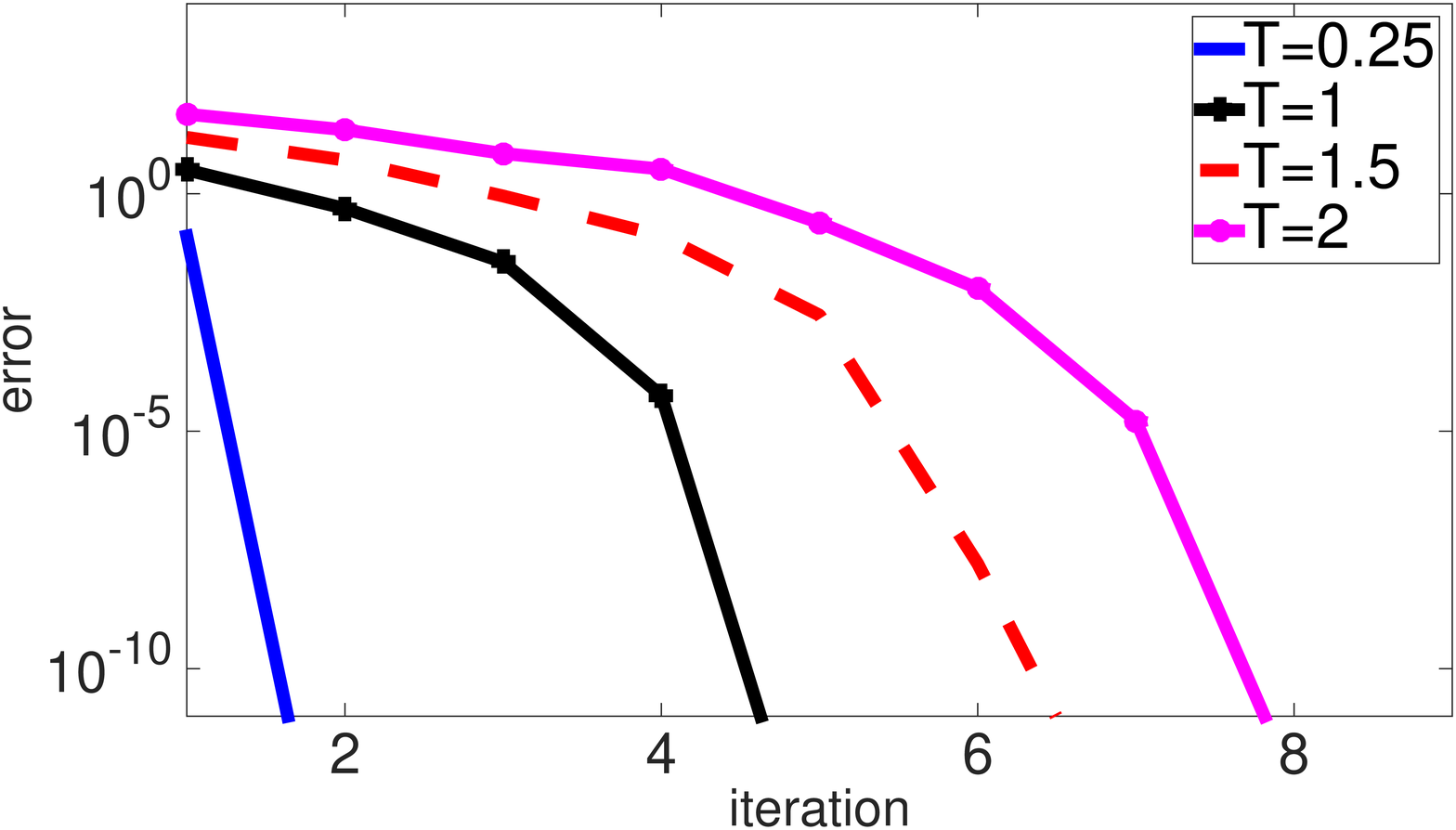}
  \caption{Convergence of DNWR in 2D: curves for different values of $\theta$ for $T=2$ on the left, and for various time lengths $T$ and $\theta=1/2$ on the right}
  \label{NumFig8}
\end{figure}   

We now compare in Figure \ref{NumFig10} the performance of the DNWR algorithm with its counterpart NNWR method \cite{Mandal2,GKM3} and the SWR algorithms with and without overlap \cite{GHN,GH2}. Here we consider the model problem
\[
\partial_{tt}u-\left(\partial_{xx}u+\partial_{yy}u\right)=0,u(x,y,0)=0=u_{t}(x,y,0),
\]
with Dirichlet boundary conditions $u(0,y,t)=t^{2}\sin(y),u(1,y,t)=y(y-\pi)t^{3}$ and $u(x,0,t)=0=u(x,\pi,t)$. We decompose our domain $\Omega:=(0,1)\times(0,\pi)$ for the two subdomains experiment into $\Omega_{1}=(0,3/5)\times(0,\pi)$ and $\Omega_{2}=(3/5,1)\times(0,\pi)$, and for the three subdomains experiment into $\Omega_{1}=(0,2/5)\times(0,\pi)$, $\Omega_{2}=(2/5,3/4)\times(0,\pi)$, $\Omega_{3}=(3/4,1)\times(0,\pi)$. We take a random initial guess to start the iteration, and for the overlapping SWR we use an overlap of length $2\Delta x$ in all the experiments. We implement first order methods with one parameter in optimized SWR iterations; for more details see \cite{GH2}. On the left panel of Figure \ref{NumFig10} we plot the comparison curves for two subdomains, and the same for three subdomains on the right. 
\begin{figure}
  \centering
  \includegraphics[width=0.49\textwidth]{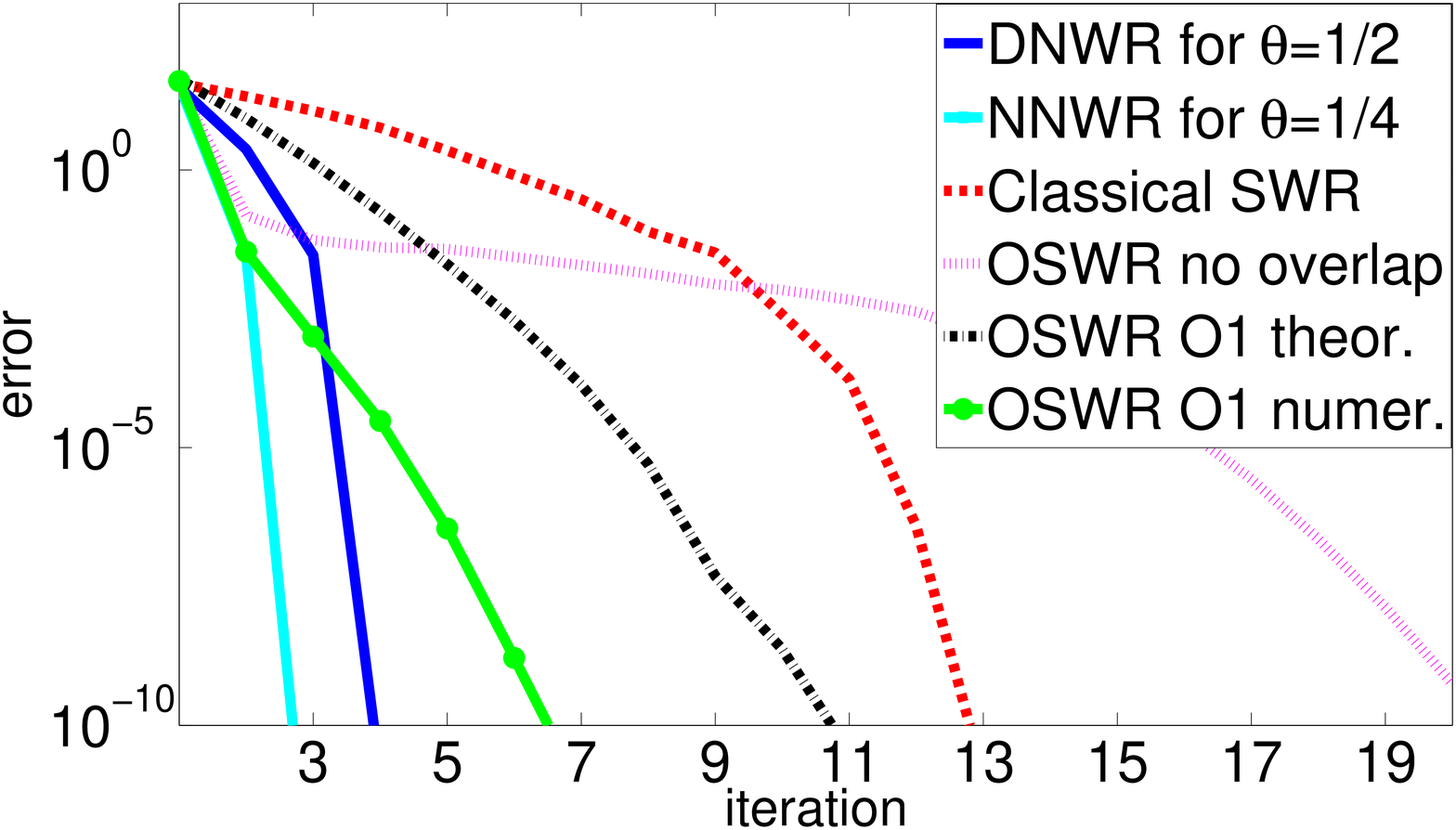}
  \includegraphics[width=0.49\textwidth]{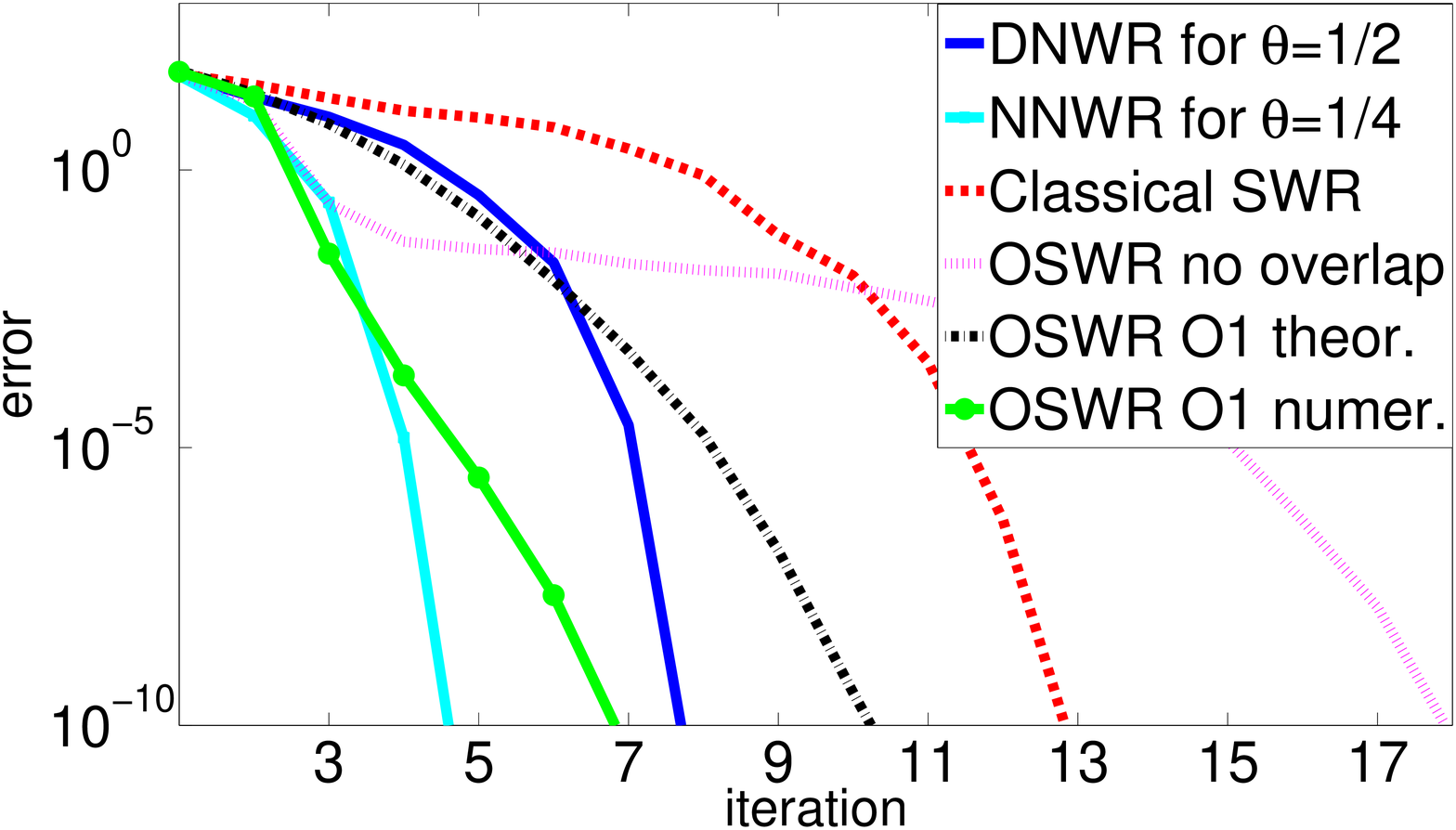}
  \caption{Comparison of DNWR, NNWR, and SWR for $T=2$ in 2D for two subdomains on the left, and three subdomains on the right}
  \label{NumFig10}
\end{figure}

Here we consider a comparison of performances between the DNWR and the NNWR algorithms for the wave equation. Table \ref{TableSWR2} gives a summary of the theoretical results from Section \ref{Section3} and \ref{Section6} and \cite{Mandal2,GKM3}, to indicate the maximum number of   
iterations needed for the 1D and 2D wave equation to converge to the exact solution.
\begin {table}\begin{center} \caption {Comparison of steps needed for convergence of DNWR and NNWR for the wave equation. \label{TableSWR2}} \begin{tabular}{|c|c|c|c|} \hline Methods & 2 subdomains, 1D & Many subdomains, 1D & Many subdomains, 2D \tabularnewline \hline DNWR & $T\leq 2kh_{\min}/c$ & $T\leq kh_{\min}/c$ & $T< kh_{\min}/c$\tabularnewline \hline NNWR & $T\leq 4kh_{\min}/c$ & $T\leq 2kh_{\min}/c$ & $T< 2kh_{\min}/c$\tabularnewline \hline \end{tabular} \end{center} \end {table} 
\begin {table} \begin{center} \caption {Propagation speed and time steps for different subdomains.}\label{Table23} \begin{tabular}{|c|c|c|c|} \hline & $\Omega_1$ & $\Omega_2$ & $\Omega_3$ \tabularnewline \hline wave speed $c$& $1/4$ & $2$ & $1/2$ \tabularnewline \hline time grids $\Delta t_i$ & $13\times 10^{-2}$ & $39\times 10^{-3}$ & $1\times 10^{-1}$ \tabularnewline \hline \end{tabular} \end{center} \end {table} 

Next we show a numerical experiment for the DNWR algorithm with different
time grids for different subdomains and discontinuous wave speed across
interfaces. We consider the model problem 
\[
\partial_{tt}u-c^{2}\partial_{xx}u=0,u(x,0)=0=u_{t}(x,0),
\]
with Dirichlet boundary conditions $u(0,t)=t^{2},u(6,t)=t^{3}$. Suppose
the spatial domain $\Omega:=(0,6)$ is decomposed into three equal
subdomains $\Omega_{i},i=1,2,3$, and the random initial guesses are
used to start the DNWR iteration. For the spatial discretization,
we take a uniform mesh with size $\Delta x=1{\times}10^{-1}$, and
for the time discretization, we use non-uniform time grids $\Delta t_{i},i=1,2,3$, as given in Table \ref{Table23}. 
For the non-uniform mesh grid, boundary data is transmitted from one subdomain to a neighboring subdomain by introducing a suitable time projection. For two dimensional problems, the interface is one dimensional. Using ideas of merge sort one can compute the projection with linear cost, see \cite{GJap} and the references therein. In Figure \ref{FigDNnonUnif}(a) we show the non-uniform time steps for different subdomains. Figure \ref{FigDNnonUnif}(b), (c) and (d) give the three-step convergence of the DNWR algorithm for $T=2$.
\begin{figure}
\centering
\begin{tabular}{cc}
\subfloat[Non-uniform time-stepping]{\includegraphics[width=0.49\textwidth]{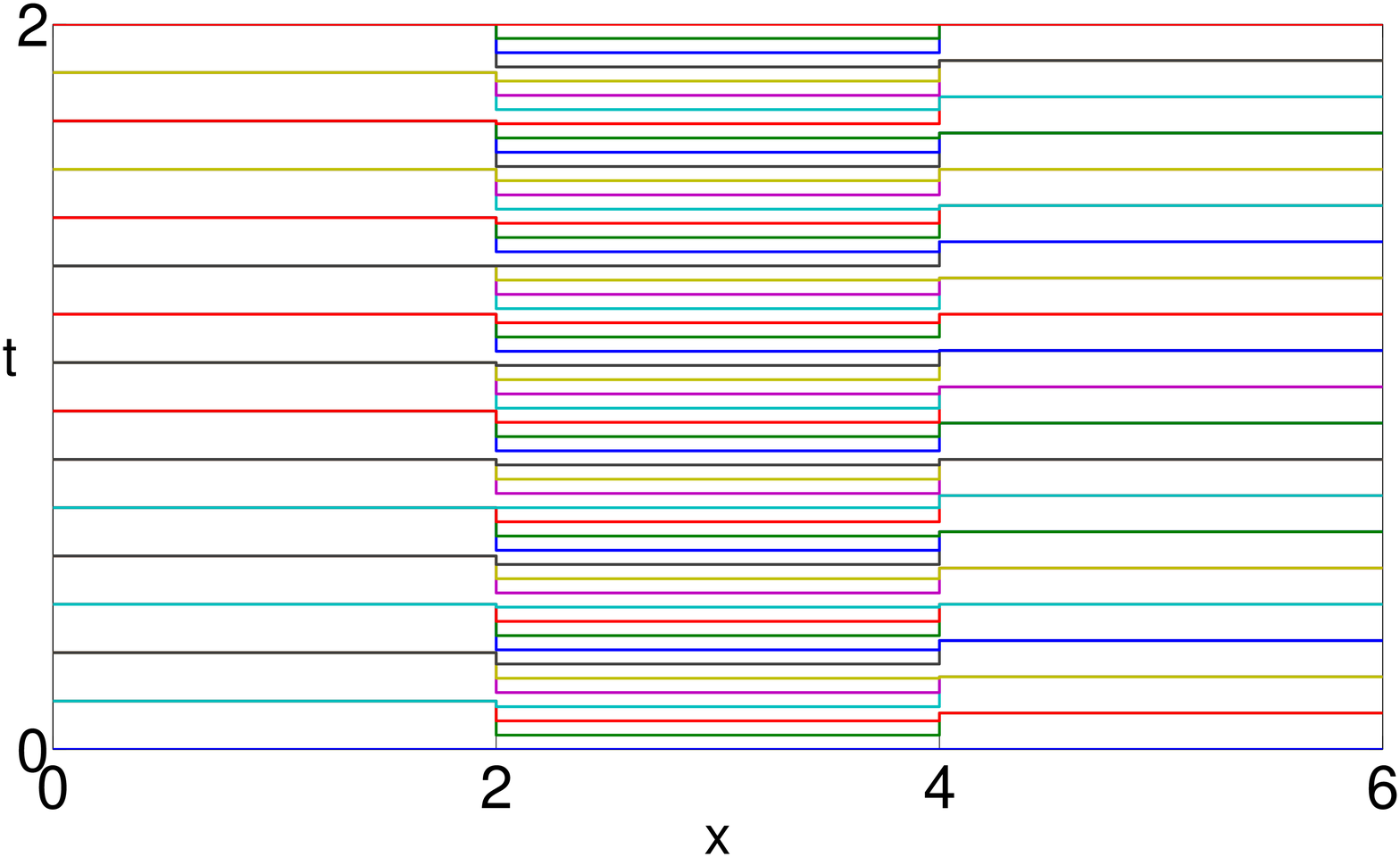}} & \subfloat[DNWR: 1st iteration]{\includegraphics[width=0.49\textwidth]{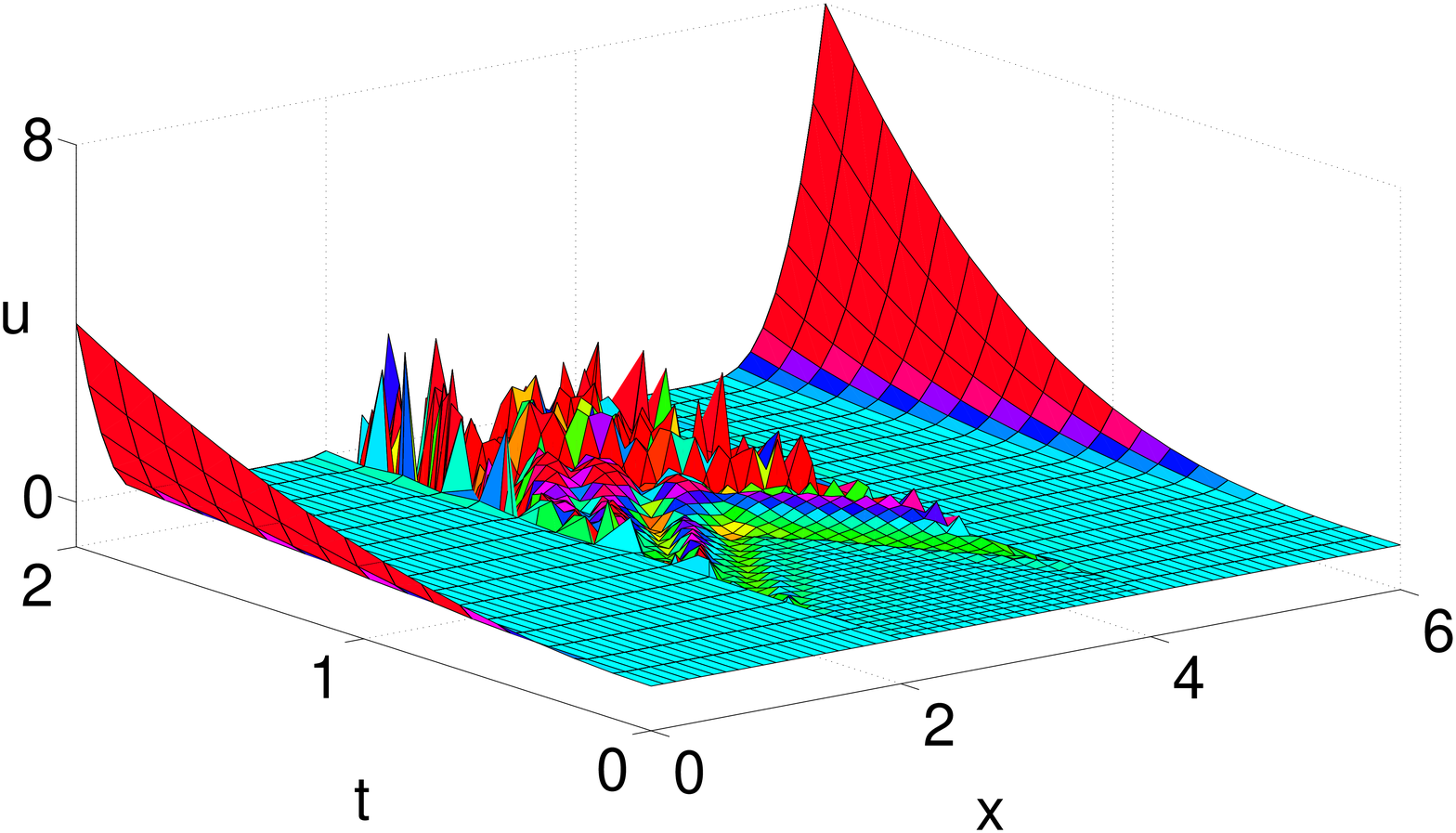}}  \\
\subfloat[DNWR: 2nd iteration]{\includegraphics[width=0.49\textwidth]{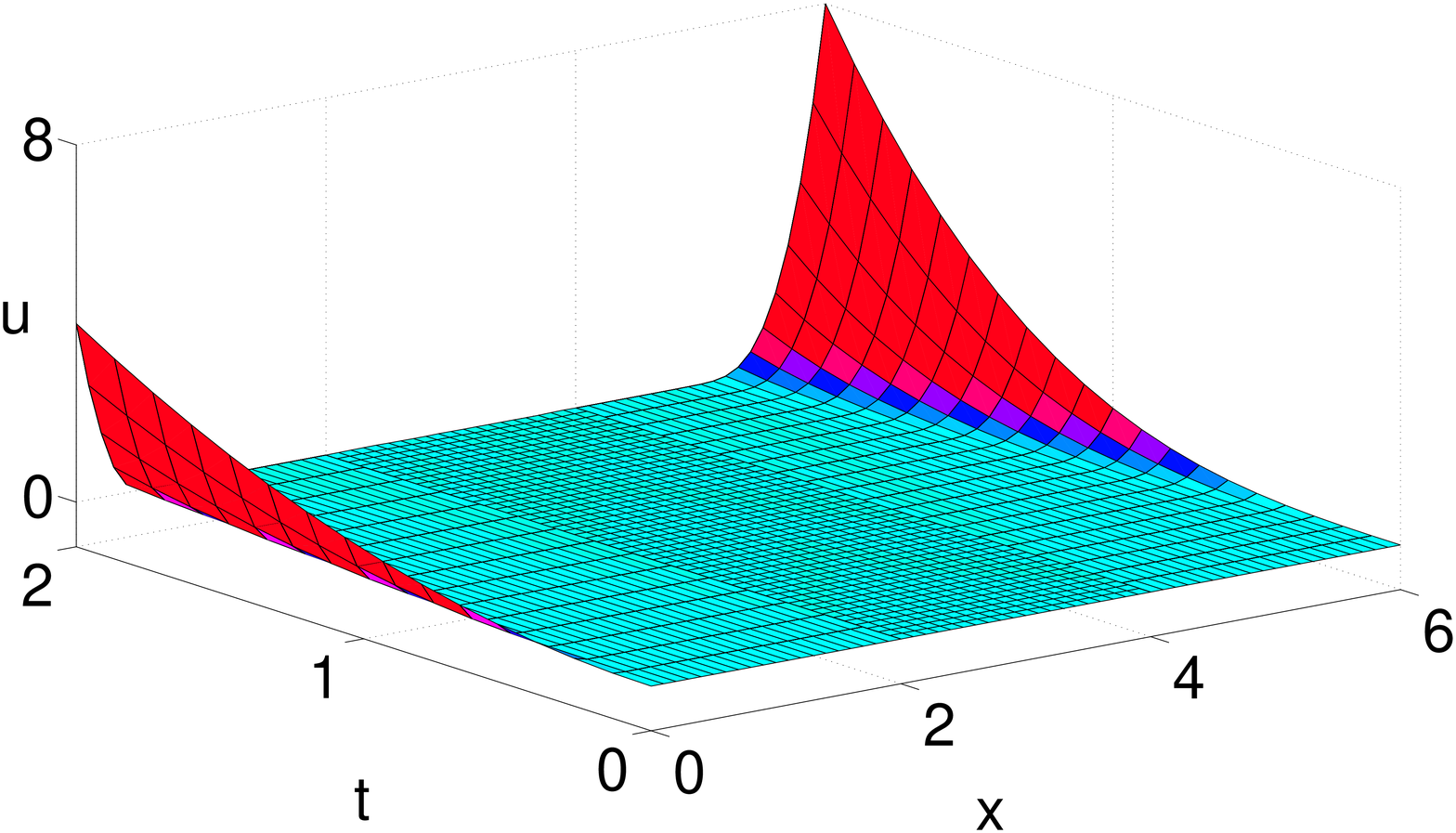}} & \subfloat[DNWR: 3rd iteration]{\includegraphics[width=0.49\textwidth]{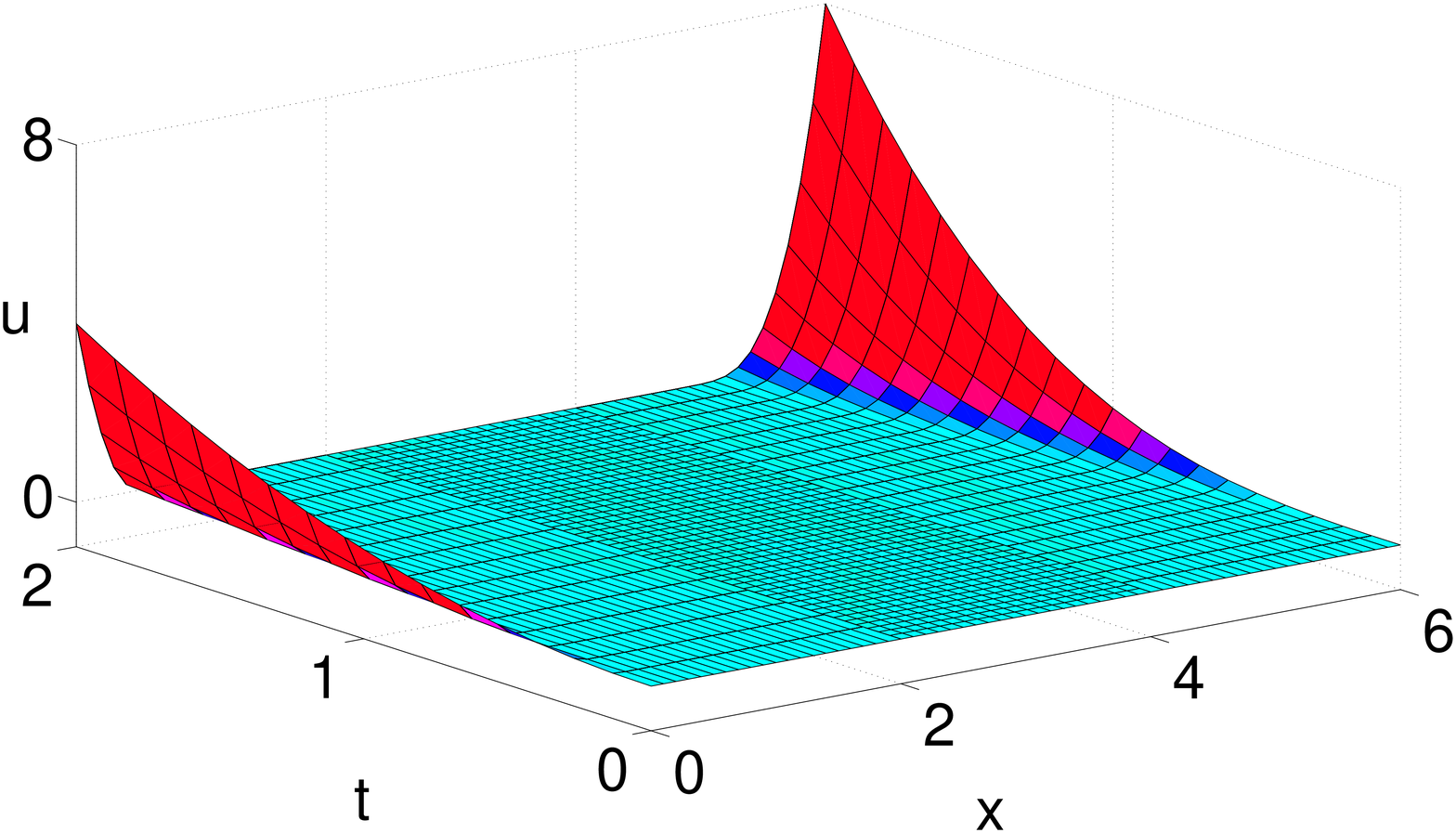}}
\end{tabular}
\caption{Convergence of the DNWR method applied to the wave equation for non-uniform time steps for $\theta=1/2$ for $T=2$}
\label{FigDNnonUnif}
\end{figure}

\section{Conclusions }

We defined the DNWR algorithm for multiple subdomains for parabolic and hyperbolic problems, and analyzed its convergence
properties for one dimensional heat and wave equations. We proved
using numerical experiments that for a particular choice of the relaxation
parameter, $\theta=1/2$, superlinear convergence can be obtained for heat equation, whereas we showed finite step convergence for wave equation.
In fact, the algorithm can be used as a two-step method for the wave equation, choosing the time window lengh $T$ small enough. We have also extended the DNWR algorithm for 2D wave equation, and analyzed
its convergence properties. We have also shown using numerical experiments
that among DNWR and NNWR, the second converges faster. But in comparison
to DNWR, the NNWR has to solve twice the number of subproblems (once
for Dirichlet subproblems, and once for Neumann subproblems) on each
subdomain at each iteration. Therefore the computational cost is almost
double for the NNWR than for the DNWR algorithm at each step. However,
we get better convergence behavior with the NNWR in terms of iteration
numbers. Finally we presented a comparison of performences between
the DNWR, NNWR and Schwarz WR methods, and showed that the DNWR and
NNWR converge faster than optimized SWR at least for higher dimensions. 
\afterpage{\clearpage}

\bibliographystyle{siam}
\bibliography{paperdnm}


\end{document}